\documentclass[reqno]{amsart}
\usepackage{mathrsfs}
\usepackage{hyperref}
\usepackage{float}
\usepackage{graphicx}
\usepackage{enumerate}
\usepackage[framemethod=tikz]{mdframed}
\newtheorem{theorem}{Theorem}[section]
\newtheorem{lemma}[theorem]{Lemma}
\newtheorem{proposition}[theorem]{Proposition}

\newtheorem{corollary}[theorem]{Corollary}
\theoremstyle{definition}
\newtheorem{definition}[theorem]{Definition}
\newtheorem{example}[theorem]{Example}
\newtheorem{remark}[theorem]{Remark}

\newcommand{\norm}[1]{\left\Vert#1\right\Vert}

\newcommand{\abs}[1]{\left\vert#1\right\vert}

\newcommand{\R}{\mathbb{R}}

\newcommand{\Sclass}{\mathscr{S}}
\newcommand{\Lp}[2][]{L^{#2}_{#1}}

\numberwithin{equation}{section}
\usepackage{amsfonts}

\usepackage{amsmath}
\usepackage{amssymb}
\usepackage{cite}
\usepackage{hyperref}
\hypersetup{
   colorlinks,
    citecolor=blue,
    filecolor=black,
    linkcolor=blue,
    urlcolor=magenta
}

\begin{document}
\font\nho=cmr10
\def\dive{\mathrm{div}}
\def\cal{\mathcal}
\def\L{\cal L}

\def \ud{\underline }
\def\id{{\indent }}
\def\f{\frac}
\def\non{{\noindent}}
 \def\le{\leqslant}
 \def\leq{\leqslant}
 \def\geq{\geqslant}
\def\rar{\rightarrow}
\def\Rar{\Rightarrow}
\def\ti{\times}
\def\i{\mathbb I}
\def\j{\mathbb J}
\def\si{\sigma}
\def\Ga{\Gamma}
\def\ga{\gamma}
\def\ld{{\lambda}}
\def\Si{\Psi}
\def\f{\mathbf F}
\def\r{\hro{R}}
\def\e{\cal{E}}
\def\B{\cal B}
\def\A{\mathcal{A}}
\def\p{\mathbb P}

\def\tet{\theta}
\def\Tet{\Theta}
\def\hro{\mathbb}
\def\ho{\mathcal}
\def\P{\ho P}
\def\E{\mathcal{E}}
\def\n{\mathbb{N}}
\def\M{\mathbb{M}}
\def\dMu{\mathbf{U}}
\def\dMcs{\mathbf{C}}
\def\dMcu{\mathbf{C^u}}
\def\vk{\vskip 0.2cm}
\def\td{\Leftrightarrow}
\def\df{\frac}
\def\Wei{\mathrm{We}}
\def\Rey{\mathrm{Re}}
\def\s{\mathbb S}
\def\l{\mathcal{L}}
\def\C+{C_+([t_0,\infty))}
\def\o{\cal O}
\newcommand{\lec}{{\ \lesssim \ }}

\newcommand{\crr}[1]{{\color[rgb]{1,0,0} #1}} %
\newcommand{\cro}[1]{{\color[rgb]{1,0.3,0} (#1)}} %
\newcommand{\crm}[1]{{\color{magenta} #1}} %
\newcommand{\crs}[1]{{\color{rgb}{0.53, 0.18, 0.09} #1}} %
\newcommand{\crb}[1]{{\color[rgb]{0,0,0.7} #1}}
\newcommand{\crdb}[1]{{\color[rgb]{0,0,0.4} #1}}
\newcommand{\crp}[1]{{\color[rgb]{0.3,0,0.6} #1}}

\title[Self-similar Solutions of Navier-Stokes equations]{Forward self-similar solutions of the Navier-Stokes equations in $\R^n$}

\author[B.S. Lai]{Baishun Lai}
\address{LCSM (MOE) and School of Mathematics and Statistics, Hunan Normal University, Changsha 410081, Hunan, People's Republic of China}
\email{laibaishun@hunnu.edu.cn}

\author[C.X. Miao]{Changxing Miao}
\address{Institute of Applied Physics and Computational Mathematics, P.O. Box 8009, Beijing 100088, People's Republic of China}
\email{miao\_changxing@iapcm.ac.cn}

\author[X.X. Zheng]{Xiaoxin Zheng} \address{School of Mathematical Sciences, Beihang University, Beijing 100191, P.R. China and Key Laboratory of Mathematics, Informatics and Behavioral Semantics, Ministry of Education, Beijing 100191, People's Republic of China}
 \email{xiaoxinzheng@buaa.edu.cn}

\begin{abstract}
We investigate the existence, regularity, and spatial decay of forward self-similar solutions to the incompressible Navier--Stokes equations with homogeneous initial data of the form
$
u_0(x)=\frac{\sigma(x/|x|)}{|x|},
\,\, x\in\mathbb R^n,
$
where \(\sigma\) is defined on the unit sphere, with particular emphasis on the critical four-dimensional case.
For the existence theory, in dimensions \(3\leq n\leq5\), we develop a new Galerkin approximation scheme for the Leray profile problem. A distinctive feature of our approach is the decomposition of the self-similar initial data, together with the nonlocal effects of the heat kernel across different spatial scales. By carefully exploiting the structural properties of the self-similar initial data and the coercivity of the linear Leray operator, we derive uniform \textit{a priori} estimates for the approximate solutions. These estimates enable us to pass to the limit and obtain weak forward self-similar solutions, thereby providing a new framework for their construction.
In the four-dimensional case, we establish new Stokes-type estimates for the associated linear Leray system. The key ingredient is an integral representation of the inverse Leray operator as a singular integral operator, whose \(L^p\)-boundedness is established by means of Calder\'on--Zygmund theory. Combining these estimates with a four-dimensional compactness estimate for the nonlinear convective term allows us to absorb the higher-order contribution and close the regularity iteration.
To investigate the spatial decay, we further exploit the integral representation of the inverse Leray operator and recast the decay problem as a boundedness problem in suitably weighted function spaces. The non-integrability of the low-frequency kernel prevents a direct application of standard convolution estimates and leads, under the assumption \(\sigma\in W^{1,\infty}(\mathbb S^3)\), to a logarithmic loss in the pointwise decay. With the additional angular regularity $
\sigma\in C^{1,\gamma}(\mathbb S^3),
\,\, 0<\gamma<1,
$
we establish frequency-localized estimates that remove this logarithmic loss and yield the sharp decay.
\end{abstract}



\maketitle

\tableofcontents

\section{Introduction}

In the present paper, we consider the Cauchy problem for the $n$-dimensional incompressible Navier--Stokes system
\begin{equation}\label{NS-n}
	\left.
	\begin{array}{rll}
		u_t+u\cdot\nabla u-\nu\Delta u+\nabla p=0\\
		\operatorname{div}u=0
	\end{array}
	\right\}
\quad\text{in }\mathbb{R}^+\times\mathbb{R}^n,
\end{equation}
subject to the initial condition
\begin{equation}\label{NS-initial}
u(0,x)=u_0(x)
\quad\text{in }\mathbb{R}^n.
\end{equation}
Here $\nu>0$ denotes the viscosity coefficient.

A fundamental feature of the Navier--Stokes equations is their invariance under the natural parabolic scaling. More precisely, if $(u,p)$ is a solution of \eqref{NS-n}--\eqref{NS-initial}, then, for every $\lambda>0$, the rescaled pair
$$
u_\lambda(x,t)
=
\lambda u(\lambda x,\lambda^2t),
\qquad
p_\lambda(x,t)
=
\lambda^2p(\lambda x,\lambda^2t)
$$
is also a solution, with the corresponding rescaled initial data
\begin{equation}\label{eq-scaling-initial}
u_{0,\lambda}(x)
=\lambda u_0(\lambda x).
\end{equation}
This scaling plays a central role in the analysis of the Navier--Stokes equations and naturally leads to the notion of self-similarity. A solution $(u,p)$ of \eqref{NS-n} is called a \emph{forward self-similar solution} if
\begin{equation}\label{eq-forward-selfsimilar}
u(x,t)
=
\lambda u(\lambda x,\lambda^2t)
\qquad
p(x,t)
=
\lambda^2p(\lambda x,\lambda^2t)
\end{equation}
for every $\lambda>0$. Equivalently, such a solution admits the self-similar representation
\begin{equation}\label{eq-selfsimilar-form}
u(x,t)
=
\frac{1}{\sqrt{t}}
U\left(\frac{x}{\sqrt{t}}\right),
\qquad
p(x,t)
=
\frac{1}{t}
P\left(\frac{x}{\sqrt{t}}\right),
\end{equation}
where $U$ and $P$ are referred to as the velocity and pressure profiles, respectively. The associated initial data are homogeneous of degree $-1$:
\begin{equation*}\label{eq-homogeneous-data}
u_0(\lambda x)=\lambda^{-1}u_0(x),
\qquad \lambda>0.
\end{equation*}
Thus, forward self-similar solutions provide a natural class of solutions associated with scale-invariant initial data.

The study of forward self-similar solutions has a long history in the analysis of the Navier--Stokes equations. In particular, such solutions provide a useful framework for investigating the interaction between the scaling structure of the equations, the regularity of solutions, and possible nonuniqueness phenomena. They have also played an important role in recent developments concerning nonuniqueness in fluid mechanics; see, for example, \cite{Albritton2022,JS} and the references therein.

More generally, self-similar solutions arise naturally in the study of nonlinear parabolic equations with scale-invariant data or singular sources. In this setting, the self-similar transformation reduces the original time-dependent problem to an elliptic equation for the corresponding profile. This reduction often reveals the precise role of the singularity, the scaling properties of the nonlinearity, and the decay behavior of the solution.
A breakthrough in the study of large forward self-similar solutions to the Navier--Stokes system \eqref{NS-n}--\eqref{NS-initial} in $\mathbb{R}^3$ was achieved by Jia and \v{S}ver\'ak in their seminal work \cite{Jia2014}. In particular, they constructed large forward self-similar solutions satisfying the pointwise decay estimate
\begin{equation}\label{pointwise behavior}
	|u(x,t)|\thicksim \frac{1}{\sqrt{t}+|x|},
	\qquad (x,t)\in\mathbb{R}^3\times(0,\infty).
\end{equation}
A key ingredient in their approach is a local-in-space smoothing estimate near the initial time, which enables them to establish the crucial a priori estimates for large forward self-similar solutions. Based on these estimates, the existence of self-similar solutions is then obtained by means of the \emph{Leray--Schauder fixed point theorem} in a weighted $C^{0,\gamma}$ framework; see \cite{Jia2014} for details.

Subsequently, Korobkov and Tsai \cite{KB} developed a different approach based on a blow-up argument and constructed forward self-similar solutions in an energy-space framework. One important advantage of their method is that it can be adapted to the half-space setting. On the other hand, the solutions obtained in this framework are initially characterized by the $L^q$ estimates
\begin{equation*}
\|u(\cdot,t)\|_{L^q(\R^3)} \leq C\|U\|_{L^q(\R^3)}t^{\frac{3}{2q}-\frac{1}{2}},
	\qquad q>3,\quad t>0,
\end{equation*}
where
$$
u(x,t)=\frac{1}{\sqrt{t}}U\left(\frac{x}{\sqrt{t}}\right).
$$
In particular, such estimates alone do not yield the pointwise decay property \eqref{pointwise behavior}. More recently, Lai, Miao, and Zheng \cite{LMZ19} recovered the pointwise bound \eqref{pointwise behavior} for forward self-similar solutions in the full space $\mathbb{R}^3$ by developing a delicate weighted $L^2$ estimate. The weighted $L^2$ technique introduced therein is also of independent interest and has subsequently been useful in related problems.

Since then, a variety of alternative approaches and extensions have been developed for forward self-similar solutions, not only for the three-dimensional incompressible Navier--Stokes system, but also for other viscous fluid systems. We refer, for instance, to \cite{Bra2017,Bra20172,Bra2018,Bra2019,ZP22,Chae2018,Tsai2014,LMZ21,Lai2024} and the references therein. More recently, forward self-similar solutions to the two-dimensional Navier--Stokes system were constructed in \cite{Ren2026,Gui2026} by means of the Leray--Schauder fixed point theorem in two different functional frameworks. More precisely, the approach of \cite{Gui2026} is based on an energy space, whereas a weighted $L^\infty$ framework is employed in \cite{Ren2026}. At the same time, Hou and Song \cite{Hou2026} constructed large forward self-similar solutions to the two-dimensional hypodissipative Navier--Stokes system within an energy-space framework.

It is worth emphasizing that the decay estimates for self-similar solutions constructed in energy spaces in \cite{KB,Gui2026,Hou2026} are closely related to, and inspired by, the weighted $L^2$ approach developed in \cite{LMZ19}. These developments demonstrate that the choice of functional framework plays a fundamental role not only in establishing the existence of large forward self-similar solutions, but also in obtaining their precise spatial decay and pointwise behavior.

In this paper, we are going to show the existence of forward self-similar solutions $u$ for $n\in(2,6)$ and their regularity and decay properties when $n=4$, with scale invariance holding as $u_\lambda=u$ and $p_\lambda=p$ for all $\lambda>0$. Moreover, we find from \eqref{eq-selfsimilar-form} that 
the profile functions $U$ and $P$ depend only on the similarity variable $\xi=x/\sqrt{t}$. Substituting this ansatz into the Navier–Stokes system reduces the time-dependent parabolic problem to a stationary elliptic system 
\begin{equation}\label{profile-system-26}
	\left.
	\begin{aligned}
		-\nu\Delta U
		-\frac12\xi\cdot\nabla U
		-\frac12U
		+(U\cdot\nabla)U
		+\nabla P
		&=0\\
		\nabla\cdot U&=0
	\end{aligned}
	\right\}\quad \text{in}\,\,\mathbb R^n.
\end{equation}
	We refer to \eqref{profile-system-26} as the self-similar profile system, which is commonly known as the Leray problem.

\subsection{Existence of the forward similar solution for  $n\in(2,6)$ }

It is convenient to introduce the linear operator
\begin{equation*}\label{L-operator}
	\mathcal L_\nu U
	:=
	-\nu\Delta U
	-\frac12\xi\cdot\nabla U
	-\frac12U.
\end{equation*}
Then \eqref{profile-system-26} can be written as
\begin{equation}\label{profile-operator}
	\mathcal L_\nu U
	+
	(U\cdot\nabla)U
	+
	\nabla P
	=0,
	\qquad
	\nabla\cdot U=0.
\end{equation}
How to solve the above profile equation? Let $\mathbb P$ denote the Leray projection onto divergence-free
vector fields. Applying $\mathbb P$ to \eqref{profile-operator}
eliminates the pressure and gives
\begin{equation*}\label{projected-profile}
	\mathbb P\mathcal L_\nu U
	+
	\mathbb P\bigl((U\cdot\nabla)U\bigr)
	=0.
\end{equation*}
Formally, this leads to the fixed-point formulation
\begin{equation*}\label{fixed-point-profile}
	U
	=
	-\bigl(\mathbb P\mathcal L_\nu\bigr)^{-1}
	\mathbb P\bigl((U\cdot\nabla)U\bigr),
\end{equation*}
whenever the inverse operator is well defined in an appropriate
function space.

The scaling \eqref{eq-forward-selfsimilar} also determines the natural class of
initial data for forward self-similar solutions. In particular, if
the initial datum is homogeneous of degree $-1$, namely
\begin{equation*}\label{homogeneous-data}
	u_0(\lambda x)=\lambda^{-1}u_0(x),
	\qquad \lambda>0,
\end{equation*}
then it can be written in the form
\begin{equation*}\label{angular-data}
	u_0(x)
	=
	\frac{\sigma(x/|x|)}{|x|},
	\qquad
	x\in\mathbb R^n\setminus\{0\},
\end{equation*}
where $\sigma$ is a vector field on $\mathbb S^{n-1}$. The divergence-free
condition imposes the corresponding compatibility condition on
$\sigma$. For such initial data, the self-similar profile naturally
has the scale-invariant behavior
\begin{equation*}\label{profile-infinity}
	U(\xi)\sim
	\frac{\sigma(\xi/|\xi|)}{|\xi|}
	\qquad\text{as }|\xi|\to\infty,
\end{equation*}
in the appropriate sense.

This scaling behavior suggests working in spaces which are adapted to
the decay $|\xi|^{-1}$. For example, one may consider weighted
pointwise spaces of the form
\begin{equation*}\label{weighted-space}
	X_\beta
	:=
	\bigg\{
	U:
	\sup_{\xi\in\mathbb R^n}
	\langle\xi\rangle^\beta |U(\xi)|<\infty
	\bigg\},
	\qquad
	\langle\xi\rangle=(1+|\xi|^2)^{1/2},
\end{equation*}
with the critical decay corresponding to $\beta=1$. Weighted
Sobolev and Hölder spaces, as well as scale-invariant spaces such as
$L^n(\mathbb R^n)$, $L^{n,\infty}(\mathbb R^n)$, and critical Besov
spaces, also arise naturally in the analysis of
\eqref{profile-system-26}.

There is, however, an important distinction between the
self-similar profile equation \eqref{profile-system-26} and elliptic
equations with a Dirac source. The profile equation
\eqref{profile-system-26} contains no point source at $\xi=0$.
Consequently, a singularity of the form $|\xi|^{-1}$ at the origin
cannot be inferred simply from a relation such as
$-\Delta U=\delta_0$. Rather, the behavior near the origin has to be
determined from the profile equation, the prescribed asymptotic
behavior, and the function class in which the solution is constructed.

On the other hand, if
\[
|U(\xi)|\lesssim |\xi|^{-1},
\qquad
|\nabla U(\xi)|\lesssim |\xi|^{-2},
\]
then the nonlinear term formally satisfies
\begin{equation*}\label{nonlinear-decay}
	|(U\cdot\nabla)U|
	\lesssim |\xi|^{-3}.
\end{equation*}
We observe that $|\xi|^{-3}$ is not locally square integrable in $\mathbb R^n$ with $n\in (2,6)$.  This observation
shows that a direct local $W^{2,p}$ treatment is not suitable for
profiles with this singular behavior. Instead, one is naturally led
to weighted spaces or to a decomposition which separates the
leading singular or scale-invariant part from a more regular
remainder.

More precisely, one may write
\begin{equation}\label{profile-decomposition}
	U=U_0+V,
\end{equation}
where $U_0=e^{\nu \Delta}u_{0}$ is a prescribed leading profile satisfying the desired
$|\xi|^{-1}$ behavior at infinity, while $V$ is a remainder with
improved decay or regularity. Substituting \eqref{profile-decomposition}
into \eqref{profile-system-26} gives
\begin{equation*}\label{remainder-equation}
	\mathcal L_\nu V+\nabla Q
	=
	-(U_0+V)\cdot\nabla(U_0+V),
	\qquad
	\nabla\cdot V=0.
\end{equation*}
After applying the Leray projection, this becomes
\begin{equation*}
	\mathbb P\mathcal L_\nu V
	=
	-\mathbb P\bigl((U_0+V)\cdot\nabla(U_0+V)\bigr), 
\end{equation*}
which implies 
\begin{equation}\label{projected-remainder}
\mathcal L_\nu V
=-
\mathbb P\bigl((V\cdot\nabla)V+L(V)-U_{0}\cdot\nabla U_{0}\bigr)
.
\end{equation}
where 	
$L(V)=-(U_{0}+V)\cdot\nabla V-V\cdot\nabla U_{0}.$

	This formulation is particularly useful when the linear profile $U_0$ is explicitly known or possesses suitable regularity and decay properties. Formally, if $\mathcal L_\nu^{-1}$ is well defined on the underlying function space, then \eqref{projected-remainder} leads to the fixed-point problem
\begin{equation*}\label{fixed-point-V}
	V=T(V),
\end{equation*}
where
\begin{equation*}\label{T-definition}
	T(V)
	:=
	-\mathcal L_\nu^{-1}\mathbb P
	\Big[
	(V\cdot\nabla)V
	+L(V)
	+U_0\cdot\nabla U_0
	\Big].
\end{equation*}
Here the inverse of the linear operator $\mathcal L_\nu$ is given by
 \begin{equation*}
 	\bigl(\mathcal L_\nu\bigr)^{-1}f = \int_0^1   \int_{\mathbb{R}^n}   \Phi \big(x-y,\nu s\big) \frac{1}{(1-s)^{3/2}} \mathbb{P} f\left(\frac{y}{\sqrt{s}}\right) \mathrm{d} y \mathrm{d} s.
 \end{equation*}
 Rather than relying on a contraction argument, which would typically require a smallness assumption, we employ Schaefer's fixed-point theorem. This is particularly suitable for the construction of large self-similar solutions, since the essential issue is to establish uniform a priori bounds for the family of solutions to the homotopy equation. Our argument  relies on Schaefer's fixed-point theorem, which we recall below.
 
 \begin{theorem}[Schaefer's fixed-point theorem]
 	Assume that $E$ is a Banach space and that $T: E \mapsto E$ is a continuous mapping. Assume moreover that
 	\begin{itemize}
 		\item $T$ is compact: the image $T\left(e_n\right), n \in \mathbb{N}$, of any bounded sequence is a relatively compact sequence;
 		\item there exists a finite constant $M$ such that, for every $\lambda \in[0,1]$, if $e=\lambda T(e)$ then $\|e\|_E \leq M$.
 	\end{itemize}
 	Then there exists at least one $e \in E$ such that $T(e)=e$.
 \end{theorem}
 
 To apply this framework, we first introduce a cutoff approximation. Let $\theta \in \mathscr{D}\left(\mathbb{R}^n\right)$ be a cut-off function satisfying $0 \leq \theta \leq 1$, $\theta(x)=1$ for $|x|<1$ and $\theta(x)=0$ for $|x|>2$. We consider the approximate problem
 \[
 -\nu \Delta V_R=-\sum_{i=1}^n \theta\left(\frac{x}{R}\right) V_{R, i} \partial_i\left(\theta\left(\frac{x}{R}\right) V_R\right)-\vec{\nabla} P_R+\cdots
 \]
 with $V_R \in \dot{H}^1$ and $\operatorname{div} V_R=0$. This leads to a truncated fixed-point problem
 \[
 \vec{u}_R=T_R\left(V_R\right),
 \]
 where
 \[
 T_R(\vec{h})=\bigl(\mathcal L_\nu\bigr)^{-1}\mathbb{P} \left(\sum_{i=1}^n \theta\left(\frac{x}{R}\right) h_i \partial_i\left(\theta\left(\frac{x}{R}\right) \vec{h}\right)-\cdots\right).
 \]
 
 The core of the analysis lies in the a-priori estimate. Suppose $\vec{v}=\lambda T_R(\vec{v})$ for some $0 \leq \lambda \leq 1$. Then
 \[
 \nu \Delta \vec{v}=\lambda \mathbb{P}\left(\sum_{i=1}^n \theta\left(\frac{x}{R}\right) v_i \partial_i\left(\theta\left(\frac{x}{R}\right) \vec{v}\right)+L(\vec{v})-U_{0}\cdot\nabla U_{0}\right).
 \]
 Taking the $L^2$ scalar product with $\vec{v}$, we obtain
 \begin{equation}\label{eq-energy-26}
 	 \frac{n-2}{4}\|\vec{v}\|_{L^2}^2 +\|\vec{v}\|_{\dot{H}^1}^2 \leq  {\|\nabla U_0\|_{\infty}}\|\vec{v}\|_{2}^2+\left\|U_{0}\cdot\nabla U_{0}\right\|_{L^2(\mathbb{R}^n)}\|\vec{v}\|_{2}.
 \end{equation}
 	This estimate reveals the main difficulty in the large-data problem. First, the forcing term satisfies
 \begin{equation}\label{U0-nonlinear-L2}
 	U_0\cdot\nabla U_0\in L^2(\mathbb R^n)
 	\qquad\text{for }2<n<6,
 \end{equation}
 provided that $U_0$ and $\nabla U_0$ have the expected critical decay. Second, the positive zeroth-order term
 
 $$
 \frac{n-2}{4}\|V_R\|_{L^2}^2
 $$
  is available precisely when $n>2$. This explains the particular relevance of the range $2<n<6$ in the energy method.
 The main obstruction, however, is that the coefficient
 $\|\nabla U_0\|_{L^\infty}$ is not assumed to be small. Consequently, the term
 $$
 C\|\nabla U_0\|_{L^\infty}\|V_R\|_{L^2}^2
 $$
 cannot, in general, be absorbed into the left-hand side of \eqref{eq-energy-26}. A direct energy argument therefore does not provide the desired uniform bound for large initial data.
 
 To overcome this difficulty, we exploit a crucial observation concerning the nonlocal structure of the problem, rather than relying on the blow-up argument used in \cite{KB}. This nonlocal effect allows us to recover the required a priori control without imposing a smallness condition on the initial data.
Let
 \[
 U_{0,R}^L =e^{\nu\Delta}\left(\mathbb{P}((1-\phi_R) u_0)\right)= \Phi(x,\nu)\ast\left(\mathbb{P}((1-\phi_R) u_0)\right),
 \]
 which represents the regular part of the profile $U_0$. Although $U_{0,R}^{L}$ does not belong to $L^2(\mathbb R^n)$ in general, the regularized component arising from the truncated initial datum becomes small in $L^p$ as the truncation parameter tends to infinity. More precisely, using the representation
$
 u_0(x)=\frac{\sigma(x/|x|)}{|x|},
$
 one can verify that, for every $p>n$,
 \begin{equation}\label{eq-U0R-Lp-small}
 	\lim_{R\to\infty}
 	\|U_{0,R}\|_{L^p( \R^n)}
 	=0.
 \end{equation}
 This decay property plays an important role in controlling the regular part of the self-similar profile and will be used below to derive uniform a priori estimates.

Define the remainder
$$
V:=U-U_{0,R}^{L}.
$$
Then $V$ satisfies the following perturbed system:
 \begin{equation}\label{eq-pert-26}
 	\left.\begin{aligned}
 			-\nu\Delta V- \frac{1}{2 }V-\frac{1}{2 }\xi\cdot \nabla V+V\cdot\nabla V+\nabla P=L(V)-U_{0,R}^L\cdot\nabla U_{0,R}^L+F\\
 		\operatorname{div}\,V=0
 	\end{aligned}\right\},
 \end{equation}
 where
 \[
 L(V)=-U_{0,R}^L \cdot\nabla V-V\cdot\nabla U_{0,R}^L\quad\text{and}
\quad  F= \Delta U_{0,R}^L + \frac{1}{2 }U_{0,R}^L +\frac{1}{2 }\xi\cdot \nabla U_{0,R}^L.
 \]
 By an argument similar to that leading to \eqref{eq-energy-26}, together with the smallness property \eqref{eq-U0R-Lp-small}, we obtain, for all sufficiently large $R>0$, the following uniform a priori estimate:
  \begin{equation*} 
  	\begin{aligned}
 	\frac{n-2}{4}\|V\|_{L^2}^2 +\nu\|V\|_{\dot{H}^1}^2 \leq & {\|\nabla U_{0,R}^L\|_{\infty}}\|V\|_{2}^2+\left(\left\|U_{0}\cdot\nabla U_{0}\right\|_{L^2}+\|F\|_{L^2}\right)\|V\|_{2}\\
 	\leq &\frac{n-2}{8}\|V\|_{2}^2+\left(\left\|U_{0}\cdot\nabla U_{0}\right\|_{L^2}+\|F\|_{L^2}\right)\|V\|_{2},
  	\end{aligned}
 \end{equation*}
 which implies 
 \[\frac{n-2}{4}\|V\|_{L^2}^2 +\nu\|V\|_{\dot{H}^1}^2 \leq C_n \left(\left\|U_{0}\cdot\nabla U_{0}\right\|_{L^2}+\|F\|_{L^2}\right)^2.\]
This uniform estimate enables us to establish the existence of a forward self-similar solution. More importantly, the approximation scheme \eqref{eq-pert-26} can also be incorporated into the Galerkin method developed in this paper, which provides an alternative route to the construction of forward self-similar solutions.

We are now in a position to state our first main theorem concerning the existence of forward self-similar solutions.
 \begin{theorem}\label{thm1.1}
 	Let $n\in(2,6)$ and $u_0=\frac{\sigma(x/|x|)}{|x|}$ with $\sigma\in W^{1,\infty}(\mathbb{S}^{n-1})$. Then   problem \eqref{NS-n}-\eqref{NS-initial} admits at least one forward self-similar solution $(u,p)$ satisfying 
 	$$
 	\left\|u(t)-e^{t \Delta} u_0\right\|_{L^2(\mathbb{R}^n)}=C t^{\frac{n-2}{4} }, \quad\left\|\nabla\left(u(t)-e^{t \Delta} u_0\right)\right\|_{L^2(\mathbb{R}^n)}=C t^{\frac{n-2}{4}}
 	$$
 	for all $t>0$.

 \end{theorem}
 
 \begin{remark}
 	It is worth pointing out that in the case $n=3$, the existence of weak solutions persists under the weaker assumption $\sigma\in L^{\infty}(\mathbb{S}^{2})$, relaxing the requirement $\sigma\in W^{1,\infty}(\mathbb{S}^{2})$ imposed above.
 \end{remark}
\begin{remark}
	We employ a Galerkin approximation scheme to construct weak solutions to the Leray system \eqref{eq-pert-26}. The Galerkin approximation reduces the corresponding finite-dimensional problem to a system of algebraic equations of the form
	$$
	A(d)
	=
	\bigl(A_1(d),\ldots,A_m(d)\bigr)^T
	=0,
	\quad
	d\in\mathbb R^m,
	\quad
	m\in\mathbb Z^+.
	$$
	The key coercive identity
	$$
	-\frac12(v,v)
	-\frac12(x\cdot\nabla v,v)
	=
	\frac{n-2}{4}\|v\|_{L^2}^2
	$$
	immediately yields
	\(
	d\cdot A(d)>0
	\,\,\text{for } |d|=R\gg1.
	\)
	Consequently, by the acute angle principle, which is a consequence of the Brouwer fixed point theorem, there exists $d_0\in B(0,R)\subset\mathbb R^m$ such that
	\(
	A(d_0)=0.
	\)
	This finite-dimensional solvability result provides the uniform a priori estimates needed for the Galerkin approximations. We can then pass to the limit and obtain a weak solution to the Leray system \eqref{eq-pert-26}.
	
	Compared with the conventional Leray--Schauder fixed point approach, our argument is more direct and relies only on the finite-dimensional Brouwer fixed point theorem at the approximation level. To the best of our knowledge, this provides a new application of the Brouwer fixed point theorem in the construction of forward self-similar solutions. We believe that this observation is of independent interest.
\end{remark}

 \subsection{Regularity and decay properties of the forward self-similar solution for $n=4$}
 
 Following the approach developed in \cite{ref-Ger}, in order to establish the regularity of $V$ stated in Theorem \ref{thm1.1} for $n=4$, it suffices to derive suitable Stokes-type estimates for the associated linear system
 \begin{equation}\label{eq.lll}
 	\left.
 		\begin{aligned}
 			-\nu\Delta V
 			-\frac{1}{2\alpha}V
 			-\frac12\xi\cdot\nabla V
 			+\nabla P
 			&=F
 			\\
 			\operatorname{div}V&=0
 		\end{aligned}
 		\right\}
 		\qquad\text{in }\mathbb R^4.
 	\end{equation}
 	Compared with the classical stationary Stokes system, the presence of the zeroth-order term and, in particular, the unbounded drift term
 	$
 	\frac12\xi\cdot\nabla V
 	$
 	creates an additional difficulty in establishing the desired $L^p$ estimates. To overcome this difficulty, we exploit the boundedness properties of the operator associated with the inverse of the linear operator $\mathcal L_\nu$, which is a problem naturally connected with harmonic analysis.
 	
 More precisely, the inverse operator \(\mathcal L_\nu^{-1}\) admits the following integral representation:
 \[
V(x) = -\mathcal{L}_\nu^{-1}\mathbb P\big(F\big )
= \int_0^1 \int_{\mathbb{R}^4} \mathbb{P}\Phi(x-y,\nu s)\frac{1}{(1-s)^{3/2}} F\Big(\frac{y}{\sqrt{s}}\Big)\mathrm{d}y\mathrm{d}s.
\]
 	Consequently, the desired Stokes-type estimates reduce to establishing the $L^p$ boundedness of the operator $T$, namely,
 	\begin{equation}\label{eq-T-Lp}
 		\|(I_d-\Delta )\mathcal{L}_\nu^{-1}\mathbb P \big(F\big )\|_{L^p}\leq C_{p}\|F||_{L^p},
 		\qquad
 		1<p<4.
 	\end{equation}
 	The proof of \eqref{eq-T-Lp} relies on harmonic-analysis techniques. In particular, we show that $(I_d-\Delta )\mathcal{L}_\nu^{-1}\mathbb P$ can be represented as a Calder\'on--Zygmund singular integral operator with a kernel satisfying the standard size and smoothness estimates. The classical $L^p$ theory for Calder\'on--Zygmund operators then yields the boundedness in \eqref{eq-T-Lp}.
 	As a consequence, we obtain the following Stokes-type estimates.\medskip

 \begin{mdframed}
 	Let $p\in(1,\infty)$ and $F \in W^{k,p}(\mathbb{R}^n)$ with $k \geq 0$. Assume that $U \in W^{1,p}(\mathbb{R}^n)$ is a weak solution to problem \eqref{eq.lll}.
 	Then $(U, P) \in W^{k+2,p}(\mathbb{R}^n) \times W^{k+1,p}(\mathbb{R}^n)$, and there exists a constant $C>0$ such that
 	\begin{enumerate}
 		\item[\rm(i)] for each $p\in(1,n)$,
 		\[
 		\|U\|_{\dot{W}^{2,p}(\mathbb{R}^n)}+\|\nabla P\|_{L^p(\mathbb{R}^n)}\leq C\|F\|_{L^p(\mathbb{R}^n)};
 		\]
 		\item[\rm(ii)] for each $k>0$,
 		\[
 		\|U\|_{W^{k+2,p}(\mathbb{R}^n)}+\|\nabla P\|_{W^{k,p}(\mathbb{R}^n)} \leq C\|F\|_{W^{k,p}(\mathbb{R}^n)}+C\|U\|_{L^p(\mathbb{R}^n)}.
 		\]
 	\end{enumerate}
 \end{mdframed}
 On the other hand, the critical dimension \(n=4\) requires a more delicate treatment of the nonlinear convective term. To this end, we establish a four-dimensional compactness lemma. More precisely, if \(U\in\dot H^1(\mathbb{R}^4)\) is divergence-free and \(V\in W^{2,p}(\mathbb{R}^4)\) with \(2\leq p<4\), then, for every \(\varepsilon>0\),
 $$
 \|\operatorname{div}(U\otimes V)\|_{L^p(\mathbb{R}^4)}
 \leq
 \varepsilon\|V\|_{\dot W^{2,p}(\mathbb{R}^4)}
 +C_{\varepsilon,p,U}\|V\|_{L^p(\mathbb{R}^4)},
 $$
 where \(C_{\varepsilon,p,U}>0\) depends only on \(\varepsilon\), \(p\), and \(U\). The key point is that the higher-order contribution of the nonlinear term can be made arbitrarily small and hence absorbed into the left-hand side of the Stokes-type estimate, while the remaining lower-order term can be controlled by the \(L^p\)-norm of \(V\). This perturbative structure plays a decisive role in closing the regularity iteration.

  To investigate the asymptotic decay of the self-similar profile, we make use of the integral representation induced by the inverse operator \(\mathcal{L}_\nu^{-1}\):
  \[  \begin{split}
  	V(x) &= -\mathcal{L}_\nu^{-1}\mathbb P\big((V\cdot\nabla)V+L(V)-U_0\cdot\nabla U_0\big)\\
  	&= \int_0^1 \int_{\mathbb{R}^4} \mathbb{P}\Phi(x-y,\nu s)\frac{1}{(1-s)^{3/2}}\big(U_0\cdot\nabla U_0+\text{lower-order terms}\big)\Big(\frac{y}{\sqrt{s}}\Big)\mathrm{d}y\mathrm{d}s.
  \end{split}\]
  A major difficulty is that the kernel associated with \(\mathbb P\mathcal{L}_\nu^{-1}\) is not integrable over \(\mathbb R^4\), which prevents us from applying standard convolution estimates directly to obtain the desired pointwise decay.
  
  To overcome this difficulty, we recast the problem of estimating the decay of \(V\) as a boundedness problem for the operator
  $
  -\mathcal{L}_\nu^{-1}\mathbb P
  $
  on suitable weighted spaces. This leads us to employ weighted function-space techniques. In particular, for the power weight \(w(x)=|x|^\alpha\), the Muckenhoupt \(A_p\) theory gives the boundedness of convolution operators on weighted \(L^p\) spaces whenever
  $$
  -n<\alpha<4(p-1).
  $$
  At the endpoint \(p=\infty\), if \(h\in\mathscr S(\mathbb R^4)\), then the convolution operator \(Tf=h*f\) is bounded on the weighted essential supremum space \(L^\infty_{|x|^\alpha}\) for \(0\leq\alpha<4\), where
  $$
  \|f\|_{L^\infty_w(\mathbb R^4)}
  :=
  \operatorname*{ess\,sup}_{x\in\mathbb R^4}
  |f(x)|w(x).
  $$
  This observation allows us to treat the high-frequency part of the inverse operator. Indeed, since the kernel of
  $
  -P_{\geq 0}\mathcal L_\nu^{-1}\mathbb P
  $
  is a Schwartz function, the weighted convolution estimate yields
  $$
  |P_{\geq0}V(x)|\leq C|x|^{-3}.
  $$
However, the low-frequency component cannot be handled in the same way because its kernel is not integrable. Under the assumption
  $
  \sigma\in W^{1,\infty}(\mathbb S^3),
  $
  this low-frequency contribution leads to a logarithmic loss, and we obtain only
  $$
  |V(x)|
  \leq
  C|x|^{-3}\log(2+|x|).
  $$
  To recover the sharp pointwise decay, we impose the additional angular regularity
  $
  \sigma\in C^{1,\gamma}(\mathbb S^3),
  \,\,0<\gamma<1.
  $
  The additional Hölder regularity allows us to establish the following dyadic estimate: 
  \begin{equation}\label{eq260921-1}
  	\sup_{x\in\mathbb R^4}
  	|x|^{3+\gamma}
  	\left|
  	\dot\Delta_j V(x)
  	\right|
  	\leq
  	C\,2^{-j\gamma}
  \end{equation}
  for every \(j\in\mathbb Z\).
  The crucial point is that, after localization to each dyadic frequency, the operator
$
  -\dot\Delta_j\mathcal L_\nu^{-1}\mathbb P
  $
 has a rapidly decaying kernel, which permits the use of weighted convolution estimates at each frequency. The gain \(2^{-j\gamma}\) in \eqref{eq260921-1} is precisely what compensates for the logarithmic divergence arising from the low-frequency summation. Consequently, summing the dyadic pieces yields the sharp pointwise decay
  $$
  |V(x)|\leq C|x|^{-3},
  $$
  where \(C>0\) is independent of \(x\).
  
  Thus, the additional \(C^{1,\gamma}\)-regularity of the angular part of the initial data removes the logarithmic loss and enables us to recover the optimal decay rate of the perturbation \(V\). We summarize the resulting regularity and decay estimates for the four-dimensional forward self-similar solutions in the following theorem.
  
  \begin{theorem}\label{thm1.2}
  	Let $u_0=\frac{\sigma(x/|x|)}{|x|}$ with $\sigma\in W^{1,\infty}(\mathbb{S}^{3})$. The  scale-invariant solution $(u,p) $ of  the Cauchy problem \eqref{NS-n}-\eqref{NS-initial}  obtained in Theorem \ref{thm1.1}, then satisfies
  	$$
  	u \in \mathrm{BC}_{\mathrm{w}}\left([0, \infty) ; L_\sigma^{4, \infty}(\mathbb{R}^4)\right)\cap C^\infty(\mathbb{R}^4\times \mathbb{R^+})
  	$$
  	and $$
  	|u(x, t)| \lesssim \frac{1}{|x|+\sqrt{t}}, \quad\left|u(x, t)-e^{t \Delta} u_0\right| \lesssim \frac{ t}{\left(|x|+t^{\frac{1}{2}}\right)^3}\log (2+|x|/\sqrt{t})
  	$$
  	for each $(x, t) \in \mathbb{R}^4 \times(0,+\infty)$.
  	
  	Moreover, if  $\sigma\in C^{1,\gamma}(\mathbb{S}^{3})$ for some $\gamma >0$, we have 
  	$$
  	|u(x, t)| \lesssim \frac{1}{|x|+\sqrt{t}}, \quad\left|u(x, t)-e^{t \Delta} u_0\right| \lesssim \frac{ t}{\left(|x|+t^{\frac{1}{2}}\right)^3}
  	$$
  	for each $(x, t) \in \mathbb{R}^4 \times(0,+\infty)$. In addition,
  	\begin{equation*} 
  		\sup_{x\in\mathbb R^4}
  		\left(|x|+t^{\frac{1}{2}}\right)^{3+\gamma}
  		\left|
  		\dot\Delta_j
  		\left(|u(x, t)-e^{t \Delta} u_0\right)
  		\right|
  		\leq
  		C\,2^{-j\gamma},
  	\end{equation*}
  	for each $j\in \mathbb{Z}$.
  \end{theorem}
  \begin{remark}
We remark that the regularity result in Theorem \ref{thm1.2} follows directly from Proposition \ref{prop-H2}, whereas the logarithmic-loss estimate and the sharp decay rate follow from Propositions \ref{prop-decy-loss-log} and \ref{prop-decy-sharp}, respectively.

  \end{remark}

 \paragraph{\bf Notations} Let $B(0,R)\subset\mathbb R^n$ denote the open ball centered at the origin with radius $R>0$ in $n$-dimensional Euclidean space.
 
 Let $s\in\mathbb R$. We adopt the power weight (nonsingular at the origin)
 \[
 \langle x\rangle^s = \bigl(1+|x|^2\bigr)^{s/2}.
 \]
 
 We now consider functions defined on $(x,t)$. It is sometimes convenient to introduce function spaces with distinct regularity assumptions in the spatial variable $x$ and the time variable $t$, though no universally accepted notation exists for such spaces. Throughout this monograph, we employ the notation
 \[
 C_1^2\bigl(\mathbb R^n_T \triangleq \mathbb R^n\times(0,T]\bigr)
 = \Bigl\{u\colon \mathbb R^n_T\to\mathbb R \;\Big|\; u,\,D_x u,\,D_x^2 u,\,u_t \in C(\mathbb R^n_T)\Bigr\}.
 \]
 In particular, any $u\in C_1^2(\mathbb R^n_T)$ satisfies that $u$, $D_x u$, and all other listed derivatives extend continuously to the upper boundary $\mathbb R^n\times\{t=T\}$.

\section{Preliminary }\label{S2}
In this section, we first briefly review some key facts concerning the Littlewood--Paley decomposition. We refer the reader to \cite{BCD,CWZ-book} for further details.

\subsection{Littlewood-Paley theory}
 For a function $f  \in  \mathscr{S}\left(\mathbb{R}^n\right)$,  the Schwartz space of rapidly decreasing smooth functions, its Fourier transform   $\hat{f}$ (also denoted $\mathscr{F} f$ ) is defined as:
$$
\mathscr{F} f(\xi)=(2 \pi)^{-n / 2} \int_{ \mathbb{R}^{n}  } f(x) e^{  -i x \cdot \xi} \,\mathrm{d} x,\quad \xi \in \mathbb{R}^n.
$$
To convert back from the frequency domain to the original space, the inverse Fourier transform $(\hat{f})^\vee(x)$  (also denoted $\mathscr{F}^{-1}( \hat{f})$ )  is defined as:
$$
f(x)=\mathscr{F}^{-1}( \hat{f})=(2 \pi)^ {-n / 2} \int_{ \mathbb{R}^{n}   } \hat{f}(\xi) e^{ i x \cdot \xi} \,\mathrm{d} \xi.
$$
Next, we state the partition of unity as follows:
there exists $\psi \in C_0^{\infty}\left(\mathbb{R}^n\right)$ with the support property that $\operatorname{supp} \psi \subset \mathbb{R}^n \backslash\{0\}$ and that
\begin{equation}\label{eq-LP-d-1}
\sum_{j=-\infty}^{\infty} \psi\left(2^{-j} \xi\right)=1
\end{equation}
for any $\xi \neq 0$. Moreover, $\psi$ can be chosen to be a radial function, nonegative function and no more than two terms in \eqref{eq-LP-d-1} are nonzero for any given $\xi \neq 0$.

Let us denote $\psi_j(\xi):=\psi\left(2^{-j} \xi\right)$ for any $j \in \mathbb{Z}$.
In what follows we work with the projections $\dot{\Delta}_j$ defined via
$$
\dot{\Delta }_j u:=\left(\psi_j \hat{u}\right)^{\vee}=\check{\psi}_j \ast u
$$
for any $u \in \mathscr{S}^{\prime}(\mathbb{R}^n)$.
\begin{lemma}[Bernstein's inequality]
 Suppose $f \in L^p\left(\mathbb{R}^n\right)$ with $1 \leq p \leq \infty$ satisfies $\operatorname{supp}\hat{f}\subset B(0, R)$. Then
$$
\left\|D^\gamma f\right\|_p \leq C_\gamma R^{|\gamma|}\|f\|_p
$$
for any multiindex $\gamma$ where $C_\gamma=C(n, \gamma)$.
\end{lemma}
As above, we define $\dot{\Delta }_j f=\left(\psi_j \hat{f}\right)^{\vee}$. Then by Plancherel,
\begin{equation}\label{eq-Plancherel}
C^{-1}\|f\|_2^2 \leq \sum_{j \in \mathbb{Z}}\left\|\dot{\Delta }_j f\right\|_2^2 \leq\|f\|_2^2
\end{equation}
for any $f \in L^2\left(\mathbb{R}^n\right)$. Observe that the middle expression is equal to $\|S f\|_2^2$ with
$$
S f=\left(\sum_{j\in\mathbb{Z}}\left|\dot{\Delta }_j f\right|^2\right)^{\frac{1}{2}}.
$$
This is called the Littlewood-Paley square-function. It is a famous result of theirs that \eqref{eq-Plancherel} generalizes to the Littlewood-Paley theorem, that is,
$$
C^{-1}\|f\|_p \leq\|S f\|_p \leq C\|f\|_p
$$
for any $f \in L^p\left(\mathbb{R}^n\right)$ provided $1<p<\infty$ and with $C=C(p, n)$.

Next, we introduce the truncation function $\chi(\xi)$ such that
\begin{equation*}
\chi(\xi)+\sum_{j=0}^{\infty} \psi\left(2^{-j} \xi\right)=1,
\end{equation*}
which is called the inhomogeneous decomposition.

Let
\[
\dot{S}_ju:=\left(\chi_j\widehat{u}\right)^\vee
=\check{\chi}_j*u,
\]
where
\[
\chi_j(\xi)=\chi\left(\frac{\xi}{2^j}\right).
\]
The operator $\dot{S}_j$ is called the homogeneous low-frequency
cut-off operator, which localizes the Fourier spectrum of $u$ to
the region $|\xi|\lesssim 2^j$.

We denote by $\mathscr{S}_h(\mathbb{R}^n)$ the homogeneous
Schwartz space defined as
\[
\mathscr{S}_h(\mathbb{R}^n)
=
\left\{
u\in\mathscr{S}(\mathbb{R}^n):
\lim_{j\to-\infty}\dot{S}_ju=0
\right\}.
\]
It is straightforward to verify that
\[
\dot{S}_ju
=
\sum_{k\leq j-1}\dot{\Delta}_ku,
\qquad
u\in\mathscr{S}_h(\mathbb{R}^n).
\]
Accordingly, we define the high-frequency truncation operator by
\[
P_{\geq j}u
=
\sum_{i\geq j}\dot{\Delta}_i u .
\]
We shall also use the so-called ``fat'' Littlewood--Paley
projection, denoted by \(\widetilde{\dot{\Delta}}_j\), which is defined by
\[
\widetilde{\dot{\Delta}}_j
=
\dot{\Delta}_{j-1}
+\dot{\Delta}_j
+\dot{\Delta}_{j+1}.
\]
By the support properties of the Littlewood--Paley decomposition, it satisfies the
almost orthogonality property
\[
\dot{\Delta}_j\widetilde{\dot{\Delta}}_j
=
\dot{\Delta}_j,
\qquad j\in\mathbb Z.
\]
With these notations at hand, we are ready to introduce the
homogeneous Bony paraproduct decomposition. In the following, we
first recall the basic notation and properties from homogeneous
Littlewood--Paley theory, and then establish the corresponding
paraproduct decomposition.

Next, we recall the homogeneous Bony paraproduct decomposition.
This decomposition separates the nonlinear product into three types
of interactions between different frequency regimes:
low--high frequency interactions, high--low frequency interactions,
and high--high frequency interactions with comparable frequencies.

For two tempered distributions
\[
u,v\in\mathscr{S}'(\mathbb{R}^n),
\]
their product can be formally decomposed as
\[
uv=\dot{T}_u v+\dot{T}_v u+\dot{R}(u,v),
\]
where
\begin{itemize}
	\item $\dot{T}_u v$ denotes the paraproduct term corresponding to
	the interaction between the low frequencies of $u$ and the high
	frequencies of $v$, defined by
	\[
	\dot{T}_u v
	=
	\sum_{j\in\mathbb{Z}}
	\dot{S}_{j-1}u\,\dot{\Delta}_j v;
	\]
	
	\item $\dot{T}_v u$ is the symmetric paraproduct term, which
	describes the interaction between the low frequencies of $v$ and
	the high frequencies of $u$, given by
	\[
	\dot{T}_v u
	=
	\sum_{j\in\mathbb{Z}}
	\dot{S}_{j-1}v\,\dot{\Delta}_j u;
	\]
	
	\item $\dot{R}(u,v)$ is the remainder term, which represents the
	interaction between comparable frequencies:
	\[
	\dot{R}(u,v)
	=
	\sum_{j\in\mathbb{Z}}
	\dot{\Delta}_j u\,
	\widetilde{\dot{\Delta}}_j v .
	\]
\end{itemize}
\subsection{The Weighted Lebesgue Space and Muckenhoupt $A_p$-class}
In this subsection, we introduce power-weighted Lebesgue spaces and
the Muckenhoupt $A_p$-class, which will be used throughout the
following sections. Weighted Lebesgue spaces are generalizations of
the classical Lebesgue spaces, where the integration is performed
with respect to a weighted measure.

Let $(\mathbb{R}^n,\mathcal{L},\mu)$ be a measure space, where
$\mu$ is induced by a non-negative weight function
$w:\mathbb{R}^n\rightarrow[0,\infty)$, namely,
\[
\mathrm{d}\mu(x)=w(x)\,\mathrm{d}x.
\]
For $1\leq p\leq\infty$, the weighted Lebesgue space
$L^p_w(\mathbb{R}^n)$ (also denoted by
$L^p(\mathbb{R}^n,w(x)\mathrm{d}x)$) is defined by
\[
L^p_w(\mathbb{R}^n)
=
\left\{
f:\mathbb{R}^n\rightarrow\mathbb{C}\ \big|\ 
\|f\|_{L^p_w(\mathbb{R}^n)}<\infty
\right\},
\]
where the weighted norm is given by
\[
\|f\|_{L^p_w(\mathbb{R}^n)}
=
\begin{cases}
	\displaystyle
	\left(
	\int_{\mathbb{R}^n}
	|f(x)|^p w(x)\,\mathrm{d}x
	\right)^{1/p},
	&1\leq p<\infty,
	\\[2ex]
	\displaystyle
	\mathop{\operatorname{esssup}}_{x\in\mathbb{R}^n}\abs{f(x)}\,w(x),
	&p=\infty.
\end{cases}
\]

A central theme in harmonic analysis is the characterization of
weights $w$ for which classical operators, such as the
Hardy--Littlewood maximal operator, Calder\'on--Zygmund singular
integral operators, and various convolution operators, are bounded
on weighted spaces $L^p_w(\mathbb{R}^n)$. The Muckenhoupt $A_p$
classes, introduced by Muckenhoupt in the 1970s \cite{MB}, provide
the fundamental framework for this theory. In particular, these
classes characterize the weights for which the Hardy--Littlewood
maximal operator is bounded on $L^p_w(\mathbb{R}^n)$ and guarantee
the boundedness of a broad class of singular integral operators.

\begin{enumerate}[(i)]
	\item A weight $w \in A_p$ (for $1 < p < \infty$) if there exists a constant $C > 0$ such that for all cubes $Q \subset \mathbb{R}^n$:
	\[
	\left( \frac{1}{|Q|} \int_Q w(x) \, \mathrm{d}x \right) \left( \frac{1}{|Q|} \int_Q w(x)^{-1/(p-1)} \, \mathrm{d}x \right)^{p-1} \leq C,
	\]
and  we denote
	$$
	[\omega]_{A_p}:=\sup _Q\left(\frac{1}{|Q|} \int_Q \omega(x) \mathrm{d} x\right)\left(\frac{1}{|Q|} \int_Q \omega(x)^{-\frac{1}{p-1}} \mathrm{~d} x\right)^{p-1}.
	$$
	\item For $p = 1$, $w \in A_1$ if there exists $C > 0$ such that for all cubes $Q$,
	$$
	\frac{1}{|Q|} \int_Q \omega(x) \mathrm{d} x \leq C \mathop{\operatorname{essinf}}_{x \in Q} \omega(x)
	$$
	and the infimum of $C$ is denoted by $[\omega]_{A_1}$.
	
	\item  For $p = \infty$, a weight $\omega$ is called an $A_{\infty}$ weight if
	$$
	[\omega]_{A_{\infty}}:=\sup _Q\left(\frac{1}{|Q|} \int_Q \omega(x) \mathrm{d} x\right) \exp \left(\frac{1}{|Q|} \int_Q \log \omega(x)^{-1} \mathrm{~d} x\right)<\infty.
	$$
	Indeed,  $A_{\infty}=\bigcup_{1 \leq p<\infty} A_p$.
\end{enumerate}
A classic example of $A_p$ weights-one that exemplifies their core properties (see, e.g., \cite{Stein})-is as follows:
\begin{example}[Power weights \cite{Stein} ]\label{example-weight}
	In $\mathbb{R}^n$, a weight function $w(x) = |x|^\alpha$ with $\alpha \in \mathbb{R}$ belongs to the $A_p$ class  for  $1<p<\infty$ if and only if $-n < \alpha < n(p-1)$. In particular, $w(x)=|x|^\alpha \in A _1$ if and only if  $-n<\alpha \leq 0$.
\end{example}
The following discussion concerns the role of the Muckenhoupt
$A_p$ classes in the theory of singular integral operators in
harmonic analysis. A central result in this theory asserts that
classical singular integral operators, including the Hilbert
transform in one dimension and the Riesz transforms in higher
dimensions, extend to bounded operators on weighted spaces
$L^p_w(\mathbb{R}^n)$ for weights $w\in A_p$.
\begin{theorem}[Calder\'on-Zygmund singular integrals, \cite{TH1,TH2}]\label{thm-weight-boundness} Suppose  the operator $Tf=K\ast f$  has the property that there exist constants \(C_1\) and \(C_2\) for which
	\begin{itemize}
		\item [(i)]
		\begin{equation}\label{eq-CZ-con1}
		\|Tf\|_{L^2(\mathbb{R}^n)}\leq C_1\|f\|_{L^2(\mathbb{R}^n)} \quad \text{for all}\quad f\in C_0^\infty(\mathbb{R}^n),
		\end{equation}
		\item [(ii)]
		\begin{equation}\label{eq-CZ-con2}
			|D^\alpha K(x)|\leq C_2\frac{1}{|x|^{n+\alpha}}, \quad \text{for all }\,x\neq0\,\text{and }\,\,|\alpha|\leq1.
		\end{equation}
	\end{itemize}
	Let \( 1 < p < \infty \) and \( w \in A_p \), then \( T \) is bounded on \( L_w^p(\mathbb{R}^n) \). Sharp quantitative bounds are given by
	\[
	\|T\|_{L^p_w(\mathbb{R}^n) \to L^p_w(\mathbb{R}^n)} \leq C_p [w]_{A_p}^{\max(1, 1/(p-1))}.
	\]
\end{theorem}
Weighted Young's inequality for heat kernel convolutions $\Phi_t(x)$ in Weighted \(L^p(\omega)\)-Spaces, which is governed by
\begin{equation}\label{eq-heat-kernal}
	\Phi_t(x)=\frac{1}{\left(4 \pi  t\right)^{n / 2}} \exp \left(-\frac{|x|^2}{4  t}\right), \quad t>0, \,x \in \mathbb{R}^n.
\end{equation} 
	Let \(1 \leq p < \infty\) and \(\omega \in A_p\) (see previous sections for the definition of \(A_p\)-weights). Then the heat kernel convolution operator \(T_t: L^p(\omega) \to L^p(\omega)\) is bounded, and there exists a constant \(C > 0\) (depending only on \(n, p\), and the \(A_p\)-characteristic constant of \(\omega\), but independent of \(t > 0\)) such that:
\[ 
\|\Phi_t * f\|_{L^p(\omega)} \leq C \|f\|_{L^p(\omega)}.  
\]
\begin{corollary}\label{coro-w-bound}
	Let \(w = |x|^\alpha\) with \(-n < \alpha < n(p-1)\), and let \(f \in L^p_w(\mathbb{R}^n)\) where \(1 < p < \infty\). Then, for each \(t > 0\),
		\begin{equation}\label{eq-H-weight}
			\left\|\Phi_t\ast f\right\|_{L^p_w(\mathbb{R}^n)}\leq C\|f\|_{L^p_w(\mathbb{R}^n)}
		\end{equation}
		and 
			\begin{equation}\label{eq-H-weight-nabla}
			\left\|(\nabla\Phi_t)\ast f\right\|_{L^p_w(\mathbb{R}^n)}\leq Ct^{-\frac{1}{2}}\|f\|_{L^p_w(\mathbb{R}^n)}.
		\end{equation}
\end{corollary}
\begin{proof}
	It is obvious that the heat kernel $\Phi_1$ satisfies  \eqref{eq-CZ-con1} and \eqref{eq-CZ-con2}.  Moreover, we have by Theorem \ref{thm-weight-boundness} that 
		\begin{equation}\label{eq-H-weight-1}
		\left\|\Phi_1\ast f\right\|_{L^p_w(\mathbb{R}^n)}\leq C\|f\|_{L^p_w(\mathbb{R}^n)}.
	\end{equation}
The self-similarity of the heat kernel \(\Phi_t(x) = t^{-n/2} \Phi_1(x/\sqrt{t})\) is closely related to the homogeneity of the weight \(\omega(x) = |x|^\alpha\), which satisfies \(|x|^\alpha = (\sqrt{t})^{\alpha} |x/\sqrt{t}|^\alpha\).
 Through the variable substitution \(z = (x-y)/\sqrt{t}\), the convolution integral transforms as:
	\[
	(\Phi_t * f)(x) = t^{-n/2} \int_{\mathbb{R}^n} \Phi_1\left(\frac{x - y}{\sqrt{t}}\right) f(y) \,\mathrm{d}y = \int_{\mathbb{R}^n} \Phi_1(z) f(x - \sqrt{t} z)  \,\mathrm{d}z.
	\]
	Letting $f_t(x)=f(\sqrt{t}x)$, we further obtian 
\begin{equation*}
		(\Phi_t * f)(x)=\big(\Phi_1\ast f_t\big)\big(x/\sqrt{t}\big).
\end{equation*}
In terms of \eqref{eq-H-weight-1}, we have 
	\begin{align*}
		\left\|\Phi_t\ast f\right\|_{L^p_w(\mathbb{R}^n)}=&	\left\|\left(\Phi_1\ast f_t\right)\big(\cdot/\sqrt{t}\big)\right\|_{L^p_w(\mathbb{R}^n)}\\
	=&t^{\frac{n+\alpha}{2p}} \left\|\left(\Phi_1\ast f_t\right) \right\|_{L^p_w(\mathbb{R}^n)}\\
	\leq &Ct^{\frac{n+\alpha}{2p}} \|f_t\|_{L^p_w(\mathbb{R}^n)}=C\|f\|_{L^p_w(\mathbb{R}^n)},
	\end{align*}
which implies \eqref{eq-H-weight}.

A simple calculation yields
\[\nabla\Phi_t(x)=-\frac{x}{2t}\Phi_t(x)=\frac{1}{2\sqrt{t}}\left(\frac{1}{(\sqrt{t})^n}\tilde{\Phi}_1\left(\frac{x}{\sqrt{t}}\right)\right)=:\frac{1}{2\sqrt{t}}\tilde{\Phi}_t(x),\]
 where $\tilde{\Phi}_1=x\Phi_1.$

 It is easy to check that $\tilde{\Phi}_t(x)$ satisfies  \eqref{eq-CZ-con1} and \eqref{eq-CZ-con2}. By the same argument as used in \eqref{eq-H-weight}, we can show 
 \begin{equation}\label{eq-H-weight-tilde}
 	\left\|\tilde{\Phi}_t\ast f\right\|_{L^p_w(\mathbb{R}^n)}\leq C\|f\|_{L^p_w(\mathbb{R}^n)}.
 \end{equation}
 We observe that 
 \[(\nabla\Phi_t)\ast f=\frac{1}{2\sqrt{t}}\tilde{\Phi}_t\ast f.\]
 This together with \eqref{eq-H-weight-tilde} enables us to conclude the second desired estimate  \eqref{eq-H-weight-nabla}.
\end{proof}
\subsection{Useful lemmas}
First, we establish pointwise decay estimates for the convolution by decomposing the integration domain into scales of different sizes and then incorporating these scalewise contributions.
\begin{lemma}\label{lem-point-LS}
	Let $\alpha\in\mathbb{R}$ and $f$ satisfies 
	\[\sup_{x\in\mathbb{R}^n}|x|^\alpha|f|(x)<\infty.\]
	Suppose that $\psi\in L^1(\mathbb{R}^n)$ satifies 
	$\sup_{x\in\mathbb{R}^n}|x|^n|\psi|(x)<\infty;$
	then we have that
	\begin{enumerate}
		\item [\rm (1)]for each $\alpha\geq 0,$
		\begin{equation}\label{eq-point-LS-1}
				\left|\int_{\mathbb{R}^n\backslash B(0,|x|/2)}\psi(x-y)f(y)\,\mathrm{d}y\right|\leq 2^\alpha|x|^{-\alpha}\|\psi\|_{L^1(\mathbb{R}^n)}\sup_{x\in\mathbb{R}^n}|x|^\alpha|f|(x);
		\end{equation}
		\item [\rm (2)] for each $\alpha<n,$
			\begin{equation}\label{eq-point-LS-2}
				\left|	\int_{ B(0,|x|/2)}\psi(x-y)f(y)\,\mathrm{d}y\right|
				\leq	\frac{2^\alpha}{n-\alpha}\alpha(n)|x|^{-\alpha} \sup_{x\in\mathbb{R}^n}|x|^n|\psi|(x)\sup_{x\in\mathbb{R}^n}|x|^\alpha|f|(x),
		\end{equation}
		where \(\alpha(n)\) denotes the volume of the \(n\)-dimensional unit ball.
	\end{enumerate}
\end{lemma}
\begin{proof}
	Firstly, we see by  a simple calculation that  for each $\alpha\geq 0,$
	\begin{equation*}
		\begin{aligned}
			\left|\int_{\mathbb{R}^n\backslash B(0,|x|/2)}\psi(x-y)f(y)\,\mathrm{d}y\right|\leq&2^\alpha|x|^{-\alpha} \sup_{x\in\mathbb{R}^n}|x|^\alpha|f|(x)\int_{\mathbb{R}^n}|\psi|(x-y)\,\mathrm{d}y\\
			\leq &2^\alpha|x|^{-\alpha}\|\psi\|_{L^1(\mathbb{R}^n)}\sup_{x\in\mathbb{R}^n}|x|^\alpha|f|(x),
		\end{aligned}
	\end{equation*}
	which implies the first desired estimate \eqref{eq-point-LS-1}.
	
	Secondly,   we have by  the triangle inequality that  if $|y|\leq\frac{|x|}{2},$
	\[|x-y|\geq |x|-|y|\geq |x|-\frac{|x|}{2}=\frac{|x|}{2}.\]
	Morover, we compute  for each $\alpha<n,$
		\begin{equation}\label{eq-point-LS-3}
		\begin{aligned}
		\left|	\int_{ B(0,|x|/2)}\psi(x-y)f(y)\,\mathrm{d}y\right|\leq& \sup_{x\in\mathbb{R}^n}|x|^\alpha|f|(x)\int_{B(0,|x|/2)}|\psi|(x-y)\frac{1}{|y|^\alpha}\,\mathrm{d}y\\
			\leq &\sup_{x\in\mathbb{R}^n}|x|^\alpha|f|(x)\sup_{x\in\mathbb{R}^n}|x|^n|\psi|(x)\int_{B(0,|x|/2)}\frac{1}{|x-y|^n}\frac{1}{|y|^\alpha}\,\mathrm{d}y\\
				\leq &2^n|x|^{-n}\sup_{x\in\mathbb{R}^n}|x|^\alpha|f|(x)\sup_{x\in\mathbb{R}^n}|x|^n|\psi|(x)\int_{B(0,|x|/2)}\frac{1}{|y|^\alpha}\,\mathrm{d}y\\
			\leq	&\frac{2^\alpha}{n-\alpha}\alpha(n)|x|^{-\alpha} \sup_{x\in\mathbb{R}^n}|x|^\alpha|f|(x)\sup_{x\in\mathbb{R}^n}|x|^n|\psi|(x).
		\end{aligned}
	\end{equation}
Thus, from \eqref{eq-point-LS-3} we  get \eqref{eq-point-LS-2}  and complete the proof of the lemma.
\end{proof}
Next, we present a generalized version of Young's inequality
involving the heat kernel $\Phi_t(x)$. Since the result is formulated
in terms of Lorentz spaces, we first recall some relevant definitions.
For a measurable function $f$ and $s>0$, the decreasing
rearrangement of $f$ is defined by
\[
f^*(s)
=
\inf\left\{
\lambda>0:
\mu_f(\lambda)\leq s
\right\},
\]
where the distribution function of $f$ is given by
\[
\mu_f(\lambda)
=
\left|
\left\{
x\in\mathbb{R}^n:
|f(x)|>\lambda
\right\}
\right|.
\]
 \begin{definition}
For $1 <p<\infty$ and $1 \leq q \leq \infty,$ we define the Lorentz space $L^{p, q}\left(\mathbb{R}^n\right)$ as the space of measurable functions $f$ for which
\begin{itemize}
	\item[(1)] if $1\leq q<\infty$:
	$$
	\|f\|_{L^{p, q}(\mathbb{R}^n)}=\left(\int_0^{\infty}\left[t^{1 / p} f^*(t)\right]^q \frac{\mathrm{d} t}{t}\right)^{1 / q}<\infty;
	$$
\item [(2)]if $q=\infty$:
	$$
	\|f\|_{L^{p, \infty}(\mathbb{R}^n)}=\sup _{t>0} t^{1 / p} f^*(t)<\infty.
	$$
\end{itemize}
 \end{definition}

 \begin{lemma}\label{lem-est-semi-0}
Let $\Phi(x,t)$ be the heat kernel defined in \eqref{eq-heat-kernal}, and  let $f\in L^{p,\infty}(\mathbb{R}^n)$ with $1<p<\infty$.	Then there exist two absolute constants $C_1,\,C_2$ such that 
	\begin{equation*}
		\|\Phi_t\ast f\|_{L^\infty(\mathbb{R}^n)}\leq C_1t^{-\frac{n}{2p}}\|f\|_{L^{p,\infty}(\mathbb{R}^n)}
	\end{equation*}
	and 
	\begin{equation*}
		\|D(\Phi_t\ast f)\|_{L^\infty(\mathbb{R}^n)}\leq C_2t^{-\frac{1}{2}-\frac{n}{2p}}\|f\|_{L^{p,\infty}(\mathbb{R}^n)}.
	\end{equation*}

\end{lemma}
\begin{proof}
Firstly, let us recall O’Neil’s generalized Young inequality \cite[Theorem 3.6]{Neil}: for $1 < p<\infty$ and $1 \leq q_1, q_2 \leq \infty$ with $1 / q_1+1 / q_2 \geq 1$. If $f \in L^{p, q_1}(\mathbb{R}^n)$ and $g \in L^{p^{\prime}, q_2}(\mathbb{R}^n)$, where $p^{\prime}$ is the H\"older conjugate of $p$, then $f \ast g \in L^{\infty}(\mathbb{R}^n)$ and
	\begin{equation}\label{eq-ONeil-young}
	\|f \ast  g\|_{L^{\infty}(\mathbb{R}^n)} \leq\|f\|_{L^{p, q_1}(\mathbb{R}^n)}\|g\|_{L^{p^{\prime}, q_2}(\mathbb{R}^n)}
\end{equation}
Taking $q_1=\infty, \,q_2=1$ in  \eqref{eq-ONeil-young}, we readily have 
\begin{equation}\label{eq-heat-kernal-Y-1}
		\|\Phi_t\ast f\|_{L^\infty(\mathbb{R}^n)}\leq \|\Phi_t \|_{L^{p^{\prime}, 1}(\mathbb{R}^n)}\|f\|_{L^{p, \infty}(\mathbb{R}^n)}.
\end{equation}
 By virtue of the interpolation property of Lorentz spaces, we conclude that
 \begin{align*}
 \|\Phi_t\|_{L^{p^{\prime}, 1}(\mathbb{R}^n)}\leq& \|\Phi_t\|^{\frac{1}{p^{\prime}}}_{L^{1}(\mathbb{R}^n)}\|\Phi_t\|^{\frac{1}{p}}_{L^{\infty}(\mathbb{R}^n)}\\
 \leq &\left(4 \pi  t\right)^{-\frac{n}{2p}}.
 \end{align*}
Inserting this  into \eqref{eq-heat-kernal-Y-1} yields the first desired estimate.

Secondly, by virtue of the differentiability property of convolution, it follows that:
\[D\left(\Phi_t\ast f\right)=\left(D\left(\Phi_t\right)\right)\ast f.\]
Moreover, we get by  using \eqref{eq-ONeil-young} that
	\begin{align*}
		\|D\left(\Phi_t\ast f\right)\|_{L^\infty(\mathbb{R}^n)}\leq& \|D\left(\Phi_t\right)(\cdot,t)\|_{L^{p^{\prime}, 1}(\mathbb{R}^n)}\|f\|_{L^{p, \infty}(\mathbb{R}^n)}\\
		\leq &Ct^{-1}\left\|  {|\cdot|} \Phi_t\right\|_{L^{p^{\prime}, 1}(\mathbb{R}^n)}\|f\|_{L^{p, \infty}(\mathbb{R}^n)}\\
		\leq &Ct^{-\frac{1}{2}-\frac{n}{2p}} \|f\|_{L^{p, \infty}(\mathbb{R}^n)},
	\end{align*}
	which implies the second desired estimate, thereby completing the proof of the lemma.
\end{proof}
Finally, we recall a Hodge--Bogovskii-type decomposition for the
initial data $u_0$ (see, e.g., \cite{MHM}). This decomposition will
be an essential tool in establishing the existence of weak solutions.
\begin{lemma}\label{lem-HB-decom}
	Let $u_0=\frac{\sigma(x/|x|)}{|x|}$ be the divergence-free vector field with $\sigma\in L^\infty(\mathbb{S}^3)$  and $R>0.$ Then $u_0$ admits a decomposition into two components:
	$$u_0=u_{0,1}+u_{0,2}, $$
where 	both $u_{0,1}$ and $u_{0,2}$ are divergence-free, and there exists two positive constants
 $C_1,\,C_2$  depending on $p$ such that 
 \begin{itemize}
 	\item [(1)]for each $p\in (1,4),$
 	\begin{equation*}
 		\|u_{0,1}\|_{L^p(\mathbb{R}^4)}\leq C_1\|u_0\|_{L^p(B_{2R}(0))};
 	\end{equation*}
 		\item[(2)]    for each $p\in (4,\infty),$
 		\begin{equation*}
 			\|u_{0,2}\|_{L^p(\mathbb{R}^4)}\leq C_2\|u_0\|_{L^p(\mathbb{R}^4\backslash B_R(0))}.
 		\end{equation*}
 \end{itemize} 
\end{lemma}
\begin{proof}
First, let us introduce a smoothing cut-off function $\varphi \in C_0^\infty\left(\mathbb{R}^4\right)$ such that $0 \leq \varphi \leq 1,$  
\begin{equation}\label{eq-cutoff-smooth}
\varphi(x)=	\begin{cases}
		1\quad &\text{for }|x|<1;\\
		0&\text{for }|x|>2.
	\end{cases}
\end{equation}
	Letting $\phi_R(x)=\phi\left(\frac{x}{R}\right)$ and $\phi_R^c=1-\phi_R(x)$,  we rewrite $u_0$  as 
	\[u_0=\phi_R u_0+\phi_R^c u_0.\]
Let us denote by $\mathscr{P}$ the Leray projection operator, defined by
$$
\mathscr{P} f=f-\nabla(\operatorname{div}(\Phi * f)),
$$
where the kernel is
$
\Phi(x)=\frac{1}{8 \pi^2|x|^2}.
$
Equivalently, if we write $(\mathscr{P} f)_i$ in components, then
$$
(\mathscr{P} f)_i=f_i-\sum_{j=1}^4 \mathscr{R}_i \mathscr{R}_j f_j,
$$
where $\mathscr{R}_i$ denotes the $i$-th Riesz transform.

 Setting 	 
	$$
u_{0,1}:=\mathscr{P}(\phi_Ru_0)=(\phi_Ru_0)+{\nabla}(\operatorname{div} \left(\Phi\ast (\phi_Ru_0)\right).
	$$
we find by using the incompressible condition $\operatorname{div}u_0=0 $ that 
		\begin{align*}
			u_{0,2}=&u_0-u_{0,1}\\
		=&\mathscr{P}(u_0)-\mathscr{P}(\phi_Ru_0)\\
		=& \mathscr{P}(\phi_R^cu_0).
		\end{align*}
	By virtue of the $L^p$-boundedness of singular integral operators $\mathscr{P}$, we have that   for each $p\in (1,4),$
	\begin{equation*}
		\|u_{0,1}\|_{L^p(\mathbb{R}^4)}\leq C	\|\phi_Ru_0\|_{L^p(\mathbb{R}^4)}\leq C\|u_0\|_{L^p(B_{2R}(0))};
	\end{equation*}
	and  for each $p\in (4,\infty),$
		\begin{equation*}
		\|u_{0,1}\|_{L^p(\mathbb{R}^4)}\leq C	\|\phi_R^cu_0\|_{L^p(\mathbb{R}^4)}\leq C\|u_0\|_{L^p(\mathbb{R}^n\backslash B_R(0))}.
	\end{equation*}
	So,  we finish the proof of Lemma \ref{lem-HB-decom}.
\end{proof}
\section{Linear system}\label{S3}

In this section, we investigate the fundamental properties of the linear operator
\begin{equation*}
	\mathcal L_\nu U
	:=
	-\nu\Delta U
	-\frac12 x\cdot\nabla U
	-\frac12 U,
\end{equation*}
which arises naturally from the self-similar reduction of the Navier--Stokes system. Our main focus is the associated inhomogeneous linear system
\begin{equation}\label{Leraylinear}
	\left.
		\begin{aligned}
			-\nu\Delta U-\frac12 x\cdot\nabla U-\frac12 U+\nabla P
			&= F\\
			\operatorname{div} U
			&=0
		\end{aligned}
		\right\}\quad \text{in } \mathbb{R}^n.
	\end{equation}
	The analysis of this stationary system is closely connected to the self-similar structure of the corresponding time-dependent problem. In particular, the linear heat equation provides a natural framework for understanding the scaling, regularity, and spatial decay properties underlying the Leray system.
	
	We therefore begin by studying self-similar solutions of the linear heat equation. These solutions reflect the fundamental scaling properties of the heat flow and provide the basic tools needed for the subsequent analysis of the inhomogeneous Leray system. In particular, we shall exploit the heat kernel and its derivatives to derive precise estimates for the decay and regularity of the associated velocity fields and forcing terms.

\subsection{Self-similar solution to the heat equation}
In this subsection, we study the higher-dimensional nonhomogeneous
heat equation ($n\geq 2$):
\begin{equation}\label{eq-heat-inh}
	\left\{
	\begin{aligned}
		\mu u_t-\Delta u&=f
		&&\text{in }\mathbb{R}^n\times\mathbb{R}^{+},\\
		u(x,0)&=u_0(x)
		&&\text{in }\mathbb{R}^n .
	\end{aligned}
	\right.
\end{equation}
By applying the Fourier transform method and Duhamel's principle, the
solution can be represented as
\begin{equation}\label{eq-haet-initial}
	\begin{aligned}
		u(x,t)
		&=u_{\mathrm{Hom}}(x,t)+u_{\mathrm{Inhom}}(x,t)\\
		&=\int_{\mathbb{R}^n}
		\Phi_{t/\mu}(x-y)u_0(y)\,\mathrm{d}y
		+\int_0^t\int_{\mathbb{R}^n}
		\Phi_{(t-s)/\mu}(x-y)f(y,s)\,
		\mathrm{d}y\,\mathrm{d}s .
	\end{aligned}
\end{equation}
Here,
\[
\Phi_t(x)
=
\frac{1}{(4\pi t)^{n/2}}
\exp\left(-\frac{|x|^2}{4t}\right)
\]
denotes the heat kernel.

For every $t\geq0$, we define the heat semigroup
$(P_t)_{t\geq0}$ by
\[
P_tf=e^{t\Delta}f=\Phi_t*f .
\]
The homogeneous and inhomogeneous parts can therefore be written as
\[
u_{\mathrm{Hom}}(t)=P_tu_0,
\]
and
\[
u_{\mathrm{Inhom}}(t)
=
\int_0^tP_{t-s}f(s)\,\mathrm{d}s .
\]
We note that the parameter $\mu$ can be removed by a suitable time
rescaling. Indeed, introducing the new time variable
\(\tau=\frac{t}{\mu},\)
the equation can be transformed into the standard heat equation with
$\mu=1$, with a corresponding rescaling of the forcing term.
Since our subsequent analysis focuses on qualitative properties of
the solution, including regularity and decay estimates, we may assume
without loss of generality that
\(\mu=1 .\)

Consequently, the solution takes the standard form
\[
u(t)=P_tu_0+\int_0^tP_{t-s}f(s)\,\mathrm{d}s .
\]
By the superposition principle, we shall analyze separately the
homogeneous evolution generated by the heat semigroup and the
inhomogeneous contribution arising from the forcing term.

\paragraph{\bf Initial-value problem.}
We first consider the initial-value (or Cauchy) problem
\[
\left\{
\begin{aligned}
	u_t-\Delta u&=0 &&\text{in } \mathbb{R}^n\times(0,\infty),\\
	u&=u_0 &&\text{on } \mathbb{R}^n\times\{t=0\}.
\end{aligned}
\right.
\]

The following theorem gives the existence and regularity properties of the solution to the homogeneous heat equation.

\begin{theorem}[Solution of the initial-value problem, \cite{Evans}, Chapter 2.3]
	Assume that 
	\[
	u_0\in C(\mathbb{R}^n)\cap L^\infty(\mathbb{R}^n),
	\]
	and define \(u_{\mathrm{Hom}}(x,t)\) by \eqref{eq-haet-initial}. Then:
	\begin{itemize}
		\item[(i)] 
		\(u_{\mathrm{Hom}}\in C^\infty(\mathbb{R}^n\times(0,\infty))\);
		
		\item[(ii)] 
		\[
		\partial_tu_{\mathrm{Hom}}(x,t)-\Delta u_{\mathrm{Hom}}(x,t)=0,
		\qquad x\in\mathbb{R}^n,\ t>0;
		\]
		
		\item[(iii)] 
		for every \(x^0\in\mathbb{R}^n\),
		\[
		\lim_{\substack{(x,t)\to(x^0,0)\\x\in\mathbb{R}^n,\ t>0}}
		u_{\mathrm{Hom}}(x,t)=u_0(x^0).
		\]
	\end{itemize}
\end{theorem}

The above result describes the solution behavior of the homogeneous heat equation. 
In order to treat more general heat equations with external sources, we next consider the corresponding nonhomogeneous initial-value problem. 
In this case, the forcing term \(f\) plays the role of a source that drives the evolution of the solution.

\paragraph{\bf Nonhomogeneous problem.}
We consider
\[
\left\{
\begin{aligned}
	u_t-\Delta u&=f &&\text{in } \mathbb{R}^n\times(0,\infty),\\
	u&=0 &&\text{on } \mathbb{R}^n\times\{t=0\}.
\end{aligned}
\right.
\]

The following theorem establishes the existence, regularity, and initial behavior of the solution generated by the source term \(f\).

\begin{theorem}[Solution of the nonhomogeneous problem, \cite{LA}, Theorem 4.2]
	Let \(f\) be a continuous function on \([0,\infty)\times\mathbb{R}^n\), which is \(C^1\) with respect to the spatial variables on \((0,\infty)\times\mathbb{R}^n\). Assume that \(f\) and its spatial derivatives are uniformly bounded, and define \(u_{\mathrm{Inhom}}(x,t)\) by \eqref{eq-haet-initial}. Then:
	\begin{itemize}
		\item[(i)]
		\(
		u_{\mathrm{Inhom}}\in C^2_1(\mathbb{R}^n\times(0,\infty));
		\)
		
		\item[(ii)]
		\[
		\partial_tu_{\mathrm{Inhom}}(x,t)-\Delta u_{\mathrm{Inhom}}(x,t)
		=f(x,t),
		\qquad x\in\mathbb{R}^n,\ t>0;
		\]
		
		\item[(iii)]
		for every \(x^0\in\mathbb{R}^n\),
		\[
		\lim_{\substack{(x,t)\to(x^0,0)\\x\in\mathbb{R}^n,\ t>0}}
		u_{\mathrm{Inhom}}(x,t)=0.
		\]
	\end{itemize}
\end{theorem}
In \eqref{eq-haet-initial}, the function 
\(f(x,t):\mathbb{R}^n\times[0,\infty)\rightarrow\mathbb{R}\) is prescribed. 
We seek a solution \(u\) of this equation with the following special structure:
\begin{equation}\label{eq-heat-inh-4}
	u(x,t)=\frac{1}{\sqrt{t}}\,U\left(\frac{x}{\sqrt{t}}\right),
	\qquad x\in\mathbb{R}^n,\ t>0.
\end{equation}
Here, the profile function 
\(U:\mathbb{R}^n\rightarrow\mathbb{R}\) is to be determined. The form \eqref{eq-heat-inh-4} arises naturally when we look for solutions of the heat equation that are invariant under the parabolic dilation scaling
$$
u(x, t) \mapsto \lambda u\left(\lambda x, \lambda^2 t\right),\quad u_0(x) \mapsto \lambda  u_0\left(\lambda x\right) \quad \text{and}\quad f(x, t) \mapsto
\lambda^{3 }f\left(\lambda x, \lambda^2 t\right).
$$
That is, we ask
$$
u(x, t)=\lambda u\left(\lambda  x, \lambda^2 t\right)
$$
for all $\lambda>0, x \in \mathbb{R}^n, t>0$.

Setting $\lambda=t^{-1/2}$, we derive \eqref{eq-heat-inh-4} for $$U(y):=   u(y, 1)\quad\text{and}\quad F(y)=f(y,1),$$
which implies 
\begin{equation}\label{eq-heat-inh-4-R2}
	u(x, t)=\frac{1}{\sqrt{t} }U\left(\frac{x}{\sqrt{t}}\right) \quad \text{and} \quad f(x, t)=\frac{1}{t^{\frac{3}{2}}} F\left(\frac{x}{\sqrt{t}}\right)\quad\left(x \in \mathbb{R}^n, t>0\right).
\end{equation}
Let us insert \eqref{eq-heat-inh-4-R2} into \eqref{eq-heat-inh}, and thereafter compute
\begin{equation}\label{eq-heat-inh-5}
	\frac{\mu}{2} t^{-\frac{3}{2}} U(y)+\frac{\mu}{2} t^{-\frac{3}{2}} y \cdot \nabla U(y)+t^{-\frac{3}{2}} \Delta U(y)=-t^{-\frac{\alpha+1}{2}}F(y)
\end{equation}
for $y:=t^{-1/2} x$. In order to transform \eqref{eq-heat-inh-5} into an expression involving the variable $y$ alone, we take $\alpha=\frac{1}{2}$. Then the terms with $t$ are identical, and so \eqref{eq-heat-inh-5} reduces to
\begin{equation}\label{eq-heat-inh-6}
	-\frac{\mu}{2} U-\frac{\mu}{2} x \cdot \nabla U-\Delta U=F\quad\text{in}\quad\mathbb{R}^n.
\end{equation}
Since \(u_0(x) \mapsto \lambda  u_0\left(\lambda x\right) \) for each $\lambda>0,$ we know that $u_0(x)$ takes the form 
\[u_0(x)=\frac{\sigma(x/|x|)}{|x|}).\]
Combining \eqref{eq-haet-initial} with the relation \eqref{eq-heat-inh-4}, we deduce that, provided
\(u_0\in L^p_{\mathrm{uloc}}(\mathbb{R}^n)\), \(1\leq p\leq\infty\), and
\(F\in C^2(\mathbb{R}^n)\) has compact support, 
\begin{equation*}
	U_{\text {Hom}}(x)= \int_{\mathbb{R}^n}\Phi_{1/\mu}(x-y)u_0(y)\,\mathrm{d}y= \int_{\mathbb{R}^n}\Phi_{1/\mu}(x-y)\frac{\sigma(y/|y|)}{|y|}\,\mathrm{d}y
\end{equation*}
and 
\begin{equation}\label{eq-heat-Evo-Stat}
	\begin{aligned}
		U_{\text {Inhom}}(x) & = \int_0^1   \int_{\mathbb{R}^n}  \Phi_{(1-s)/\mu}\big(x-y\big)\frac{1}{s^{\frac{3}{2}}} F\left(\frac{y}{\sqrt{s}}\right) \mathrm{d} y \mathrm{d} s \\
		& =\int_0^1 \int_{\mathbb{R}^n} \Phi_{s/\mu}(y) \frac{1}{(1-s)^{\frac{3}{2}}} F\left(\frac{x-y}{\sqrt{1-s}}\right) \mathrm{d} y \mathrm{d} s.
	\end{aligned}
\end{equation}
We now study the uniqueness and regularity of distributional solutions to \eqref{eq-heat-inh-6}. 
The notion of a distributional solution is defined as follows:
\begin{definition}[Distributional solutions]
We call $U\in W^{1,p}(\mathbb{R}^n)$ is the distributional  solution to equations \eqref{Leraylinear}, if $F\in L^p(\mathbb{R}^n)$ and
\begin{equation}\label{eq-distri-1inear}
-\frac{1}{2}\int_{\mathbb{R}^n}U\cdot\varphi\,\mathrm{d}x-\frac{1}{2}\int_{\mathbb{R}^n}(x\cdot\nabla) U\cdot\varphi\,\mathrm{d}x+\int_{\mathbb{R}^n}  \nabla U\cdot\nabla\varphi\,\mathrm{d}x=\int_{\mathbb{R}^n}  F\cdot\varphi\,\mathrm{d}x\quad \text{for each}\,\,\varphi\in\mathscr{S}(\mathbb{R}^n).
\end{equation}
\end{definition}
With this definition in hand, we next establish  \(W^{1,p}(\mathbb{R}^n)\) estimate for the distributional solution to problem \eqref{eq-heat-inh-6}. The main result is stated as follows.
\begin{theorem}\label{thm-p-n}
	Let $p\in(1,\infty),$ and  suppose $U\in W^{1,p}(\mathbb{R}^n)$ is the distributional  solution  to problem \eqref{eq-heat-inh-6}.
	Then 
	\begin{equation}\label{eq-est-W1p>n}
		\big\|\nabla U\big\|_{L^p(\mathbb{R}^n)}\leq C_{n,p}\left(\left\| U\right\|_{L^p(\mathbb{R}^n)}+\left\|   F\right\|_{L^p(\mathbb{R}^n)}\right).
	\end{equation}
	In particular, if $p\in(1,n),$ the distributional  solution $U$ is unique and satisfies 
	\begin{equation}\label{eq-est-W1p}
		\big\|  U\big\|_{W^{1,p}(\mathbb{R}^n)}\leq C_{n,p} \left\|F\right\|_{L^p(\mathbb{R}^n)}.
	\end{equation}
	Here  $C_{n, p}>0$ is a constant depending only on $n$ and $p$.
\end{theorem}
\begin{remark}
	Here is why the condition \( p \in (1, n) \) is required. Let us onsider the self-similar ansatz \( u(x,t) = \frac{1}{\sqrt{t}}U(x/\sqrt{t}) \) and \( f (x,t)= \frac{1}{ t^\frac{3}{2}}F(x/\sqrt{t}) \), which is compatible with the following inhomogeneous heat equation
	\[\partial_tu-\Delta u=f.\]
	When \( p \in (1, n) \), one can verify that \( \lim_{t \to 0^+}\|u(t)\|_{L^p(\mathbb{R}^n)} = 0 \). 
	Consequently, the evolution of the $L^p$ norm is controlled by the source term, i.e.
	$
	\|u(t)\|_{L^p\left(\mathbb{R}^n\right)}$is governed by $  \|f(t)\|_{L^p\left(\mathbb{R}^n\right)} .
	$
	Equivalently, there is a (time-independent) relationship such as
	$$
	\big\|  U\big\|_{W^{1,p}(\mathbb{R}^n)}\leq C_{n,p} \left\|F\right\|_{L^p(\mathbb{R}^n)}.
	$$
	In contrast, if  $p \notin(1, n)$, the limit
	$
	\lim _{t \rightarrow 0^{+}}\|u(t)\|_{L^p\left(\mathbb{R}^n\right)}
	$
	need not vanish. Thus the $L^p$ norm of the solution $u(x,t)$ cannot be deduced solely from $\|f(t)\|_{L^p\left(\mathbb{R}^n\right)}$,
	additional information or different techniques such as working in weighted spaces (see Theorem \ref{thm-heat}) are required to relate them.
\end{remark}
\begin{proof} [Proof of Theorem \ref{thm-p-n}]
By setting $\varphi(x)=\Phi_{1/N}(x-y)$ in \eqref{eq-distri-1inear}, we find that  $U^N=\Phi_{1/N}\ast U$ satisfies 
	\begin{equation}\label{Lerayl-Sj-heat-f}
		-\frac{1}{2} U^N-\frac{1}{2} \Phi_{1/N}\ast( {x}\cdot\nabla U)-  \Delta U^N =   F^N \quad\text{in}\quad \mathbb{R}^n,
	\end{equation}
		where $F^N=\Phi_{1/N}\ast F.$ 
		
		Integrating by parts and using $\nabla \Phi_{1/N}(x)=-\frac{N}{2}x\Phi_{1/N}(x)$, we get 
	\begin{align*}
		&\Phi_{1/N}\ast( {x}\cdot\nabla U)\\
		=&\int_{\mathbb{R}^n}\Phi_{1/N}(x-y)(y\cdot\nabla_y)
		U(y)\,\mathrm{d}y\\
		=&-n\int_{\mathbb{R}^n}\Phi_{1/N}(x-y)
		U(y)\,\mathrm{d}y- \int_{\mathbb{R}^n}(y\cdot\nabla_y)\big(\Phi_{1/N}(x-y)\big)
		U(y)\,\mathrm{d}y\\
		=&-n\Phi_{1/N}\ast U+\int_{\mathbb{R}^n}(y\cdot\nabla_x)\big(\Phi_{1/N}(x-y)\big)
		U(y)\,\mathrm{d}y\\
		=&-n  U^N+\int_{\mathbb{R}^n}\left((y-x)\cdot\nabla_x)\Phi_{1/N}(x-y)\right)
		U(y)\,\mathrm{d}y +\int_{\mathbb{R}^n}\left((x\cdot\nabla_x)\Phi_{1/N}(x-y)\right)
		U(y)\,\mathrm{d}y\\
		=&-n  U^N+\frac{N}{2}\int_{\mathbb{R}^n}|x-y|^2 \Phi_{1/N}(x-y) 
		U(y)\,\mathrm{d}y +(x\cdot\nabla)  U^N.
	\end{align*}
	Thus we have 
\begin{equation}\label{eq-250922-1}
	\begin{aligned}
		-\frac{1}{2}	\Phi_{1/N}\ast( {x}\cdot\nabla U)=&\frac{n}{2}  U^N-\left(\frac{x}{2}\cdot\nabla\right)   U^N 
	 -\frac{N}{4}\int_{\mathbb{R}^n}|x-y|^2  \Phi_{1/N}(x-y)
		U(y)\,\mathrm{d}y\\
		=&\frac{n}{2}  U^N-\left(\frac{x}{2}\cdot\nabla\right)  U^N 
	 -\frac{1}{N}\int_{\mathbb{R}^n}\Delta\Phi_{1/N}(x-y)
		U(y)\,\mathrm{d}y -\frac{n}{2} U^N\\
		=&-\left(\frac{x}{2}\cdot\nabla\right)   U^N-\frac{1}{N} \Delta 
		U^N.
	\end{aligned}
\end{equation}
Inserting \eqref{eq-250922-1}  into \eqref{Lerayl-Sj-heat-f} leads to 
\begin{equation}\label{Lerayl-Sj-heat}
	-\frac{1}{2}U^N-\frac{x}{2}\cdot\nabla U^N-\left(1+\frac{1}{N}\right)\Delta U^N=F^N.
\end{equation}
	Setting $u_N(x,t)=\frac{1}{\sqrt{t}} U^N(\frac{x}{\sqrt{t}})$, it is easy to check that  $u_N$ solves
	\begin{equation}\label{Lerayl-Sj-u-heat}
		\partial_tu_N-\left(1+\frac{1}{N}\right)\Delta u_N=\frac{1}{t^{\frac{3}{2}}} F^N\left(\frac{x}{\sqrt{t}}\right) \quad\text{in}\,\,\mathbb{R}^n\times\mathbb{R}^+.
	\end{equation}
Furthermore, it is easy to verify that $w_N(x,t)=t u_N(x,t)={\sqrt{t}}  U^N(\frac{x}{\sqrt{t}})$  satisfies 
	\begin{equation}\label{Lerayl-Sj-w-heat}
		\partial_tw_N-\left(1+\frac{1}{N}\right)\Delta w_N= \frac{1}{\sqrt{t} } F^N\left(\frac{x}{\sqrt{t}}\right)+\frac{1}{\sqrt{t} }  U^N\left(\frac{x}{\sqrt{t}}\right) \quad\text{in}\,\,\mathbb{R}^n\times\mathbb{R}^+.
	\end{equation}
		Letting $v_{k,N}(x,t)=\partial_{x_k} w_N(x,t)=\sqrt{t}\big(\partial_{x_k}  U^N\big)(\frac{x}{\sqrt{t}})$, it follows from \eqref{Lerayl-Sj-w-heat} that  $v_{k,N}$ solves
	\begin{equation}\label{Lerayl-Sj-u-heat-D}
		\partial_tv_{k,N}-\left(1+\frac{1}{N}\right)\Delta v _{k,N}=\frac{3}{2\sqrt{t}}\big(\partial_{x_k}  U^N\big)\left(\frac{x}{\sqrt{t}}\right) +\frac{1}{\sqrt{t}}\partial_{x_k} F^N\left(\frac{x}{\sqrt{t}}\right)  
	\end{equation}
	in $\mathbb{R}^n\times\mathbb{R}^+$. 
	
	Taking the standard $L^p$ estimate of $v_{k,N}$ in \eqref{Lerayl-Sj-u-heat-D} for $k\in\{1,2,\ldots,n\}$, we immediately obtain
	\begin{equation}\label{Lerayl-Sj-u-heat-1-p}
		\begin{split}
			&\frac{1}{p}\frac{\mathrm{d}}{\mathrm{d}t}\big\|v_{k,N}(t)\|^p_{L^p(\mathbb{R}^n)}+(p-1)\left(1+\frac{1}{N}\right)\int_{\mathbb{R}^n}|v_{k,N}(t)|^{p-2}|\nabla v_{k,N}(t)|^2\,\mathrm{d}x\\
			=&\int_{\mathbb{R}^n}  \frac{1}{\sqrt{t}}\partial_{x_k} F^N\left(\frac{x}{\sqrt{t}}\right)|v_{k,N}(t)|^{p-2}v_{k,N}(t)\,\mathrm{d}x 
		 +\int_{\mathbb{R}^n}\frac{3}{2\sqrt{t}}\big(\partial_{x_k}  U^N\big)\left(\frac{x}{\sqrt{t}}\right)|v_{k,N}(t)|^{p-2}v_{k,N}(t)\,\mathrm{d}x.
		\end{split}
	\end{equation}
	Integrating by parts, one has by using the H\"older  inequality  and the Young inequality that
	\begin{align*}
		&\int_{\mathbb{R}^n}\frac{1}{t}\partial_{x_k} F^N\left(\frac{x}{\sqrt{t}}\right)|v_{k,N}(t)|^{p-2}v_{k,N}(t)\,\mathrm{d}x\\
		=&- (p-1)\int_{\mathbb{R}^n} F^N\left(\frac{x}{\sqrt{t}}\right) |v_{k,N}(t)|^{p-2}\nabla v_{k,N}(t)\,\mathrm{d}x\\
		\leq &(p-1)t^{\frac{n}{2p}} \left\|  F^N\right\|_{L^p(\mathbb{R}^n)}\|v_{k,N}\|^{\frac{p}{2}-1}_{L^p(\mathbb{R}^n)}\left\||v_{k,N}|^{\frac{p}{2}-1}\nabla v_{k,N}\right\|_{L^2(\mathbb{R}^n)}\\
		\leq & C_p t^{\frac{n}{p}} \left\|   F\right\|^2_{L^p(\mathbb{R}^n)} \|v_{k,N}\|^{p-2}_{L^p(\mathbb{R}^n)}+\frac{p-1}{4}\left\||v_{k,N}|^{\frac{p}{2}-1}\nabla v_{k,N}\right\|^2_{L^2(\mathbb{R}^n)},
	\end{align*}
	and
	\begin{align*}
		&\int_{\mathbb{R}^n}\frac{3}{2\sqrt{t}}\big(\partial_{x_k}  U^N\big)\left(\frac{x}{\sqrt{t}}\right)|v_{k,N}(t)|^{p-2}v_{k,N}(t)\,\mathrm{d}x\\
		=&-\frac{3(p-1)}{2}\int_{\mathbb{R}^n}\big(   U^N\big)\left(\frac{x}{\sqrt{t}}\right)|v_{k,N}(t)|^{p-2}\partial_{x_k}v_{k,N}(t)\,\mathrm{d}x\\
		\leq& C_p t^{\frac{n}{p}} \left\|  U\right\|^2_{L^p(\mathbb{R}^n)} \|v_{k,N}\|^{p-2}_{L^p(\mathbb{R}^n)}+\frac{p-1}{4}\left\||v_{k,N}|^{\frac{p}{2}-1}\partial_{x_k}v_{k,N}\right\|^2_{L^2(\mathbb{R}^n)}.
	\end{align*}
	Inserting both estimate above  into \eqref{Lerayl-Sj-u-heat-1-p} yields
	\begin{equation}\label{Lerayl-Sj-u-heat-1-p-1}
		\frac{1}{2}\frac{\mathrm{d}}{\mathrm{d}t}\big\|v_{k,N}(t)\|^2_{L^p(\mathbb{R}^n)}
		\leq C_p t^{\frac{n}{p}}\left(\left\| U\right\|^2_{L^p(\mathbb{R}^n)}+\left\|   F\right\|^2_{L^p(\mathbb{R}^n)}\right).
	\end{equation}
	Since  $v_{k,N}(x,t)= \partial_{x_k}u_N=\sqrt{t}\big(\partial_{x_k}  U^N\big)(\frac{x}{\sqrt{t}})$, we have
	\[\big\|v_{k,N}(t)\|_{L^p(\mathbb{R}^n)} =t^{\frac{n+p}{2p}}\big\| \partial_{x_k} U^N\big\|_{L^p(\mathbb{R}^n)}.\]
	Inserting this estimate into \eqref{Lerayl-Sj-u-heat-1-p-1} and taking the supremum over  $k\in\{1,2,\ldots,n\}$ yields
	\begin{equation*}
		\sup_{k} \big\|\partial_{x_k} U^N\big\|^2_{L^p(\mathbb{R}^n)}\leq C_{n,p}\left(\left\| U\right\|^2_{L^p(\mathbb{R}^n)}+\left\|   F\right\|^2_{L^p(\mathbb{R}^n)}\right).
	\end{equation*}
	Furthermore, by using Fatou's lemma, we find that
	\begin{equation}\label{Lerayl-Sj-u-heat-1-p-2}
		\big\|\nabla U\big\|^2_{L^p(\mathbb{R}^n)}\leq C_{n,p}\left(\left\| U\right\|^2_{L^p(\mathbb{R}^n)}+\left\|   F\right\|^2_{L^p(\mathbb{R}^n)}\right),
	\end{equation}
	which means the first desired estimate \eqref{eq-est-W1p>n}.
	
	Now we proceed to show \eqref{eq-est-W1p}.
	Multiplying both sides of \eqref{Lerayl-Sj-u-heat} by $\big|u_N\big|^{p-2}u_N$  and then integrating in space over $\mathbb{R}^n,$ we immediately have  that
	\begin{equation}\label{Lerayl-Sj-u-heat-H0}
		 \frac{1}{p}\frac{\mathrm{d}}{\mathrm{d}t}\left\|u_N(t)\right\|^p_{L^p(\mathbb{R}^n)}+\frac{4(p-1)}{p^2}\left(1+\frac{1}{N}\right)\left\|\nabla |u_N(t)|^{\frac{p}{2}}\right\|_{L^2(\mathbb{R}^n)}^2 
			=  \int_{\mathbb{R}^n} \frac{1}{t^{\frac{3}{2}}} F^N\left(\frac{x}{\sqrt{t}}\right) \big|u_N\big|^{p-2}u_N\,\mathrm{d}x.
	\end{equation}
	By the H\"older inequality, the two quantities in the right side of \eqref{Lerayl-Sj-u-heat-H0} can be bounded by
	\begin{equation}\label{Lerayl-Sj-u-heat-H1}
		\frac{t^{\frac{n}{2p}}}{t^{\frac{3}{2}}} \left\| F^N\right\|_{L^p(\mathbb{R}^n)} \left\|u_N(t)\right\|^{p-1}_{L^p(\mathbb{R}^n)}.
	\end{equation}
	Inserting \eqref{Lerayl-Sj-u-heat-H1} into \eqref{Lerayl-Sj-u-heat-H0} yields
	\begin{equation}\label{Lerayl-Sj-u-heat-H20}
	\frac{1}{p}\frac{\mathrm{d}}{\mathrm{d}t}\left\|u_N(t)\right\|^p_{L^p(\mathbb{R}^n)}+\frac{4(p-1)}{p^2}\left\|\nabla |u_N(t)|^{\frac{p}{2}}\right\|_{L^2(\mathbb{R}^n)}^2
			\leq C {t^{\frac{n-3p}{2p}}} \left\|F\right\|_{L^p(\mathbb{R}^n)} \left\|u_N(t)\right\|^{p-1}_{L^p(\mathbb{R}^n)},
	\end{equation}
	which means
	\begin{equation}\label{Lerayl-Sj-u-heat-H2}
			\frac{\mathrm{d}}{\mathrm{d}t}\left\|u_N(t)\right\|_{L^p(\mathbb{R}^n)}\leq C {t^{\frac{n-3p}{2p}}} \left\|F\right\|_{L^p(\mathbb{R}^n)}.
	\end{equation}
	Since $u_N(x,t)=\frac{1}{\sqrt{t}} U^N(\frac{x}{\sqrt{t}})$, we compute
	\begin{align*}
		\frac{\mathrm{d}}{\mathrm{d}t}\left\|u_N(t)\right\|_{L^p(\mathbb{R}^n)}=\frac{n-p}{2p}t^{\frac{n-3p}{2p}}\left\| U^N\right\|_{L^p(\mathbb{R}^n)}=\frac{1}{2}\left(\frac{n}{p}-1\right)t^{ \frac{n-3p}{2p}}\left\| U^N\right\|_{L^p(\mathbb{R}^n)}.
	\end{align*}
	Plugging this  equality in \eqref{Lerayl-Sj-u-heat-H2}  and using the fact that $p\in(1,n)$ lead  to
	\begin{equation}\label{Lerayl-Sj-u-heat-H2-2R}
		\begin{split}
			\left\| U^N\right\|_{L^p(\mathbb{R}^n) }
			\leq C_{n,p} \left\|F\right\|_{L^p(\mathbb{R}^n)}.
		\end{split}
	\end{equation}
	Since $U\in W^{1,p}(\mathbb{R}^n)$, we get by taking $N\to\infty$ and using Fatou's lemma that
	\begin{equation}\label{Lerayl-Sj-u-heat-H3}
		\left\| U\right\|_{L^p(\mathbb{R}^n)}\leq C_{n,p} \left\|F\right\|_{L^p(\mathbb{R}^n)}.
	\end{equation}
	A simple calculations yields
	\begin{align*}
		&\frac{1}{p}\frac{\mathrm{d}}{\mathrm{d}t}\left\|u_N(t)\right\|^p_{L^p(\mathbb{R}^n)}+\frac{4(p-1)}{p^2}\left\|\nabla |u_N(t)|^{\frac{p}{2}}\right\|_{L^2(\mathbb{R}^n)}^2\\
		=&t^{-1+\frac{n-p}{2}}\left(\frac{n-p}{2p}\left\| U^N\right\|^p_{L^p(\mathbb{R}^n)}+\frac{4(p-1)}{p^2}\left\|\nabla  |U^N |^{\frac{p}{2}}\right\|_{L^2(\mathbb{R}^n)}^2\right).
	\end{align*}
	Inserting this estimate into \eqref{Lerayl-Sj-u-heat-H20}, we obtain
	\begin{equation}\label{Lerayl-Sj-u-heat-H200}
	 \left(\frac{n-p}{2p}\left\| U^N\right\|^p_{L^p(\mathbb{R}^n)}+\frac{4(p-1)}{p^2}\left\|\nabla  |U^N |^{\frac{p}{2}}\right\|_{L^2(\mathbb{R}^n)}^2\right)\\
			\leq C  \left\|F\right\|_{L^p(\mathbb{R}^n)} \left\|U^N\right\|^{p-1}_{L^p(\mathbb{R}^n)}.
	\end{equation}
	Plugging \eqref{Lerayl-Sj-u-heat-H2-2R} in \eqref{Lerayl-Sj-u-heat-H200} and then taking $N\to \infty$, we eventually have that for each $p\in(1,n)$,
	\begin{equation*}\label{Lerayl-Sj-u-heat-0-p}
		\|U\|_{L^p(\mathbb{R}^n)}^p+\left\|\nabla|U|^{\frac{p}{2}}\right\|_{L^2(\mathbb{R}^n)}^2\leq C_{n,p}\left\|F\right\|^p_{L^p(\mathbb{R}^n)}.
	\end{equation*}
	Combined the estimate with estimate \eqref{Lerayl-Sj-u-heat-1-p-2} yields the desired estimate \eqref{eq-est-W1p} in Theorem \ref{thm-p-n}.
	
	Lastly, we show the uniqueness of the distributional  solution for $p\in(1,n)$. Suppose that $U,\,\widetilde{U}\in W^{1,p}(\mathbb{R}^n)$ are two distributional  solution to problem \eqref{eq-heat-inh-6}  with the same force $F.$ Letting $\Delta U=U-\widetilde{U},$ it is easy to verify that $\delta U\in W^{1,p}(\mathbb{R}^n)$ fulfills
	\begin{equation*} 
		-\frac{1}{2}\int_{\mathbb{R}^n}\delta U\cdot\varphi\,\mathrm{d}x-\frac{1}{2}\int_{\mathbb{R}^n}(x\cdot\nabla) \delta U\cdot\varphi\,\mathrm{d}x+\int_{\mathbb{R}^n}  \nabla \delta U\cdot\nabla\varphi\,\mathrm{d}x=0\quad \text{for each}\,\,\varphi\in\mathscr{S}(\mathbb{R}^n).
	\end{equation*}
We now apply \eqref{eq-est-W1p} to  the above problem to obtian $\|\delta U\|_{L^p(\mathbb{R}^n)}=0,$ which implies that $U\equiv\widetilde{U}$ almost everywhere in $\mathbb{R}^n.$
\end{proof}
\begin{theorem}[Representation formula]\label{thm-Rep-forHeatSelf}
Let $p\in(1,n)$ and $F\in L^p(\mathbb{R}^n)$.    Suppose that $U\in W^{1,p}(\mathbb{R}^n)$ be a distributional  solution  to problem \eqref{eq-heat-inh-6} with $\mu=1$.
Then $U$  is necessarily of the form
$$
U(x)=\int_0^1\int_{ \mathbb{R}^{n}  }\Phi_{1-s}(x-y)\frac{1}{s^{\frac{3}{2}}} F\left(\frac{y}{\sqrt{s}}\right)\,\mathrm{d}y\mathrm{d}s\quad\left(x \in \mathbb{R}^n\right).
$$
\end{theorem}
\begin{proof}
	Recall from \eqref{Lerayl-Sj-u-heat} that  $u_N(x,t)=\frac{1}{\sqrt{t}}  U^N(\frac{x}{\sqrt{t}})$ satisfies 
	\begin{equation}\label{Lerayl-Sj-u-heat-R}
		\partial_tu_N-\left(1+\frac{1}{N}\right)\Delta u_N= F^N(x,t) \quad\text{in}\,\,\mathbb{R}^n\times\mathbb{R}^+.
	\end{equation}
	A simple calculation yields 
	\[\|u_N(t)\|_{L^p(\mathbb{R}^n)}=t^{\frac{n-p}{2p}}\| U^N\|_{L^p(\mathbb{R}^n)}.\]
This, coupled with 	$U\in W^{1,p}(\mathbb{R}^n)$ for $p\in(1,n)$, implies that 
	\begin{equation*}
		\lim_{t\to0+}\|u_N(t)\|_{L^p(\mathbb{R}^n)}=0.
	\end{equation*}
Furthermore, given the solution formula for the heat equation \eqref{eq-haet-initial}, along with the homogeneous initial datum \eqref{eq-init-0} and the uniqueness theorem of the linear heat equation, we can conclude that
\begin{equation*}
	u_N(x,t)=\frac{N}{N+1}\int_0^t\int_{\mathbb{R}^n}\Phi_{\frac{(N+1)(t-s)}{N}}(x-y)F^N(y,s)\,\mathrm{d}y\mathrm{d}s
\end{equation*}
is the solution to \eqref{Lerayl-Sj-u-heat-R} with vanishing initial datum.
 
	Consequently, 
	\begin{equation}\label{eq-RF-jj}
	 U^N(x)=u_N(x,1)=\frac{N}{N+1}\int_0^1\int_{\mathbb{R}^n}\Phi_{\frac{(N+1)(1-s)}{N}}(x-y)F^N(y,s)\,\mathrm{d}y\mathrm{d}s.
	\end{equation}
	Since $p\in(1,n)$, we obtain, via the Minkowski and Young inequalities, that
	\begin{equation}\label{eq-R-conver-1}
		\begin{aligned}
	&\left\|	\int_0^1\int_{\mathbb{R}^n}\Phi_{\frac{(N+1)(1-s)}{N}}(x-y)\frac{1}{s^{\frac{3}{2}}} F^N\left(\frac{x}{\sqrt{s}}\right)\,\mathrm{d}y\mathrm{d}s-\int_0^1\int_{\mathbb{R}^n}\Phi_{\frac{(N+1)(1-s)}{N}}(x-y)\frac{1}{s^{\frac{3}{2}}} F\left(\frac{x}{\sqrt{s}}\right)\,\mathrm{d}y\mathrm{d}s
	\right\|_{L^p(\mathbb{R}^n)}\\
	\leq &\int_0^1\left\|\Phi_{\frac{(N+1)(1-s)}{N}}\right\|_{L^1(\mathbb{R}^n)}s^{\frac{n-3p}{2p}}\left\|F^N- F\right\|_{L^{p}(\mathbb{R}^n)}\,\mathrm{d}s\leq C\left\|F^N- F\right\|_{L^{p}(\mathbb{R}^n)}.
\end{aligned}
	\end{equation}
	On the other hand,  we see that 
	\begin{align*}
		&\int_0^1\int_{\mathbb{R}^n}\Phi_{\frac{(N+1)(1-s)}{N}}(x-y)\frac{1}{s^{\frac{3}{2}}} F\left(\frac{x}{\sqrt{s}}\right)\,\mathrm{d}y\mathrm{d}s-\int_0^1\int_{\mathbb{R}^n}\Phi_{1-s}(x-y)\frac{1}{s^{\frac{3}{2}}} F\left(\frac{x}{\sqrt{s}}\right)\,\mathrm{d}y\mathrm{d}s\\
		=&\int_0^1\int_{\mathbb{R}^n}\int_{1-s}^{\frac{(N+1)(1-s)}{N}}\partial_\tau\Phi_{\tau}(x-y)\,\mathrm{d}\tau\frac{1}{s^{\frac{3}{2}}} F\left(\frac{x}{\sqrt{s}}\right)\,\mathrm{d}y\mathrm{d}s.
	\end{align*}
Furthermore, the Minkowski and Young inequalities allow us to infer that
\begin{equation}\label{eq-R-conver-2-f}
	\begin{aligned}
		&\left\|\int_0^1\int_{\mathbb{R}^n}\Phi_{\frac{(N+1)(1-s)}{N}}(x-y)\frac{1}{s^{\frac{3}{2}}} F\left(\frac{x}{\sqrt{s}}\right)\,\mathrm{d}y\mathrm{d}s-\int_0^1\int_{\mathbb{R}^n}\Phi_{1-s}(x-y)\frac{1}{s^{\frac{3}{2}}} F\left(\frac{x}{\sqrt{s}}\right)\,\mathrm{d}y\mathrm{d}s	\right\|_{L^p(\mathbb{R}^n)}\\
		\leq&\int_0^1 \int_{1-s}^{\frac{(N+1)(1-s)}{N}}  \left\|\tau\Phi_{\tau}\right\|_{L^1(\mathbb{R}^n)}s^{\frac{n-3p}{2p}}\left\| F\right\|_{L^{p}(\mathbb{R}^n)}           \,\mathrm{d}\tau\mathrm{d}s \\
			\leq &\int_0^1 \int_{1-s}^{\frac{(N+1)(1-s)}{N}}  \tau^{-1}\,\mathrm{d}\tau s^{\frac{n-3p}{2p}}        \mathrm{d}s\left\| F\right\|_{L^{p}(\mathbb{R}^n)}  \leq \frac{C}{N} \left\| F\right\|_{L^{p}(\mathbb{R}^n)}.
	\end{aligned}
\end{equation}
By the Young inequality again, we conclude that 
	\begin{equation}\label{eq-R-conver-2}
		 \frac{1}{N+1}\left\|\int_0^1\int_{\mathbb{R}^n}\Phi_{1-s}(x-y)\frac{1}{s^{\frac{3}{2}}} F\left(\frac{x}{\sqrt{s}}\right)\,\mathrm{d}y\mathrm{d}s	\right\|_{L^p(\mathbb{R}^n)}\leq \frac{C}{N} \left\| F\right\|_{L^{p}(\mathbb{R}^n)}.
	\end{equation}
	Combining estimates \eqref{eq-R-conver-1}, \eqref{eq-R-conver-2-f} and \eqref{eq-R-conver-2} yields
		\begin{equation*} 
		\begin{aligned}
				&\left\|\frac{N}{N+1}\int_0^1\int_{\mathbb{R}^n}\Phi_{\frac{(N+1)(1-s)}{N}}(x-y)F^N(y,s)\,\mathrm{d}y\mathrm{d}s-\int_0^1\int_{\mathbb{R}^n}\Phi_{1-s}(x-y)\frac{1}{s^{\frac{3}{2}}} F\left(\frac{x}{\sqrt{s}}\right)\,\mathrm{d}y\mathrm{d}s\right\|_{L^p(\mathbb{R}^n)}\\
			\leq& C\left(\left\|F^N- F\right\|_{L^{p}(\mathbb{R}^n)}+\frac{1}{N}\|F\|_{L^{p}(\mathbb{R}^n)} \right).
		\end{aligned}
	\end{equation*}
		Since $F\in L^{p}(\mathbb{R}^n)$, we have 
	\begin{equation}\label{eq-R-conver-0}
		\lim _{N \rightarrow \infty}\left\|F^N-F\right\|_{L^{p}(\mathbb{R}^n)}=0.
	\end{equation}
	Hence, by \eqref{eq-R-conver-0}, we obtain that as \(N\to\infty\),
		\begin{equation}\label{eq-R-conver-3}
		\left\|\frac{N}{N+1}\int_0^1\int_{\mathbb{R}^n}\Phi_{\frac{(N+1)(1-s)}{N}}(x-y)F^N(y,s)\,\mathrm{d}y\mathrm{d}s-\int_0^1\int_{\mathbb{R}^n}\Phi_{1-s}(x-y)\frac{1}{s^{\frac{3}{2}}} F\left(\frac{x}{\sqrt{s}}\right)\,\mathrm{d}y\mathrm{d}s\right\|_{L^p(\mathbb{R}^n)}\to 0.
			\end{equation}
	Passing to the limit as $N\to\infty$ on both sides of  equality \eqref{eq-RF-jj}, and using the convergence properties \eqref{eq-R-conver-0} and \eqref{eq-R-conver-3}, we eventually obtain \begin{equation}\label{eq-expression-U}
			U(x)=\int_0^1\int_{ \mathbb{R}^{n}  }\Phi_{1-s}(x-y)\frac{1}{s^{\frac{3}{2}}} F\left(\frac{y}{\sqrt{s}}\right)\,\mathrm{d}y\mathrm{d}s 
	\end{equation}
	in the sense of $L^p(\mathbb{R}^n)$.  The uniqueness of the distributional solution to problem \eqref{eq-heat-inh-6},  
	established in Theorem \ref{thm-p-n}, directly yields the desired result in Theorem \ref{thm-Rep-forHeatSelf}.
\end{proof}
\begin{theorem}\label{thm-heat}
	Let $p\in(1,\infty)$ and $F\in L^p_{|x|^\alpha}(\mathbb{R}^n)\cap L^p(\mathbb{R}^n)$ with $\alpha\in(-n,n(p-1))$. Suppose $U\in W^{1,p}(\mathbb{R}^n)$ is a distributional solution to problem \eqref{eq-heat-inh-6}. Then
	\begin{equation}\label{eq-Heat-p-7-F}
		\|\nabla U\|_{L^p_{|x|^\alpha}(\mathbb{R}^n)}
		\leq C_{n,p}\left(\|F\|_{L^p_{|x|^\alpha}(\mathbb{R}^n)}+\|U\|_{L^p_{|x|^\alpha}(\mathbb{R}^n)}\right).
	\end{equation}
	In particular, if $p\in(1,n)$, then the distributional solution $U$ satisfies that for every $\alpha\in(p-n,n(p-1))$,
	\begin{equation}\label{eq-Heat-p-6-aaaa-F}
		\|U\|_{L^p_{|x|^\alpha}(\mathbb{R}^n)}+\|\nabla U\|_{L^p_{|x|^\alpha}(\mathbb{R}^n)}
		\leq C_{n,p}\|F\|_{L^p_{|x|^\alpha}(\mathbb{R}^n)}.
	\end{equation}
\end{theorem}

\begin{remark}\label{rem-heat-weighted}
Note that if $\alpha \in [0, n(p-1))$ and $F \in L^p_{\langle x \rangle^\alpha}(\mathbb{R}^n)$, then by Theorem \ref{thm-p-n} and Theorem \ref{thm-heat}, for any $p \in (1, \infty)$,
	\begin{equation*}\label{eq-Heat-p-7-F-J}
	\big\|\nabla U\big\|_{L^p_{\langle x \rangle ^{\alpha}}(\mathbb{R}^n)}
	\leq C_{n,p}\left(\left\|F\right\|_{L^{p}_{\langle x \rangle ^{\alpha}}(\mathbb{R}^n)}+\left\|  U\right\|_{L^p_{\langle x \rangle ^{\alpha}}(\mathbb{R}^n)}\right).
\end{equation*}
Especially, when $p \in (1, n)$, the distributional solution $U$ satisfies
		\begin{equation*}\label{eq-Heat-p-6-aaaa-F-J}
		\left\|  U\right\|_{L^p_{\langle x \rangle ^{\alpha}}(\mathbb{R}^n)}+ \big\|\nabla U\big\|_{L^p_{\langle x \rangle ^{\alpha}}(\mathbb{R}^n)}
		\leq C_{n,p}\left\|F\right\|_{L^{p}_{\langle x \rangle ^{\alpha}}(\mathbb{R}^n)}.
	\end{equation*}
\end{remark}
\begin{proof}[Proof of Theorem \ref{thm-heat}.]
Recall from \eqref{Lerayl-Sj-w-heat} that $w_N(x,t)=t u_N(x,t)={\sqrt{t}}  U^N(\frac{x}{\sqrt{t}})$  satisfies 
	\begin{equation}\label{Lerayl-Sj-w-heat-0827}
		\partial_tw_N-\left(1+\frac{1}{N}\right)\Delta w_N= \frac{1}{\sqrt{t} } F^N\left(\frac{x}{\sqrt{t}}\right)+\frac{1}{\sqrt{t} }  U^N\left(\frac{x}{\sqrt{t}}\right) =:G_N\quad\text{in}\,\,\mathbb{R}^n\times\mathbb{R}^+.
	\end{equation}
Since $U\in W^{1,p}(\mathbb{R}^n)$ and 
$\|w_N(t)\|_{L^p(\mathbb{R}^n)}=t^{\frac{n+p}{2p}}\| U^N\|_{L^p(\mathbb{R}^n)},$
we have 
\begin{equation*}\label{eq-init-0}
	\lim_{t\to0+}\|w_N(t)\|_{L^p(\mathbb{R}^n)}=0.
\end{equation*}
Furthermore, given the solution formula for the heat equation \eqref{eq-haet-initial}, along with the homogeneous initial datum \eqref{eq-init-0} and the uniqueness theorem of the linear heat equation, we can conclude that
\begin{equation*}
	w_N(x,t)=\frac{N}{N+1}\int_0^t\int_{\mathbb{R}^n}\Phi_{\frac{(N+1)(t-s)}{N}}(x-y)G_N(y,s)\,\mathrm{d}y\mathrm{d}s.
\end{equation*}
is the solution to \eqref{Lerayl-Sj-u-heat-R} with vanishing initial datum.
 
Consequently, 
\begin{equation*}\label{eq-RF-jj-2}
 U^N(x)=u_N(x,1)=\frac{N}{N+1}\int_0^1\int_{\mathbb{R}^n}\Phi_{\frac{(N+1)(1-s)}{N}}(x-y)G_N(y,s)\,\mathrm{d}y\mathrm{d}s.
\end{equation*}
The same argument as in proof of \eqref{eq-expression-U} enables us to show that
\begin{equation}\label{eq-expression-UU}
	U(x)=\int_0^1\int_{ \mathbb{R}^{n}  }\Phi_{1-s}(x-y)\left(\frac{1}{ \sqrt{s}} F\left(\frac{y}{\sqrt{s}}\right)+\frac{1}{ \sqrt{s}} U\left(\frac{y}{\sqrt{s}}\right)\right)\,\mathrm{d}y\mathrm{d}s 
\end{equation}
in the sense of $L^p(\mathbb{R}^n)$.

Differentiating both sides of the above equality \eqref{eq-expression-UU} yields
\begin{equation}\label{eq-expression-nabla-UU}
	\nabla U(x)=\int_0^1\int_{ \mathbb{R}^{n}  }(\nabla\Phi_{1-s})(x-y)\left(\frac{1}{ \sqrt{s}} F\left(\frac{y}{\sqrt{s}}\right)+\frac{1}{ \sqrt{s}} U\left(\frac{y}{\sqrt{s}}\right)\right)\,\mathrm{d}y\mathrm{d}s
\end{equation}
The condition $\alpha \in(-n, n(p-1))$, by Example \ref{example-weight}, implies that $|x|^\alpha \in A_p$. Letting $$g_t(x)=\frac{1}{\sqrt{t} }F(x/\sqrt{t})+\frac{1}{\sqrt{t}}U(x/\sqrt{t}),$$ by the  Minkowski inequality and Corollary \ref{coro-w-bound}, we obtain 
\begin{equation*}
	\begin{aligned}
		\|\nabla U\|_{L^p_{|x|^\alpha}(\mathbb{R}^n)}\leq& \int_0^1 \|(\nabla \Phi_{1-s})\ast g_s(x)\|_{L^p_{|x|^\alpha}(\mathbb{R}^n)}\,\mathrm{d}s\\
		\leq& C\int_0^1(1-s) ^{-\frac{1}{2}}\|  g_s(\cdot)\|_{L^p_{|x|^\alpha}(\mathbb{R}^n)}\,\mathrm{d}s\\
		\leq &C\left(\| F\|_{L^p_{|x|^\alpha}(\mathbb{R}^n)}+\| U\|_{L^p_{|x|^\alpha}(\mathbb{R}^n)}\right)\int_0^1(1-s) ^{-\frac{1}{2}}s^{-\frac{p-\alpha-n}{2p}}\,\mathrm{d}s\\
		\leq &C\left(\| F\|_{L^p_{|x|^\alpha}(\mathbb{R}^n)}+\| U\|_{L^p_{|x|^\alpha}(\mathbb{R}^n)}\right).
	\end{aligned}
\end{equation*}	
Combined this with \eqref{eq-est-W1p>n} yields the first desired estimate \eqref{eq-Heat-p-7-F}.
	
Secondly, we consider the case where $p\in(1,n)$. By    Theorem \ref{thm-p-n}, we know that 
\begin{equation}\label{eq-est-W1p-0826}
	\big\|  U\big\|_{W^{1,p}(\mathbb{R}^n)}\leq C_{n,p} \left\|F\right\|_{L^p(\mathbb{R}^n)}.
\end{equation}
According to Theorem \ref{thm-Rep-forHeatSelf}, we write $U$ as
$$
U(x)=\int_0^1\int_{ \mathbb{R}^{n}  }\Phi_{1-s}(x-y)\frac{1}{s^{\frac{3}{2}}} F\left(\frac{y}{\sqrt{s}}\right)\,\mathrm{d}y\mathrm{d}s\quad\left(x \in \mathbb{R}^n\right).
$$
Letting $f_t(x)=\frac{1}{t^{\frac{3}{2}}}F(x/\sqrt{t})$, we infer from
\[
p\in(1,n), \qquad \alpha\in(p-n,n(p-1)),
\]
the Minkowski inequality, and Corollary~\ref{coro-w-bound} that
 \begin{equation*}
 	\begin{aligned}
 	\|U\|_{L^p_{|x|^\alpha}(\mathbb{R}^n)}\leq& \int_0^1 \|\Phi_{1-s}\ast f_s(x)\|_{L^p_{|x|^\alpha}(\mathbb{R}^n)}\,\mathrm{d}s\\
 	\leq& C\int_0^1 \|  f_s(\cdot)\|_{L^p_{|x|^\alpha}(\mathbb{R}^n)}\,\mathrm{d}s\\
 	\leq &C\| F\|_{L^p_{|x|^\alpha}(\mathbb{R}^n)}\int_0^1s^{-\frac{3p-\alpha-n}{2p}}\,\mathrm{d}s\leq C\| F\|_{L^p_{|x|^\alpha}(\mathbb{R}^n)}.
 	\end{aligned}
 \end{equation*}
In the last line, we have used the bound $\frac{3p-\alpha-n}{2p}<1$, which holds under the condition $\alpha\in(p-n,n(p-1))$.
 This together with \eqref{eq-est-W1p-0826} leads to the second desired estimate 
 \begin{equation*}
 	 	\|U\|_{L^p_{| x|^\alpha}(\mathbb{R}^n)} +	\|\nabla U\|_{L^p_{|x| ^\alpha}(\mathbb{R}^n)} \leq C\| F\|_{L^p_{| x|^\alpha}(\mathbb{R}^n)}.
 \end{equation*}
 	So, we finish the proof Theorem \ref{thm-heat}.
\end{proof}
\subsection{The linear of system of Leray equations}
In this subsection, we focus primarily on the maximal regularity of distributional solutions to equation  \eqref{Leraylinear}, which are defined as follows:
\begin{definition}[Distributional solutions]
	Let  $F\in L^{p}(\mathbb{R}^n)$ for $1<p<\infty$ and $P$ satisfies 
	\begin{equation}\label{eq-int-P=0}
		\int_{\mathbb{R}^n}P\,\mathrm{d}x=0.
	\end{equation}
	If the couple $(U,P)\in  W^{1,p}(\mathbb{R}^n)\times \dot{W}^{1,p}(\mathbb{R}^n)$ solves
	\begin{equation}\label{eq-distri-1inear-stokes}
		-\frac{1}{2}\int_{\mathbb{R}^n}U\cdot\varphi\,\mathrm{d}x-\frac{1}{2}\int_{\mathbb{R}^n}(x\cdot\nabla) U\cdot\varphi\,\mathrm{d}x+\int_{\mathbb{R}^n}  \nabla U\cdot\nabla\varphi\,\mathrm{d}x-\int_{\mathbb{R}^n}  P\nabla \cdot\varphi\,\mathrm{d}x=\int_{\mathbb{R}^n}  F\cdot\varphi\,\mathrm{d}x
	\end{equation}
	for each $\varphi\in\mathscr{S}(\mathbb{R}^n),$
	we call the couple $(U,P)$ is the distributional  solution to equations \eqref{Leraylinear}.
\end{definition}
\begin{remark}
	The vanishing condition \eqref{eq-int-P=0} ensures the uniqueness  of $P$. However, when $p\in(1,n)$,  this condition can be omitted due to the following embedding:
	\[\dot{W}^{1,p}(\mathbb{R}^n)\hookrightarrow L^{\frac{np}{n-p}}(\mathbb{R}^n).\]
\end{remark}
Let us begin by addressing the uniqueness of distributional solutions to equation  \eqref{Leraylinear} in the case $1<p<n.$
\begin{lemma}[Uniqueness of distributional solutions]\label{lem-uniqueness}
Let $F\in L^p(\mathbb{R}^n)$  with $1<p< n.$
Assume that $(U_1,P_1)$ and $(U_2,P_2)$  are distributional solutions  of  \eqref{Leraylinear} with the same force $F.$ Then $U_1= U_2$ and $P_1=P_2$.
\end{lemma}
\begin{proof}
Since $(U_1,P_1)$ and $(U_2,P_2)$  are distributional solution of  \eqref{Leraylinear} with the same force $F,$
we easily find that the difference $(\delta U,\delta P):=(U_1-U_2,P_1-P_2)$ fulfills
\begin{equation}\label{eq-weak-1inear-diff}
-\frac{1}{2}\int_{\mathbb{R}^n}U\cdot\varphi\,\mathrm{d}x-\frac{1}{2}\int_{\mathbb{R}^n}(x\cdot\nabla) U\cdot\varphi\,\mathrm{d}x+\int_{\mathbb{R}^n}  \nabla U\cdot\nabla\varphi\,\mathrm{d}x-\int_{\mathbb{R}^n}  P\nabla \cdot\varphi\,\mathrm{d}x=0
\end{equation}
for each $\varphi\in \mathscr{S}(\mathbb{R}^n).$

By setting  $\varphi(x)=\Phi_{1/N}(x-y)$ in \eqref{eq-weak-1inear-diff} and then and then following the same reasoning as in  \eqref{Lerayl-Sj-u-heat}, we arrive at  the conclusion that  $\delta U^N=\Phi_{1/N}\ast \delta U$  and $\delta P^N=\Phi_{1/N}\ast \delta P$ solve 
\begin{equation}
 \left.\begin{aligned}
-\frac{1}{2} \delta U^N-\frac{1}{2} x \cdot\nabla  \delta U^N-\left(1+\frac{1}{N}\right)\Delta  \delta U^N+ \nabla \delta P^N=0\\
\operatorname{div} \delta U^N=0
\end{aligned}\right\}\quad\text{in}\quad \mathbb{R}^n.
\end{equation}
With the help the incompressible condition $\operatorname{div} \delta U^N=0,$  it is easy to check that
\[\Delta  \delta P^N=0.\]
By the Young inequality and $\delta P\in \dot{W}^{1,p}(\mathbb{R}^n)$, one has
\[\| \delta P^N\|_{L^\infty(\mathbb{R}^n)}\leq C2^{\frac{N(n-p)}{p}}\|\delta P\|_{L^{\frac{np}{n-p}}(\mathbb{R}^n)}\leq C2^{\frac{N(n-p)}{p}}\|\delta P\|_{\dot{W}^{1,p}(\mathbb{R}^n)}<\infty.\]
Thus, the Liouville theorem enables us to infer that $ \delta P^N=0$, and the taking $N\to\infty$, we have $\delta P=0,$ which implies  $P_1=P_2$. Henceforth, \eqref{eq-weak-1inear-diff} reduces to
\begin{equation*}
-\frac{1}{2}\int_{\mathbb{R}^n}U\cdot\varphi\,\mathrm{d}x-\frac{1}{2}\int_{\mathbb{R}^n}(x\cdot\nabla) U\cdot\varphi\,\mathrm{d}x+\int_{\mathbb{R}^n}  \nabla U\cdot\nabla\varphi\,\mathrm{d}x=0.
\end{equation*}
Furthermore, by Theorem \ref{thm-p-n}, we eventually obtain $U_1\equiv U_2$.
The desired result concerning uniqueness of distributional solutions follows.
\end{proof}
Next, we present the maximal regularity of distributional solutions to   \eqref{Leraylinear}, namely the $W^{2,p}$ estimates for $U$.
\begin{theorem}[Stokes type estimates] \label{thm-Stokes}
  Let $q\in(1,\infty)$ and $F \in W^{k,q}(\mathbb{R}^n)$ with $k \geq 0$.  Suppose that  the couple  $(U,P)\in  W^{1,p}(\mathbb{R}^n)\times \dot{W}^{1,p}(\mathbb{R}^n)$  is a distributional solution of
equations \eqref{Leraylinear}.
Then $(U, P) \in W^{k+2,q}(\mathbb{R}^n) \times W^{k+1,q}(\mathbb{R}^n)$,  and there exists $C>0$ such  that  
for each $q\in(1,\infty),$
\begin{equation}\label{eq-stokes-est-1}
\|  U\|_{W^{k+2,q}(\mathbb{R}^n)}+\|\nabla P\|_{W^{k,q}(\mathbb{R}^n)} \leq C\|F\|_{W^{k,q}(\mathbb{R}^n)}+C\left\|  U\right\|_{L^q(\mathbb{R}^n)}.
\end{equation}
In particular, for each  $q\in(1,n),$
\begin{equation}\label{eq-stokes-est-2}
\|  U\|_{W^{k+2,q}(\mathbb{R}^n)}+\|\nabla P\|_{W^{k,q}(\mathbb{R}^n)} \leq C\|F\|_{W^{k,q}(\mathbb{R}^n)}
\end{equation}
holds for some positive constant $C$.
\end{theorem}
Before proving Theorem \ref{thm-Stokes}, we first recall the following lemma concerning the maximal \(L^p(L^q)\) regularity of the heat kernel, which plays a crucial role in the proof.
\begin{lemma}[\cite{LA-O} ]\label{lem-maximal}
Let $T \in(0, \infty]$. The operator $A$ defined by $$f(t, x) \mapsto A f(t, x)=\int_0^t e^{(t-s) \Delta} \Delta f(s,x) \,\textnormal{d} s$$
is bounded from $L^p\left((0, T), L^q\left(\mathbb{R}^n\right)\right)$ to $L^p\left((0, T), L^q\left(\mathbb{R}^n\right)\right)$ for every $1<p<\infty$ and $1<q<\infty$.
\end{lemma}
\begin{proof}[Proof of Theorem \ref{thm-Stokes}]
	Let us recall that  $U\in W^{1,q}(\mathbb{R}^n)$  and $P\in \dot{W}^{1,q}(\mathbb{R}^n)$ satisfy
	\begin{equation}\label{eq-weak-1}
			-\frac{1}{2}\int_{\mathbb{R}^n}U\cdot\varphi\,\mathrm{d}x-\frac{1}{2}\int_{\mathbb{R}^n}(x\cdot\nabla) U\cdot\varphi\,\mathrm{d}x+\int_{\mathbb{R}^n}  \nabla P\cdot\varphi\,\mathrm{d}x+\int_{\mathbb{R}^n}  \nabla U\cdot\nabla\varphi\,\mathrm{d}x=\int_{\mathbb{R}^n}  F\cdot\varphi\,\mathrm{d}x	
	\end{equation} for each $\varphi\in\mathscr{S}(\mathbb{R}^n).$

	Taking  test function $\varphi(x)=\Phi_{1/N}(x-y)$ in \eqref{eq-weak-1}, we find that that  $ U^N=\Phi_{1/N}\ast  U$  and $  P^N=\Phi_{1/N}\ast  P$ solve 
	\begin{equation}\label{Lerayl-Sj}
		\left.\begin{aligned}
			-\frac{1}{2}  U^N-\frac{1}{2} x \cdot\nabla U^N-\left(1+\frac{1}{N}\right)\Delta  U^N+ \nabla  P^N= F^N \\
			\operatorname{div} U^N=0
		\end{aligned}\right\}\quad\text{in}\quad \mathbb{R}^n,
	\end{equation}
	where $F^N=\Phi_{1/N}\ast  F$.
	
	Let us denote
	$$
	u_N(x, t)=\frac{1}{\sqrt{t}} U^N\left(\frac{x}{\sqrt{t}}\right),\quad p_N(x, t)=\frac{1}{t} P^N\left(\frac{x}{\sqrt{t}}\right)
	$$
	it is easy to verify that $(u_N,p_N)$ fulfills
	\begin{equation}\label{eq-t-Stokes}
		\left.\begin{array}{rll}
			\partial_tu_N-\left(1+\frac{1}{N}\right)\Delta u_N+ \nabla p_N=\frac{1}{t^{\frac32}}F^N\Big(\frac{x}{\sqrt{t}}\Big)\\
			\operatorname{div} u_N=0
		\end{array}\right\}\quad\text{in}\quad \mathbb{R}^+\times\mathbb{R}^n.
	\end{equation}
Now,  we focus on the case where $1<q<n$ and $k=0.$ A simple computer yields
\[\lim_{t\to0+}\|u_N(t)\|_{L^q(\mathbb{R}^n)}=\| U^N\|_{L^q(\mathbb{R}^n)}\lim_{t\to0+}t^{\frac{n-q}{2q}}=0,\]
which implies 
\[u_N(x,0)=0.\]
This homogeneous initial data together with \eqref{eq-t-Stokes} enables us to conclude that 
\begin{equation}\label{eq-uj-0828}
	u_N(x,t)=\int_0^te^{\frac{N+1}{N}(t-s)\Delta}\mathscr{P}s^{-\frac32}f_j\Big(\frac{x}{\sqrt{s}}\Big)\,\mathrm{d}s,
\end{equation}
where $\mathscr{P}$ is the Leray projector.

Since  $U\in W^{1,q}(\mathbb{R}^n)$  and $P\in \dot{W}^{1,q}(\mathbb{R}^n)$, we obtain by taking $N\to\infty$ on both side of \eqref{eq-uj-0828} that 
\begin{equation}\label{eq-u-0828-1-R}
u(x,t):=\frac{1}{\sqrt{t}} U\left(\frac{x}{\sqrt{t}}\right)=\int_0^te^{(t-s)\Delta}\mathscr{P}s^{-\frac32}F\Big(\frac{x}{\sqrt{s}}\Big)\,\mathrm{d}s.
\end{equation}
Moreover, we get by Lemma \ref{lem-maximal} that for each  $1<p<\infty$ and $1<q<n$
\begin{equation}\label{eq-II-1}
\begin{split}
\|\Delta u\|_{L^p(0,1;L^q(\mathbb{R}^n))}\leq& C\left\|\mathscr{P}s^{-\frac32}F\Big(\frac{\cdot}{\sqrt{s}}\Big)\right\|_{L^p(0,1;L^q(\mathbb{R}^n))}\\
\leq &C\left\| s^{-\frac32}F\Big(\frac{\cdot}{\sqrt{s}}\Big)\right\|_{L^p(0,1;L^q(\mathbb{R}^n))}.
\end{split}
\end{equation}
A simple calculation yields
\begin{equation}\label{eq-II-1-a}
\|\Delta u\|^p_{L^p(0,1;L^q(\mathbb{R}^n))}=\|\Delta U\|^p_{L^q(\mathbb{R}^n)}\int_0^1s^{-\frac{p}{2}+\frac{np}{2q}}\,\mathrm{d}s.
\end{equation}
On the other hand,
\begin{equation}\label{eq-II-1-b}
\left\| s^{-\frac32}F\Big(\frac{\cdot}{\sqrt{s}}\Big)\right\|_{L^p(0,1;L^q(\mathbb{R}^n))}^p=\|F\|^p_{L^q(\mathbb{R}^n)}\int_{0}^1s^{-\frac{3p}{2}+\frac{np}{2q}}\,\mathrm{d}s.
\end{equation}
Since $q\in(1,n)$, there exists $p_0\in(1,\infty)$ such that 
\[-\frac{3}{2}+\frac{n}{2q}>-\frac{1}{p_0}.\]
Inserting \eqref{eq-II-1-a} and \eqref{eq-II-1-b} into \eqref{eq-II-1} and then taking $p=p_0$, we readily have that for each $q\in(1,n),$
\begin{equation}\label{eq-II-c}
\begin{split}
\|\Delta U\|_{L^q(\mathbb{R}^n)}\leq& C\|F\|_{L^q(\mathbb{R}^n)}\left(\int_{0}^1s^{-\frac{p_0}{2}+\frac{np_0}{2q}}\,\mathrm{d}s\right)^{\frac{1}{p_0}}
\left(\int_0^1s^{-\frac{3p_0}{2}+\frac{np_0}{2q}}\,\mathrm{d}s\right)^{-\frac{1}{p_0}}\\
\leq& C\|F\|_{L^q(\mathbb{R}^n)}.
\end{split}
\end{equation}
Next, we  handle  the case where $p\in[n,\infty)$ and $k=0$. Denoting $v_N(x,t)=tu_N(x,t)$ and $\pi_N(x,t)=tp_N(x,t)$,
it follows from \eqref{eq-t-Stokes} that $(v_N,\pi_N)$ fulfills
 \begin{equation}\label{eq-t-Stokes-v}
 \left.\begin{array}{rll}
 \partial_tv_N-\left(1+\frac{1}{N}\right)\Delta v_N+ \nabla \pi_N=\frac{1}{\sqrt{t}}F^N\Big(\frac{x}{\sqrt{t}}\Big)+u_N\\
\operatorname{div} v_N=0
\end{array}\right\}\quad\text{in}\quad \mathbb{R}^+\times\mathbb{R}^n.
\end{equation}
 Since $U\in W^{1,q}(\mathbb{R}^n)$ and 
$\|v_N(t)\|_{L^q(\mathbb{R}^n)}=t^{\frac{n+q}{2q}}\|U^N\|_{L^q(\mathbb{R}^n)},$
we can supplement \eqref{eq-t-Stokes-v} with the initial condition 
\[v_N(x,0)=0.\]
Then, following the same process as in \eqref{eq-u-0828-1-R}, we can show that 
\begin{equation*}
	v(x,t):= \sqrt{t} U\left(\frac{x}{\sqrt{t}}\right)=\int_0^te^{(t-s)\Delta}\mathscr{P}\left(s^{-\frac12}F\Big(\frac{x}{\sqrt{s}}\Big)+s^{-\frac12}U\Big(\frac{x}{\sqrt{s}}\Big)\right)\,\mathrm{d}s.
\end{equation*}
Furthermore, we get by Lemma \ref{lem-maximal} and uniqueness of the linear heat equations that for each  $1<p<\infty$ and $n\leq q<\infty$
\begin{equation*}
\begin{split}
\|\Delta v\|_{L^p(0,1;L^q(\mathbb{R}^n))}\leq& C\left\|\mathscr{P}s^{-\frac12}F\Big(\frac{\cdot }{\sqrt{s}}\Big)\right\|_{L^p(0,1;L^q(\mathbb{R}^n))}+C\left\| s^{-\frac12}U\Big(\frac{\cdot}{\sqrt{s}}\Big)\right\|_{L^p(0,1;L^q(\mathbb{R}^n))}\\
\leq &C\left\| s^{-\frac12}F\Big(\frac{\cdot}{\sqrt{s}}\Big)\right\|_{L^p(0,1;L^q(\mathbb{R}^n))}+C\left\| s^{-\frac12}U\Big(\frac{\cdot}{\sqrt{s}}\Big)\right\|_{L^p(0,1;L^q(\mathbb{R}^n))}.
\end{split}
\end{equation*}
The same argument as above, we can show that for each $q\in[n,\infty),$
\begin{equation}\label{eq-II-c>n}
\|\Delta U\|_{L^q(\mathbb{R}^n)}\leq  C\left(\|F\|_{L^q(\mathbb{R}^n)}+\|U\|_{L^q(\mathbb{R}^n)}\right).
\end{equation}
Proceeding, we observe by applying the divergence operator to the first equation of \eqref{Leraylinear} and using the incompressible condition $\operatorname{div}U=0$ that
\begin{equation}\label{eq-Stokes-P}
-\Delta P=-\operatorname{div}F.
\end{equation}
By elliptic estimate, we immediate have that for each  $1<q<\infty$,
\begin{equation}\label{eq-Stokes-P-a}
\|\nabla P\|_{L^q(\mathbb{R}^n)}\leq C\|F\|_{L^q(\mathbb{R}^n)}.
\end{equation}
From \eqref{eq-II-c}, \eqref{eq-II-c>n}, and \eqref{eq-Stokes-P-a}, the classical elliptic estimate yields, for each \(q \in (1,\infty)\), immediately:
 \begin{equation}\label{Lerayl-Sj-2}
\sup_{|\alpha|= 2}\|D^\alpha  U\|_{L^q(\mathbb{R}^n)}+\|\nabla  P\|_{L^q(\mathbb{R}^n)}\leq C\|F\|_{L^q(\mathbb{R}^n)}.
\end{equation}
On the other hand, the incompressible condition enables us to rewrite \eqref{eq-weak-1} as
	\begin{equation}\label{eq-weak-1-P0829}
	-\frac{1}{2}\int_{\mathbb{R}^n}U\cdot\varphi\,\mathrm{d}x-\frac{1}{2}\int_{\mathbb{R}^n}(x\cdot\nabla) U\cdot\varphi\,\mathrm{d}x =\int_{\mathbb{R}^n}  \mathscr{P}F\cdot\varphi\,\mathrm{d}x\quad\text{for each}\quad\varphi\in\mathscr{S}(\mathbb{R}^n).
\end{equation}
Applying Theorem \ref{thm-p-n} to equation \eqref{eq-weak-1-P0829} yields that for each $q\in(1,\infty)$
\begin{equation}\label{eq-est-W1p-tq>n}
	\begin{aligned}
		\big\|\nabla U\big\|_{L^q(\mathbb{R}^n)}\leq&C_{n,q}\left(\left\| U\right\|_{L^q(\mathbb{R}^n)}+\left\|  \mathscr{P} F\right\|_{L^q(\mathbb{R}^n)}\right)\\
		\leq& C_{n,q}\left(\left\| U\right\|_{L^q(\mathbb{R}^n)}+\left\|   F\right\|_{L^q(\mathbb{R}^n)}\right).
	\end{aligned}
\end{equation}
More specifically, when   $q\in(1,n),$  the solution $U$ satisfies 
\begin{equation}\label{eq-est-tq-W1p}
	\big\|  U\big\|_{W^{1,q}(\mathbb{R}^n)}\leq C_{n,q} \left\|\mathscr{P}F\right\|_{L^q(\mathbb{R}^n)}\leq C_{n,q} \left\| F\right\|_{L^q(\mathbb{R}^n)}.
\end{equation}
Combining \eqref{Lerayl-Sj-2} and  \eqref{eq-est-W1p-tq>n} leads to the  first desired estimate \eqref{eq-stokes-est-1} for $k=0$.

We turn to establish the estimate for $\|U\|_{\dot{W}^{k+2,q}(\mathbb{R}^n)}$ with $k>0$, where  the homogeneous  Sobolev space $\dot{W}^{k+2,q}(\mathbb{R}^n)$  is defined via the Fourier transform
\begin{equation}\label{eq-def-Sobolev-Fourier}
\dot{W}^{k, q}\left(\mathbb{R}^n\right)=\left\{u \in \mathscr{S}^{\prime}\left(\mathbb{R}^n\right);\,\, \mathscr{F}^{-1}\left( |\xi|^ k\hat{u}(\xi)\right) \in L^q\left(\mathbb{R}^n\right)\right\}.
\end{equation}
 For simplicity, we denote
 \[\Lambda^k u(x):=\mathscr{F}^{-1}\left( |\xi|^ k\hat{u}(\xi)\right).\]
Setting $v_{k,N}(x,t)=t^{k+3/2}\Lambda^k u_N(x,t)$ and $\pi_{k,N}(x,t)=t^{k+3/2}\Lambda^kp_N(x,t)$,
it follows from \eqref{eq-t-Stokes} that $(v_{k,N},\pi_{k,N})$ fulfills
\begin{equation}\label{eq-t-Stokes-v-kn}
	\left.\begin{array}{rll}
		\partial_tv_{k,N}-\left(1+\frac{1}{N}\right)\Delta v_{k,N}+ \nabla \pi_{k,N}=t^{\frac{k}{2}}\Lambda^k F^N+(k+3/2)t^{k+\frac12}\Lambda^ku_N\\
		\operatorname{div} v_{k,N}=0
	\end{array}\right\}\quad\text{in}\quad \mathbb{R}^+\times\mathbb{R}^n.
\end{equation}
Since $U\in W^{k,q}(\mathbb{R}^n)$ and 
$\|v_{k,N}(t)\|_{L^q(\mathbb{R}^n)}=t^{\frac{n+q(k+2)}{2q}}\|\Lambda^kU^N\|_{L^q(\mathbb{R}^n)},$
we can supplement \eqref{eq-t-Stokes-v-kn} with the initial condition 
\[v_{k,N}(x,0)=0.\]
Then, following the same process as in \eqref{eq-uj-0828}, we can show that 
\begin{equation*}
	v_{k,N}(x,t)=\int_0^ts^{\frac {k}{2}}e^{\frac{N+1}{N}(t-s)\Delta}\mathscr{P}\left(\Lambda^kF^N\Big(\frac{x}{\sqrt{s}}\Big)+(k+3/2) \Lambda^kU^N\Big(\frac{x}{\sqrt{s}}\Big)\right)\,\mathrm{d}s.
\end{equation*}
Taking $N\to\infty$ in the above equality leads to 
\begin{equation*}\label{eq-u-0828-1}
	v_{k}(x,t)=t^{k+3/2}\Lambda^k u(x,t)=\int_0^ts^{\frac {k}{2}}e^{(t-s)\Delta}\mathscr{P}\left(\Lambda^kF\Big(\frac{x}{\sqrt{s}}\Big)+(k+3/2) \Lambda^kU\Big(\frac{x}{\sqrt{s}}\Big)\right)\,\mathrm{d}s.
\end{equation*}
Moreover, we obtain by Lemma \ref{lem-maximal} and uniqueness of the linear heat equations that for each  $1< q<\infty$,
\begin{equation}\label{eq-II-v-1}
	 \|\Delta v_{k}\|_{L^q(0,1;L^q(\mathbb{R}^n))}
		\leq  C\left\|s^{\frac{k}{2}}\Lambda^kF\Big(\frac{\cdot }{\sqrt{s}}\Big)\right\|_{L^q(0,1;L^q(\mathbb{R}^n))}+C\left\| s^{\frac{k}{2}}\Lambda^kU\Big(\frac{\cdot}{\sqrt{s}}\Big)\right\|_{L^q(0,1;L^q(\mathbb{R}^n))}.
\end{equation}
Since $\Lambda^kU= v_{k}(x,1),$ it follows from \eqref{eq-II-v-1} that 
\begin{equation*}\label{eq-est-w-j-k0829}
	\begin{aligned}
		\|\Delta\Lambda^kU\|_{L^q(\mathbb{R}^n)}=& \|\Delta v_{k}\|_{L^q(0,1;L^q(\mathbb{R}^n))}\left(\int_0^1s^{\frac{n+q(k+2)}{2}}\,\mathrm{d}s\right)^{-\frac{1}{q}}\\
		\leq &C\left\| \Lambda^kF \right\|_{ L^q(\mathbb{R}^n))}+C\left\|  \Lambda^kU \right\|_{ L^q(\mathbb{R}^n)}.
	\end{aligned}
	\end{equation*}
By Fatou's lemma and the definition of Sobolev space \eqref{eq-def-Sobolev-Fourier}, we obtain   that for each $q\in(1,\infty)$  and $k>0,$
\begin{equation}\label{eq-U-dot-w2q-0828}
	\|U\|_{\dot{W}^{k+2,q}(\mathbb{R}^n)}\leq C	\|F\|_{\dot{W}^{k,q}(\mathbb{R}^n)}+C	\|U\|_{\dot{W}^{k,q}(\mathbb{R}^n)}.
\end{equation}
We apply the classical elliptic estimate to \eqref{eq-Stokes-P}  to obtain that for each $1<q<\infty$ and $k>0$,
\begin{equation}\label{eq-Stokes-P-a-k}
	\|\nabla P\|_{W^{k,q}(\mathbb{R}^n)}
	\leq C\|F\|_{W^{k,q}(\mathbb{R}^n)}.
\end{equation}
By using \eqref{eq-U-dot-w2q-0828} and \eqref{eq-Stokes-P-a-k},  we immediately obtain that for each $1<q<\infty$ and $k>0$,
\begin{equation}\label{eq-U-w2q-0828}
	\begin{aligned}
		\| U\|_{W^{k+2,q}(\mathbb{R}^n)}+\|\nabla  P\|_{ {W}^{k,q}(\mathbb{R}^n)}
	\leq&	\|U\|_{L^q(\mathbb{R}^n)}+ 	\| U\|_{\dot{W}^{k+2,q}(\mathbb{R}^n)}+\|\nabla  P\|_{ {W}^{k,q}(\mathbb{R}^n)}\\
		\leq & C	\|F\|_{\dot{W}^{k,q}(\mathbb{R}^n)}+C	\|U\|_{\dot{W}^{k,q}(\mathbb{R}^n)}.
	\end{aligned}
\end{equation}
By the interpolation theorem and the Young inequality, we see that for each $\varepsilon>0,$
\begin{align*}
	\|U\|_{\dot{W}^{k,q}(\mathbb{R}^n)}\leq& \|U\|^{\frac{2}{k+2}}_{L^q(\mathbb{R}^n)}\|U\|^{\frac{k}{k+2}}_{\dot{W}^{k+2,q}(\mathbb{R}^n)}\\
	\leq &C_\varepsilon \|U\|_{L^q(\mathbb{R}^n)}+\varepsilon \|U\|_{\dot{W}^{k+2,q}(\mathbb{R}^n)}.
\end{align*}
Inserting this estimate into \eqref{eq-U-w2q-0828} and then choosing  $\varepsilon$ appropriately small, we finally establish that for each $1<q<\infty$ and $k>0$,
\begin{equation}\label{eq-U-w2q-0828-low}
		\| U\|_{W^{k+2,q}(\mathbb{R}^n)}+\|\nabla  P\|_{ {W}^{k,q}(\mathbb{R}^n)}
		\leq  C	\left(\|F\|_{\dot{W}^{k,q}(\mathbb{R}^n)}+	\|U\|_{L^q(\mathbb{R}^n)}\right),
\end{equation}
which is the  first desired estimate \eqref{eq-stokes-est-1} for $k>0$.

Combining \eqref{eq-U-w2q-0828-low} with \eqref{eq-est-tq-W1p}, we eventually obtain the second desired estimate \eqref{eq-stokes-est-2}. Thus, we complete the proof of Theorem \ref{thm-Stokes}.
\end{proof}
\subsection{Pointwise  estimate}
In this subsection, we establish pointwise estimates for solutions to the Poisson equation and the nonhomogeneous heat equation. These estimates will serve as fundamental tools in Section \ref{S4}, where we study more complex nonlinear partial differential equations \eqref{PL} in the case \(n=4\). We begin with the analysis of the Poisson equation.
{The boundedness of the convolution operator }
Let $w(x)=|x|^\alpha$ and assume $w\in A_p$ for some $1<p<\infty$,
i.e.
\[
-n<\alpha<n(p-1).
\]
For $h\in\Sclass(\R^n)$ define $Tf=h*f$. The standard weighted
norm inequality
\[
\norm{Tf}_{\Lp[|x|^\alpha]{p}}\leq C\,\norm{f}_{\Lp[|x|^\alpha]{p}},
\qquad 1<p<\infty,
\]
is well known. The natural question is what happens at the endpoint
$p=\infty$. As we shall see, the answer depends sensitively on the
definition of the weighted $L^\infty$ norm, and the admissible range
of $\alpha$ is \emph{not} obtained by sending $p\to\infty$ in the
$A_p$ condition.
\begin{proposition}\label{prop:main}
	Let $w(x)=|x|^\alpha$ and
	$h\in\Sclass(\R^n)$.	Then the convolution operator $Tf=h*f$ is bounded on
	$L^\infty_{|x|^\alpha}$ if and only if
	$
	{\,0\leq \alpha<n.\,}
	$
\end{proposition}
	\begin{remark}
		The above boundedness is stable under the natural scaling of the
		convolution kernel. More precisely, let
		\[
		h_\lambda(x):=\lambda^n h(\lambda x),
		\qquad \lambda>0,
		\]
		where $h\in\mathscr{S}(\mathbb{R}^n)$. If
		$0\leq\alpha<n$, then
		\begin{equation}\label{eq-prop-main-lambda}
				\|h_\lambda*f\|_{L^\infty_{|x|^\alpha}}
			\leq C_{\alpha,h}
			\|f\|_{L^\infty_{|x|^\alpha}},
			\qquad \lambda>0,
		\end{equation}
		where the constant $C_{\alpha,h}$ is independent of $\lambda$.
		Thus, the family of convolution operators
		$\{h_\lambda*\}_{\lambda>0}$ is uniformly bounded on
		$L^\infty_{|x|^\alpha}$.
		
		Indeed, setting
		\[
		g(z):=f(z/\lambda),
		\]
		we have
		\[
		(h_\lambda*f)(x)
		=(h*g)(\lambda x).
		\]
		Consequently,
		\[
			\|h_\lambda*f\|_{L^\infty_{|x|^\alpha}}
			=
			\sup_{x\in\mathbb{R}^n}
			|x|^\alpha |(h*g)(\lambda x)|=
			\lambda^{-\alpha}
			\|h*g\|_{L^\infty_{|x|^\alpha}}.
		\]
		On the other hand,
		\[
		\|g\|_{L^\infty_{|x|^\alpha}}
		=
		\lambda^\alpha
		\|f\|_{L^\infty_{|x|^\alpha}}.
		\]
		Therefore, by the preceding proposition,
		\[
		\|h_\lambda*f\|_{L^\infty_{|x|^\alpha}}
		\leq
		C_{\alpha,h}
		\|f\|_{L^\infty_{|x|^\alpha}},
		\]
		with a constant independent of $\lambda$.
	\end{remark}
\begin{proof}[Proof of Proposition \ref{prop:main}]
	We prove the three directions separately.
	
	\noindent 	(1) \textbf{Positive direction: $0\leq\alpha<n$}
	
	Boundedness is equivalent to the pointwise kernel estimate
	\begin{equation}\label{eq:kernel-1}
		\int_{\R^n}\abs{h(x-y)}\,\abs{y}^{-\alpha}\,\mathrm{d}y
		\leq C\,\abs{x}^{-\alpha}.
	\end{equation}
	Indeed, if $M=\norm{f}_{\Lp[|x|^\alpha]{\infty}}$, then
	$\abs{f(y)}\leq M\abs{y}^{-\alpha}$, and
	\[
	\abs{(h*f)(x)}\leq M\int_{\mathbb{R}^n}\abs{h(x-y)}\abs{y}^{-\alpha}\,\mathrm{d}y
	\leq CM\abs{x}^{-\alpha},
	\]
	which is exactly $$\norm{h*f}_{\Lp[|x|^\alpha]{\infty}}\leq CM.$$
	It remains to establish \eqref{eq:kernel-1}. Since $h$ is Schwartz,
	for every $N>0$ there exists $C_N$ such that
	\[
	\abs{h(z)}\leq C_N(1+\abs{z})^{-N}.
	\]
	\textbf{Region I: $|x|\leq 1$.}
	Split the integral into $|y|\leq 2$ and $|y|>2$.
	\begin{itemize}
		\item For $|y|\leq 2$: $\abs{h(x-y)}\leq C$, and
		\[
		\int_{|y|\leq 2}\abs{y}^{-\alpha}\,\mathrm{d}y<\infty
		\]
		precisely because $\alpha<n$.
		\item For $|y|>2$: $\abs{x-y}\geq \abs{y}-\abs{x}\geq \abs{y}/2$,
		so $\abs{h(x-y)}\leq C_N\abs{y}^{-N}$, and
		\[
		\int_{|y|>2}\abs{y}^{-N-\alpha}\,\mathrm{d}y<\infty
		\]
		for $N$ sufficiently large.
	\end{itemize}
	Hence the left-hand side of \eqref{eq:kernel-1} is bounded by a
	constant. Since $|x|^{-\alpha}\geq 1$ for $|x|\leq 1$ and
	$\alpha\geq 0$, estimate \eqref{eq:kernel-1} follows.
	
	\noindent	\textbf{Region II: $|x|>1$.}
	Split into $|y|\leq |x|/2$ and $|y|>|x|/2$.
	\begin{itemize}
		\item For $|y|\leq |x|/2$: $\abs{x-y}\geq \abs{x}/2$, so
		$\abs{h(x-y)}\leq C_N\abs{x}^{-N}$. Then
		\[
		\int_{|y|\leq |x|/2}\abs{x}^{-N}\abs{y}^{-\alpha}\,\mathrm{d}y
		\leq C\abs{x}^{-N}\abs{x}^{n-\alpha}
		=C\abs{x}^{n-\alpha-N}.
		\]
		Choosing $N>n$ gives a bound $\leq C\abs{x}^{-\alpha}$ (since
		$|x|>1$ and $n-N<0$).
		
		\item For $|y|>|x|/2$: since $\alpha\geq 0$,
		$\abs{y}^{-\alpha}\leq C\abs{x}^{-\alpha}$, and therefore
		\[
		\int_{|y|>|x|/2}\abs{h(x-y)}\abs{y}^{-\alpha}\,\mathrm{d}y
		\leq C\abs{x}^{-\alpha}\int\abs{h(x-y)}\,\mathrm{d}y
		=C\abs{x}^{-\alpha}\norm{h}_{L^1}.
		\]
	\end{itemize}
	Combining the two pieces yields \eqref{eq:kernel-1} for $|x|>1$. This
	completes the proof of the positive direction.
	
	\noindent(2)	\textbf{Failure for $\alpha<0$}
	
	Let $\beta=-\alpha>0$, so $w(x)=|x|^{-\beta}$. Take
	\[
	f(x)=\abs{x}^{-\alpha}=\abs{x}^{\beta},
	\]
	which satisfies $\norm{f}_{\Lp[|x|^\alpha]{\infty}}=1$. Every
	$g\in L^\infty_{|x|^\alpha}$ obeys
	\[
	\abs{g(x)}\leq C\abs{x}^{-\alpha}=C\abs{x}^{\beta}\longrightarrow 0
	\quad\text{as }x\to 0,
	\]
	so necessarily $g(0)=0$ (for any continuous representative). However,
	\[
	(h*f)(0)=\int_{\R^n}h(-y)\,\abs{y}^{\beta}\,\mathrm{d}y\neq 0
	\]
	for a generic Schwartz function $h$ (for instance, a Gaussian
	$h(x)=e^{-\pi|x|^2}$ gives a strictly positive value). Thus
	$h*f\notin L^\infty_{|x|^\alpha}$, and $T$ does not even map the
	space into itself.
	
	\noindent (3)	\textbf{Failure for $\alpha\geq n$}
	
	Again $f(x)=|x|^{-\alpha}$ belongs to $L^\infty_{|x|^\alpha}$ with
	norm $1$, but $f\notin L^1_{\mathrm{loc}}(\mathbb{R}^n)$ because
	\[
	\int_{|y|<1}\abs{y}^{-\alpha}\,\mathrm{d}y=\infty
	\qquad(\alpha\geq n).
	\]
	For a generic Schwartz $h$ (nonzero on a set of positive measure),
	the convolution integral
	\[
	(h*f)(x)=\int_{\R^n} h(x-y)\,f(y)\,\mathrm{d}y
	\]
	diverges for every $x$. Hence $Tf$ is not defined as a measurable
	function.
\end{proof}
\subsubsection{Possion equation}
Let us consider the Poisson equation 
\begin{equation}\label{eq-Poisson}
 -\Delta P=\operatorname{div} {F}.
\end{equation}
The fundamental solution  \(\Phi(x) \)and its properties play a crucial role in the analysis
$$
	\Phi(x):= \begin{cases}-\frac{1}{2 \pi} \log |x| & (n=2) \\ \frac{1}{n(n-2) \alpha(n)} \frac{1}{|x|^{n-2}} & (n \geq 3)\end{cases}
	$$
which is	defined for $x \in \mathbb{R}^n, x \neq 0$, and where \(\alpha(n)=|B(0,1)|\) denotes the volume of the unit ball
\(B(0,1)\subset \mathbb{R}^n\). 
In terms of the fundamental solution \(\Phi\), the function \(P\) admits the following representation:
\begin{equation}\label{eq-re-P}
	P(x)=\int_{\mathbb{R}^n}\Phi(x-y)F(y)\,\mathrm{d}y .
\end{equation}
Under appropriate assumptions on the source term \(F\), we derive precise pointwise bounds for \(P\) and its derivatives, highlighting the dependence on the regularity and decay properties of \(F\).

	\begin{proposition}\label{prop-0902-P-D}
		Let $\alpha\in[0,n)$ and ${F}$ satisfies 
		\[\sup_{x\in\mathbb{R}^n}|x|^\alpha \big|{F}\big|(x)<\infty.\]
		  Assume that $P$ is the solution of \eqref{eq-Poisson}. 
		Then there exists a postive constant $C_{n,\alpha}$ such that 
		\begin{enumerate}
			\item[\rm (1)]  for each $\alpha\in(1,n)$,
			\begin{equation*}
				\sup_{x\in\mathbb{R}^n}|x|^{\alpha-1} \left|P\right|(x)\leq C_{n,\alpha} \sup_{x\in\mathbb{R}^n}|x|^\alpha \big|{F}\big|(x);
					\end{equation*}
			\item [\em (2)] for each $\alpha \in[0,n)$ and each $j\in \mathbb{Z},$
			\begin{equation}\label{eq-decayest-nabla-p-1}
				\begin{split}
					\sup_{x\in\mathbb{R}^n}|x|^\alpha\left|\dot{\Delta}_j\nabla P\right|(x)\leq C_{n,\alpha}\sup_{x\in\mathbb{R}^n}|x|^\alpha\left|\dot{\Delta}_jF\right|(x).
				\end{split}
			\end{equation}
		\end{enumerate}
		 Moreover, if $F=\operatorname{div} \mathbf{A}_{n \times n}$ with matrix $\mathbf{A}_{n \times n}$,  we have for each $\alpha \in[0,n)$ and each $j\in \mathbb{Z},$
		\begin{equation}\label{eq-decayest-nabla-p-2}
			\begin{split}
				\sup_{x\in\mathbb{R}^n}|x|^\alpha\left|\dot{\Delta}_j\nabla P\right|(x)\leq C_{n,\alpha}2^j\sup_{x\in\mathbb{R}^n}|x|^\alpha\left|\dot{\Delta}_j\mathbf{A}_{n \times n}\right|(x).
			\end{split}
		\end{equation}
	\end{proposition}
	\begin{proof}
		By the differentiation property of convolution, it follows from \eqref{eq-re-P} that
		\begin{equation*}
			\begin{aligned}
				P(x)
				&=\int_{\mathbb{R}^n}\Phi(x-y)\operatorname{div}F(y)\,\mathrm{d}y\\
				&=\int_{\mathbb{R}^n}(\nabla\Phi)(x-y)\cdot F(y)\,\mathrm{d}y .
			\end{aligned}
		\end{equation*}
		Moreover, for each \(i=1,\ldots,n\), differentiating the above expression yields
		\begin{equation*}
			\begin{aligned}
				\partial_{x_i}P(x)
				&=\sum_{j=1}^n
				\int_{\mathbb{R}^n}
				\frac{\partial^2\Phi}{\partial x_i\partial x_j}(x-y)
				F_j(y)\,\mathrm{d}y\\
				&=\sum_{j=1}^n
				\int_{\mathbb{R}^n}
				K_{ij}(x,y)F_j(y)\,\mathrm{d}y ,
			\end{aligned}
		\end{equation*}
		where
		\[
		K_{ij}(x,y)
		=
		\frac{1}{\alpha(n)}
		\frac{(x_i-y_i)(x_j-y_j)}
		{|x-y|^{n+2}}
		-\frac{1}{n\alpha(n)}
		\frac{\delta_{ij}}{|x-y|^n},
		\]
		and \(\delta_{ij}\) denotes the Kronecker delta.

 To derive pointwise estimates for \(P\), we split the integral representation into the near-field and far-field contributions. More precisely, we decompose \(P\) into
 two parts as follows:
 \begin{equation*}
 	P=\int_{ B(0,|x|/2)} (\nabla\Phi)(x-y) \cdot {F}(y) \,\mathrm{d} y+\int_{\mathbb{R}^n\backslash B(0,|x|/2)} (\nabla\Phi)(x-y) \cdot {F}(y) \,\mathrm{d} y=:P^\sharp+P^\natural.
 \end{equation*}
A straightforward computation shows that, for each \(\alpha\in(0,n)\),
 \begin{equation}\label{eq-possion-1}
 	\begin{aligned}
 		|x|^{\alpha-1} \left|P^\sharp\right|(x)\leq &C	|x|^{\alpha-1} \sup_{x\in\mathbb{R}^n}|x|^\alpha \big|{F}\big|(x)\int_{B(0,|x|/2)}\frac{1}{|x-y|^{n-1}}\frac{1}{|y|^{\alpha}}\,\mathrm{d}y\\
 		\leq &C	|x|^{\alpha-n} \sup_{x\in\mathbb{R}^n}|x|^\alpha \big|{F}\big|(x)\int_{B(0,|x|/2)}\frac{1}{|y|^{\alpha}} \,\mathrm{d}y\\
 		\leq &C_{n,\alpha} \sup_{x\in\mathbb{R}^n}|x|^\alpha \big|{F}\big|(x).
 	\end{aligned}
 \end{equation}
 For anthoer part,
 \begin{equation}\label{eq-possion-2}
 		|x|^{\alpha-1} \left|P^\natural\right|(x)\leq 	|x|^{\alpha-1} \sup_{x\in\mathbb{R}^n}|x|^\alpha \big|{F}\big|(x)\int_{\mathbb{R}^n\backslash B(0,|x|/2)}\frac{1}{|x-y|^{n-1}}\frac{1}{|y|^{\alpha}}\,\mathrm{d}y
 \end{equation}
 We calculate for each $\alpha >1,$
 \begin{align*}
 &	\int_{\mathbb{R}^n\backslash B(0,|x|/2)}\frac{1}{|x-y|^{n-1}}\frac{1}{|y|^{\alpha}}\,\mathrm{d}y\\
 	=&\int_{\mathbb{R}^n\backslash B(0,2|x| )}\frac{1}{|x-y|^{n-1}}\frac{1}{|y|^{\alpha}}\,\mathrm{d}y +\int_{C(0,|x|/2,2|x| )}\frac{1}{|x-y|^{n-1}}\frac{1}{|y|^{\alpha}}\,\mathrm{d}y\\
 	\leq &2^n\int_{\mathbb{R}^n\backslash B(0,2|x| )} \frac{1}{|y|^{n-1+\alpha}}\,\mathrm{d}y
 	+2^\alpha|x|^{-\alpha}\int_{C(0,|x|/2,2|x| )}\frac{1}{|x-y|^{n-1}}\,\mathrm{d}y 
 	\leq  C_{n,\alpha}|x|^{-\alpha+1}.
 \end{align*}
 Inserting this inequality  into \eqref{eq-possion-2} gives 
  \begin{equation}\label{eq-possion-3}
 	|x|^{\alpha-1} \left|P^\natural\right|(x)\leq C_{n,\alpha} \sup_{x\in\mathbb{R}^n}|x|^\alpha \big|{F}\big|(x).
 \end{equation}
 Collecting estimate \eqref{eq-possion-1} and estimate \eqref{eq-possion-3}, we readily have that for each $\alpha\in(1,n)$,
 \[\sup_{x\in\mathbb{R}^n}|x|^{\alpha-1} \left|P\right|(x)\leq C_{n,\alpha} \sup_{x\in\mathbb{R}^n}|x|^\alpha \big|{F}\big|(x).\]
 Now, we turn to establish the decay estimate for \(\nabla P\).
 Taking the Fourier transform of \eqref{eq-Poisson}, we obtain
 \begin{equation}\label{eq-Poisson-Fourier}
 	|\xi|^2\widehat{P}(\xi)
 	=
 	\widehat{\operatorname{div}F}(\xi).
 \end{equation}
Multiplying both sides by \(\psi(2^{-j}\xi)\) and using the almost orthogonality property, we obtain
 \[
 \begin{aligned}
 	\psi(2^{-j}\xi)\widehat P(\xi)
 	&=
 	\psi(2^{-j}\xi)
 	\frac{1}{|\xi|^2}
 	\widehat{\operatorname{div}F}(\xi)
 	\\
 	&=
 	\tilde\psi(2^{-j}\xi)
 	\frac{1}{|\xi|^2}
 	\widehat{\dot\Delta_j\operatorname{div}F}(\xi)
 	\\
 	&=
 	2^{-2j}
 	h(2^{-j}\xi)
 	\widehat{\dot\Delta_j\operatorname{div}F}(\xi),
 \end{aligned}
 \]
 where $h(\xi)=\tilde{\psi}(\xi)\frac{1}{|\xi|^2}$ with $\tilde{\psi}=\psi(\xi/2)+\psi(\xi)+\psi(2\xi)$.
 
 Moreover, we get by the inverse of Fourier transform that 
 \begin{equation}\label{eq-nabla-p-j}
 	\begin{aligned}
 	\dot{\Delta}_j\partial_{x_i}P=&2^{j(n-1)}(\partial_{x_i}\check{h})(2^j\cdot)\ast\operatorname{div} 	\dot{\Delta}_j{F}\\
 	 =&\sum_{k=1}^n2^{jn}\check{h}_{ik}(2^j\cdot)\ast  	\dot{\Delta}_j{F}_k,
 	\end{aligned}
 \end{equation}
 where $\check{h}_{ik}(x)=\partial^2_{x_ix_k}\check{h}(x).$
 
 According to the fundamental properties of the Fourier transform is that it is a linear isomorphism from the Schwartz space to itself $\mathscr{S}(\mathbb{R}^n),$
 it is easy to check that 
 \(\check{h}_{ik}(x)\in\mathscr{S}(\mathbb{R}^n).\) Moreover, we get by Lemma \ref{lem-point-LS} that  for each $\alpha \in[0,n)$ and each $j\in \mathbb{Z},$
 \begin{equation*}
 	\begin{split}
 		\sup_{x\in\mathbb{R}^n}|x|^\alpha\left|\dot{\Delta}_j\partial_{x_i}P\right|(x)\leq C_{n,\alpha}\sup_{x\in\mathbb{R}^n}|x|^\alpha\left|\Delta_j F\right|(x),
 	\end{split}
 \end{equation*}
 where we have used  the following two equalities
 \[\left\|2^{jn}\check{h}_{ik}(2^j\cdot)\right\|_{L^1(\mathbb{R}^n)}=\left\|\check{h}_{ik}\right\|_{L^1(\mathbb{R}^n)}\]
 and 
 \[\sup_{x\in\mathbb{R}^n}|x|^n2^{jn}\check{h}_{ik}(2^jx)=\sup_{x\in\mathbb{R}^n}|x|^n|\check{h}_{ik}|(x).\]
 Taking the supremum over all $i\in\{1,2,\ldots,n\}$, we can derive the second desired estimate  in the proposition.

 We decompose the general term in the right of \eqref{eq-nabla-p-j}  as follows:
 \begin{equation*}
 	\begin{aligned}
 		 2^{jn}\check{h}_{ik}(2^j\cdot)\ast  	\dot{\Delta}_j{F}_k
 		 =&\int_{B(0,|x|/2)}2^{jn}\check{h}_{ik}(2^j(x-y)) \dot{\Delta}_j{F}_k(y)\,\mathrm{d}y\\&+\int_{\mathbb{R}^n\backslash B(0,|x|/2) }2^{jn}\check{h}_{ik}(2^j(x-y)) \dot{\Delta}_j{F}_k(y)\,\mathrm{d}y\\
 		 =:&\partial_{x_i}P_k^\natural+\partial_{x_i}P_k^\sharp.
 	\end{aligned}
 \end{equation*}
 We see that for each $\alpha\geq0,$ 
 \begin{equation*}
 \begin{aligned}
 |x|^\alpha	\left|\partial_{x_i}P_k^\sharp\right|\leq&2^\alpha \int_{\mathbb{R}^n\backslash B(0,|x|/2) }2^{jn}|\check{h}_{ik}|(2^j(x-y)) |y|^\alpha|\dot{\Delta}_j{F}_k|(y)\,\mathrm{d}y\\
 \leq&2^\alpha \sup_{x\in\mathbb{R}^n}|x|^\alpha|\dot{\Delta}_j{F}_k|(x) \int_{\mathbb{R}^n }2^{jn}|\check{h}_{ik}|(2^j(x-y)) \,\mathrm{d}y\\
 \leq&C_{\alpha}   \sup_{x\in\mathbb{R}^n}|x|^\alpha|\dot{\Delta}_j{F}_k|(x).
 \end{aligned}
 \end{equation*}
 On the other hand, we arrive at that for each $\alpha\in(0,n),$ 
  \begin{equation*}
 	\begin{aligned}
 		 |x|^\alpha	\left|\partial_{x_i}P_k^\natural\right|\leq& |x|^\alpha  \sup_{x\in\mathbb{R}^n}|x|^\alpha|\dot{\Delta}_j{F}_k|(x)\int_{ B(0,|x|/2) }2^{jn}|\check{h}_{ik}|(2^j(x-y)) \frac{1}{|y|^\alpha} \,\mathrm{d}y\\\leq& |x|^\alpha \sup_{x\in\mathbb{R}^n}|x|^\alpha|\dot{\Delta}_j{F}_k|(x)\int_{ B(0,|x|/2) }2^{jn}|\check{h}_{ik}|(2^j(x-y)) \frac{1}{|y|^\alpha} \,\mathrm{d}y\\
 		  \leq&C_{\alpha}   \sup_{x\in\mathbb{R}^n}|x|^\alpha|\dot{\Delta}_j{F}_k|(x).
  \end{aligned}
\end{equation*}		
 		Since $F=\operatorname{div} \mathbf{A}_{n \times n}$, we get from \eqref{eq-nabla-p-j} that 
 		 \begin{equation*} 
 				\dot{\Delta}_j\partial_{x_i}P 
 				=\sum_{k=1}^n2^{j(n+1)}(\nabla\check{h}_{ik})(2^j\cdot)\ast  	\dot{\Delta}_j\mathbf{A}.
 		\end{equation*}
 	From  this, proceeding with the proof procedure for \eqref{eq-decayest-nabla-p-1},  we can  deduce  the desired estimate \eqref{eq-decayest-nabla-p-2} as a natural conclusion.
\end{proof}
\subsubsection{Heat equation}
We now extend our analysis to the nonhomogeneous heat equation
\[
\mu\partial_t u-\Delta u=f
\quad \text{in } \mathbb{R}^n\times(0,\infty),
\]
subject to the initial condition \(u(x,0)=0\). For the self-similar source term
\[
f(x,t)=\frac{1}{t^{\frac32}}F\left(\frac{x}{\sqrt{t}}\right),
\]
the representation formula \eqref{eq-heat-Evo-Stat} yields
\begin{equation}\label{eq-heat-Evo-Stat-1}
	U_{\rm Inhom}(x)=u(x,1)
	=\int_0^1\int_{\mathbb{R}^n}
	\Phi\big(x-y,s/\mu\big)
	\frac{1}{(1-s)^{\frac32}}
	F\left(\frac{x-y}{\sqrt{1-s}}\right)
	\,\mathrm{d}y\mathrm{d}s .
\end{equation}

Under suitable decay and regularity assumptions on \(F\), the function
\(U_{\mathrm{Inhom}}\) defined in \eqref{eq-heat-Evo-Stat-1} satisfies the following pointwise estimate:
\begin{proposition}\label{prop-heat-decay-point}
	Let $\alpha>1$ and $U$ defined in \eqref{eq-heat-Evo-Stat-1} satisfies 
	\[\sup_{x\in\mathbb{R}^n}|x|^\alpha|F|(x)<\infty.\]
	\begin{itemize}
		\item[(i)] 	If  $\alpha\in (1,n)$, there exists  a contant $C_{n,\alpha}$ such that 
		\begin{equation}\label{eq-heat-decay-point-1}
			\sup_{x\in\mathbb{R}^n}|x|^\alpha \big|U_{\rm Inhom} \big|(x)\leq C_{n,\alpha} \sup_{x\in\mathbb{R}^n}|x|^\alpha|F|(x).
		\end{equation}
		\item [(ii)]	If  $\alpha\in (n,\infty)$, there exist two cosntant $C_{n,\alpha}$ and  $C_{n,\alpha,\beta}$ such that  for each $\beta\in (1,n),$
		\begin{equation}\label{prop-heat-decay-point-2}
			\sup_{x\in\mathbb{R}^n}|x|^\alpha \big|U_{\rm Inhom} \big|(x)\leq C_{n,\alpha} \sup_{x\in\mathbb{R}^n}|x|^\alpha|F|(x)+C_{n,\alpha,\beta} \sup_{x\in\mathbb{R}^n}|x|^\beta|F|(x).
		\end{equation}
	\end{itemize}

\end{proposition}

\begin{proof}
	By the inequality \eqref{eq-prop-main-lambda} and $\alpha\in(1,n)$, we readily have 
	\begin{equation*}
		\begin{aligned}
			\sup_{x\in\mathbb{R}^n}	|x|^\alpha \big|U\big|(x)\leq &C_\alpha\sup_{x\in\mathbb{R}^n}|x|^\alpha|F|(x)\int_0^1  s^{\frac{\alpha}{2}-\frac{3}{2}}    \mathrm{d} s\\
			\leq& C_\alpha \sup_{x\in\mathbb{R}^n}|x|^\alpha|F|(x),
		\end{aligned}
	\end{equation*}
	which implies \eqref{eq-heat-decay-point-1}.
	
	We now address the case where $\alpha >n$. 	We decompose $U_{\rm Inhom}$ into  two parts as follows:
	\begin{align*}
		U_{\rm Inhom}(x)   =& \int_0^1   \int_{B(0,|x|/2)}  \Phi\big(x-y,s \big) \frac{1}{s^{\frac{3}{2}}} F\left(\frac{y}{\sqrt{s}}\right) \mathrm{d} y \mathrm{d} s\\
		&+\int_0^1   \int_{\mathbb{R}^n\backslash B(0,|x|/2)}  \Phi\big(x-y,s \big) \frac{1}{s^{\frac{3}{2}}} F\left(\frac{y}{\sqrt{s}}\right) \mathrm{d} y \mathrm{d} s\\
		=&U^\sharp(x)+U^\natural(x).
	\end{align*}
	From this decomposition, we obtain
	\begin{equation}\label{eq-hd-1}
		\begin{aligned}
			|x|^\alpha \big|U^\natural\big|(x)\leq &2^\alpha\int_0^1   \int_{\mathbb{R}^n\backslash B(0,|x|/2)}  \Phi\big(x-y,1-s \big) \frac{1}{s^{\frac{3}{2}}} |y|^{\alpha}|F|\left(\frac{y}{\sqrt{s}}\right) \mathrm{d} y \mathrm{d} s\\
			\leq &2^\alpha \sup_{x\in\mathbb{R}^n}|x|^\alpha|F|(x)\int_0^1  s^{\frac{\alpha}{2}-\frac{3}{2}}  \int_{\mathbb{R}^n}  \Phi\big(x-y,1-s \big)\mathrm{d} y \mathrm{d} s\\
			=&2^\alpha \sup_{x\in\mathbb{R}^n}|x|^\alpha|F|(x)\int_0^1  s^{\frac{\alpha}{2}-\frac{3}{2}}    \mathrm{d} s\leq C_\alpha \sup_{x\in\mathbb{R}^n}|x|^\alpha|F|(x),
		\end{aligned}
	\end{equation}
	in the last line, we have used the fact $\alpha>1.$
	
	By using that heat kernel is a monotonically decreasing function with respect to $x$, we can show that the estimate holds for all $\beta \in(1,n)$,
 	\begin{equation}\label{eq-hd-3}
 	\begin{aligned}
 		|x|^\alpha \big|U^\sharp\big|(x)\leq & |x|^\alpha\int_0^1  \int_{ B(0,|x|/2)}\frac{1}{|y|^\beta}  \Phi\big(x-y,1-s \big) \frac{1}{s^{\frac{3}{2}}} |y|^{\beta}|F|\left(\frac{y}{\sqrt{s}}\right) \mathrm{d} y \mathrm{d} s\\
 		\leq &|x|^\alpha \sup_{x\in\mathbb{R}^n}|x|^\beta|F|(x)\int_0^1s^{\frac{\alpha}{2}-\frac{3}{2}} \Phi\big(x/2,1-s \big) \int_{ B(0,|x|/2)}\frac{1}{|y|^\beta} \mathrm{d} y \mathrm{d} s\\
 		\leq  &C_{n,\alpha} \sup_{x\in\mathbb{R}^n}|x|^\beta|F|(x)\int_0^1s^{\frac{\alpha}{2}-\frac{3}{2}}\Phi\big(x/2,1-s \big) |x|^{n+\alpha-\beta} \mathrm{d} s \\
 		\leq &C_{n,\alpha,\beta} \sup_{x\in\mathbb{R}^n}|x|^\beta |F|(x)\sup_{x\in\mathbb{R}^n}|x|^{n+\alpha-\beta}\Phi(x,1)\int_0^1s^{\frac{\alpha}{2}-\frac{3}{2}} (1-s)^{\frac{\alpha-\beta}{2}} \mathrm{d} s\\
 		\leq& C_{n,\alpha,\beta} \sup_{x\in\mathbb{R}^n}|x|^\beta|F|(x).
 	\end{aligned}
 \end{equation}
 Combined  \eqref{eq-hd-3} with \eqref{eq-hd-1} leads to that for ecah $\beta \in(1,n)$,
 	\begin{equation*}
 	\sup_{x\in\mathbb{R}^n}|x|^\alpha \big|U_{\rm Inhom} \big|(x)\leq C_{n,\alpha} \sup_{x\in\mathbb{R}^n}|x|^\alpha|F|(x)+C_{n,\alpha,\beta} \sup_{x\in\mathbb{R}^n}|x|^\beta|F|(x).
 \end{equation*}
 Thus we complete the proof of Proposition \ref{prop-heat-decay-point}.
\end{proof}
If the force \(F\) in \eqref{eq-heat-Evo-Stat-1} is given by
\(\mathscr{P}\operatorname{div}\mathbf{G}\), then \eqref{eq-heat-Evo-Stat-1} becomes
\begin{equation}\label{eq-heat-Evo-Stat-1-G}
	\begin{aligned}
		U_{\rm Inhom}(x)
		&=
		\int_0^1\int_{\mathbb{R}^n}
		\Phi\big(x-y,s/\mu\big)
		\frac{1}{(1-s)^{\frac32}}
		(\mathscr P\operatorname{div}\mathbf G)
		\left(\frac{y}{\sqrt{1-s}}\right)
		\,\mathrm{d}y\,\mathrm{d}s .
	\end{aligned}
\end{equation}
The following proposition establishes the pointwise decay behavior of the
inhomogeneous solution \(U_{\rm Inhom}\) under suitable assumptions on
\(\mathbf{G}\).
 \begin{proposition}\label{prop-decay-divergence}
 	Let  $\alpha\in(0,n)$ and $\mathbf{G}$ satisfies 
 	\[\sup_{x\in\mathbb{R}^n} |x|^\alpha |\mathbf{G}|(x)<\infty.\]
 Suppose that $U$ is  defined in  \eqref{eq-heat-Evo-Stat-1-G}. Then we have 
 	 \begin{equation}
 			|x|^\alpha|	U_{\rm Inhom}|(x) 
 			\leq  C_{n,\alpha,\mu}\sup_{x\in\mathbb{R}^n} |x|^\alpha |\mathbf{G}|(x).
 	\end{equation}
 \end{proposition}
\begin{proof}
	According to the property of convolution, \eqref{eq-heat-Evo-Stat-1-G}  can be rewritten as 
	\begin{equation}\label{eq-heat-Evo-Stat-1-G_R}
		\begin{split}
			U_{\rm Inhom}(x)   =& \int_0^1   \int_{\mathbb{R}^n} \nabla  \left(\Phi \big(x-y,s/\mu\big)\right) \frac{1}{(1-s)} \cdot \mathscr{P} \mathbf{G}\left(\frac{y}{\sqrt{1-s}}\right) \mathrm{d} y \mathrm{d} s\\
			=&\sqrt{\mu} \int_0^1  \frac{1}{\sqrt{s}} \int_{\mathbb{R}^n} \mathscr{P}\left(\nabla  \Phi \right) \big(x-y,s/\mu\big)\frac{1}{(1-s)} \cdot \mathbf{G}\left(\frac{y}{\sqrt{1-s}}\right) \mathrm{d} y \mathrm{d} s.
		\end{split}
	\end{equation}
	Noting that 
	\begin{equation*}
		\left \|\mathscr{P}\left(\nabla  \Phi \right) \big(\cdot,s/\mu\big)\right\|_{L^1(\mathbb{R}^n)}=\|\mathscr{P}\nabla \Phi(\cdot,1)\|_{L^1(\mathbb{R}^n)}
	\end{equation*}
	and 
	\begin{equation*}
		\sup_{x\in\mathbb{R}^n}|x|^n|\mathscr{P}\left(\nabla  \Phi \right) \big(x,s/\mu\big)|\leq C_n	\sup_{x\in\mathbb{R}^n}|x|^n|\mathscr{P}\left(\nabla  \Phi \right) \big(x,1\big)|,
	\end{equation*}
	so we get by Lemma \ref{lem-point-LS} that for each $\alpha\in(0,n),$
	\begin{equation*}
		\begin{aligned}
			|x|^\alpha|	U_{\rm Inhom}|(x) \leq& C_{n,\alpha,\mu}\sup_{x\in\mathbb{R}^n} |x|^\alpha |\mathbf{G}|(x)\int_0^1  s^{-\frac12} (1-s)^{-1+\alpha/2}\mathrm{d}s\\
			\leq& C_{n,\alpha,\mu}\sup_{x\in\mathbb{R}^n} |x|^\alpha |\mathbf{G}|(x).
		\end{aligned}
	\end{equation*}
	Hence, we complete the proof of the proposition.
\end{proof}

\section{Leray System: Existence ($2<n<6$, $n=$ spatial dimension), Regularity and Decay for $n=4$} \label{S4}
\subsection{Existence}
Throughout this subsection, we assume, without loss of generality, that \(\nu=1\). Under this normalization, we establish the existence of weak solutions to the following Leray problem:
\begin{equation}\label{Leray-no-Scaling}
	\left.
	\begin{array}{rll}
		-\frac{1}{2}  {U}-\frac{1}{2} x \cdot\nabla  {U}-\Delta  {U}+ {U} \cdot\nabla  {U}+\nabla {P}=0\\
		\operatorname{div}  {U}=0
	\end{array}\right\}\quad\text{in}\quad \mathbb{R}^n
\end{equation}
with
$ n\in (2,6).$

As $U\notin L^2(\mathbb{R}^n)$, we accordingly consider the difference $V = U - U_{0,R}^L$ with $U_{0,R}^L = \Phi(x,1)\ast\left(\mathscr{P}(\phi_R^c u_0)\right)$ and $\phi_R^c $ defined in Lemma \ref{lem-HB-decom}, and study the corresponding perturbed system to prove the existence of weak solutions
\begin{equation}\label{PL}
	\left.
	\begin{aligned}
		-\Delta V- \frac{ 1}{2 }V-\frac{1}{2 }x\cdot \nabla V+V\cdot\nabla V+\nabla P=L(V)-U_{0,R}^L\cdot\nabla U_{0,R}^L+F\\
		\operatorname{div}\,V=0
	\end{aligned}\ \right\} , 
\end{equation}
where 
$L(V)=-U_{0,R}^L \cdot\nabla V-V\cdot\nabla U_{0,R}^L$
and $$
F=	\Delta U_{0,R}^L + \frac{ 1}{2 }U_{0,R}^L +\frac{1}{2 }x\cdot \nabla U_{0,R}^L.
$$
Now we aim to establish the existence of weak solutions to problem \eqref{PL} via the Galerkin approximation method.
Firstly, we denote
\begin{align*}
	C_{0,\sigma}^\infty (B (0,N))&= \bigl\{ \xi \in C_0^\infty\bigl(B(0,N); \mathbb{R}^n\bigr) \,\big|\, \nabla \cdot \xi = 0 \bigr\}, \\
	\mathbf{X} &= \dot{H}_{0,\sigma}^1(B (0,N)) = \overline{C_{0,\sigma}^\infty}^{\|\nabla\cdot\|_{L^2(B(0,N))}}
\end{align*}
for all $N>0$.
In
$
\mathbf{X}=\dot{H}_{0,\sigma}^1(B(0,N)),
$
we choose an orthonormal basis
\[
\{w_k\}\subset C_0^\infty(B(0,N)),
\]
such that the finite-dimensional subspace
\[
X_m=\operatorname{span}(w_1,\cdots,w_m)
\]
is dense in \(\mathbf{X}\). Here orthonormality means that
\[
(w_k,w_i)_{L^2}=\delta_{ki}
\]
for any \(k,i=1,2,\cdots,m\).

We now prove the existence of finite-dimensional approximate solutions. Namely, we seek
\[
V_m=\sum_{j=1}^m d_j^m w_j\in X_m
\]
satisfying the projected equations
\begin{equation*}
	\begin{split}
	&	-\frac{1}{2} (V_m,w_k)-\frac{1}{2}(x\cdot \nabla V_m, w_k)+(\nabla V_m,\nabla w_k)+((V_m\cdot\nabla)V_m,w_k)\\
		=&(L(V_m),w_k)+(f,w_k),
		\quad k=1,\cdots,m,
	\end{split}
\end{equation*}
where $f=-U_{0,R}^L\cdot\nabla U_{0,R}^L+F.$

We transform this problem into a system of algebraic equations. Define a mapping \(A\) on \(\mathbb{R}^m\) by
\[
A(d)_k
=
-\frac{1}{2} (v,w_k)-\frac{1}{2}(x\cdot \nabla v,w_k)+(\nabla v,\nabla w_k)+((v\cdot\nabla)v,w_k)
-(L(v),w_k)-(f,w_k)
\]
where
$\displaystyle 
v=\sum_{j=1}^m d_j^m w_j.$

Using $-\frac{1}{2} (v,v)-\frac{1}{2}(x\cdot \nabla v,v)=\frac{n-2}{4}\|v\|^2_{L^2(B(0,N))}$ and 
\[
((v\cdot\nabla)v,v)=((U_{0,R}^L\cdot\nabla)v,v)=0,
\]
we obtain
\[
 d\cdot  A( d)
=\frac{n-2}{4}\|v\|_{L^2(B(0,N))}^2+
\|\nabla v\|_{L^2(B(0,N))}^2+\left((v\cdot\nabla)U^L_{0,R},v\right)-(f,v).
\]
By the H\"older inequality, it follows that 
\[
(f,v)
\leq
\|f\|_{L^2(B(0,N))}\| v\|_{L^2(B(0,N))}.
\]
Since $u_0(x)=\frac{\sigma(x/|x|)}{|x|}$, we get by Lemma \ref{lem-HB-decom} that 
\begin{equation*}
	\begin{split}
		\left((v\cdot\nabla)U^L_{0,R},v\right)\leq &\|\nabla u^{L}_{0,R}\|_{L^\infty(\mathbb{R}^n)}\| v\|^2_{L^2(B(0,N))}\\
		\leq&\frac{C_0}{R}\|\sigma\|_{L^\infty(\mathbb{S}^{n-1})}\| v\|^2_{L^2(B(0,N))}
	\end{split}
\end{equation*}
Hence, we have that for $R>\frac{8}{(n-2)C_0\|\sigma\|_{L^\infty(\mathbb{S}^{n-1})}}$,
\begin{equation}\label{eq-exi-26-1}
	\begin{split}
	  d\cdot A( d)
		\geq&
		\frac{n-2}{8}\|v\|_{L^2}
		\left(
		\|v\|_{L^2(B(0,N))}-\frac{8}{n-2}\|f\|_{L^2(B(0,N))}
		\right).
	\end{split}
\end{equation}
 It remains to prove that $f\in L^2(B(0,N))$ for $n=3,4,5$. The function $f$ consists of two components, which we estimate separately. First, by the H\"older inequality and Lemma  \ref{lem-HB-decom} ,  we have that for each $n\in(2,6)$,
 \begin{equation}\label{eq-exi-26-2}
 \begin{split}
 	\left\|-U_{0,R}^L\cdot\nabla U_{0,R}^L\right\|_{L^2(B(0,N))}\leq &	\left\|U_{0,R}^L\right\|_{L^6(B_N(0))}	\left\|\nabla U_{0,R}^L\right\|_{L^3(B(0,N))}\\
 	\leq &C\|\sigma\|^2_{W^{1,\infty}(\mathbb{S}^{n-1})}\left\|\frac{1}{|\cdot|}\right\|^3_{L^6(\mathbb{R}^n\backslash B(0,R))}\\
 \leq	&C\|\sigma\|^2_{W^{1,\infty}(\mathbb{S}^{n-1})} \frac{1}{R^{3(6-n)}}.
 \end{split}
 \end{equation}
 As for $F$, we observe that  it involves
 \[\Delta U_{0}+ \frac{ 1}{2 }U_{0} +\frac{1}{2 }x\cdot \nabla U_{0}=0,\]
 where $ U_{0}=\big(\Phi(\cdot,1)\ast u_{0}\big)(x)$. This in turn implies that
 \[ F=-\left( \Delta U_{0,R}^S+ \frac{ 1}{2 }U_{0,R}^S +\frac{1}{2 }x\cdot \nabla U_{0,R}^S\right)\]
 with $U_{0,R}^S= \Phi(x,1)\ast\left(\mathscr{P}(\phi_R u_0)\right)$. Furthermore we have 
 \begin{equation}\label{eq-exi-26-3}
\|F\|_{L^2(B(0,N))}\leq C \|\sigma\|_{L^\infty(\mathbb{S}^{n-1})}\left(
  \left\|\frac{1}{|\cdot|}\right\|_{L^2(B (0,2R))}+R
  \left\|\frac{1}{|\cdot|}\right\|_{L^2(B (0,2R))}\right).
 \end{equation}
 Combining these two estimates \eqref{eq-exi-26-2} and \eqref{eq-exi-26-3}, we conclude that
 \begin{equation}\label{eq-exi-26-4}
 \|f\|_{L^2(\mathbb{R}^n)}\leq C(n,\|\sigma\|_{W^{1,\infty}(\mathbb{S}^{n-1})})<\infty\quad \text{for each }\,n=3,4,5.
 \end{equation}  
 Thus, for sufficiently large $R_m = \|d\|$, it follows from \eqref{eq-exi-26-1} that
\[
 d\cdot  A(\ d)>0.
\]
Before proceeding with the proof, we first recall Acute Angle Principle,   a consequence of the Brouwer fixed point theorem.
\begin{lemma}[Acute Angle Principle]
	Let
	\[
	A:\overline{B(0,R)}\subset\mathbb{R}^n\to\mathbb{R}^n,\,n>1
	\]
	be a continuous mapping. If
	\[
	x\cdot A(x)>0
	\]
	for every \(x\in\partial B(0,R)\), then there exists
	\(x_0\in B(0,R)\) such that
	\[
	A(x_0)=0.
	\]
\end{lemma}
 With this lemma in hand, we conclude that 
there exists \( d_0\in B (0,{R_m})\) such that
\[
A(d_0)=0.
\]
The corresponding function
\[
v_m=\sum_{j=1}^m d_{0j}^m w_j
\]
is therefore a finite-dimensional approximate solution.

Finally, we verify the a priori estimate and convergence. From
\[
 d_0\cdot A(d_0)=0,
\]
we get
\begin{equation} \label{eq-exi-26-5}
	\begin{split}
		&\frac{n-2}{4}\|v_m\|_{L^2(B(0,N))}^2+\|\nabla v_m\|_{L^2(B(0,N))}^2\\=&\langle f,v_m\rangle-	\left((v\cdot\nabla)U^L_{0,R},v\right)\\
		\leq&C_n\|f\|_{L^2(\mathbb{R}^n)}^2+	\frac{n-2}{16}\|v_m\|_{L^2(B(0,N))}^2+\frac{C_0}{R}\|\sigma\|_{L^\infty(\mathbb{S}^{n-1})}\| v\|^2_{L^2(B(0,N))}.
	\end{split}
\end{equation}
Hence,
\begin{equation}\label{eq-exi-26-6}
	\frac{n-2}{16}\|v_m\|_{L^2(B(0,N))}^2+\|\nabla v_m\|^2_{L^2(B(0,N)}
	\leq C_{n,R}
	\|f\|^2_{L^2(\mathbb{R}^n)},
\end{equation}
where the constant \(C_{n,R}\) is independent of \(m\) and $N$.

The uniform bound \eqref{eq-exi-26-6} allows us to extract a subsequence, still denoted by
\(v_{m}\), such that, as $m\to\infty,$
\[
v_{m}\rightharpoonup V_N
\]
weakly in \(H^1_{\mathrm{loc}}\), and
\[
v_{m}\to  V_N
\]
strongly in \(L_{\mathrm{loc}}^2\). 
For each fixed \(N\), we may pass to the limit \(m\to\infty\) in the projected equations. The strong convergence in \(L^2_{\mathrm{loc}}\) ensures the convergence of the nonlinear term.
\[
\int_{B(0,N)}(v_m\cdot\nabla)v_m\cdot w_k\,\mathrm{d} x
\to
\int_{B (0,N)} (V_N\cdot\nabla)V_N\cdot w_k\,\mathrm{d} x.
\]
By an argument analogous to the one above, and using the uniform bounds \eqref{eq-exi-26-6} derived previously, we fix any $R>R_0$ and conclude that there exists $V\in H^1(\mathbb{R}^n)$   such that $V_N\rightharpoonup V$ as $N\to\infty$, with the following convergence properties:
\[
\int_{B(0,N)} (V_N\cdot\nabla)V_N\cdot w_k\,\mathrm{d} x
\to
\int_{\mathbb{R}^n} (V
\cdot\nabla)V\cdot w_k\,\mathrm{d} x\quad \text{for each }\, k\in \mathbb{N}^+.
\]
Hence, the existence of a weak solution  $V\in H^1(\mathbb{R}^4)$  to problem \eqref{PL} for large $R$ follows.
\begin{theorem}\label{thm-W-mu}
	Let $n\in(2,6)$ and let $\sigma\in W^{1,\infty}(\mathbb{S}^{n-1})$.
	Then there exists $R_0>0$ such that for every $R>R_0$, problem \eqref{PL} admits at least one solution
	$V\in{H}^1(\mathbb{R}^n)$
	satisfying
	\begin{equation*}
		\|V\|_{H^1(\mathbb{R}^n)}\leq C_{n,R_0} \,\|f\|_{L^2(\mathbb{R}^n)}.
	\end{equation*}
\end{theorem}
Since $V = U - U_{0,R}^L$ with $U_{0,R}^L = \Phi(x,1)\ast\left(\mathscr{P}(\phi_R^c u_0)\right)$, Theorem \ref{thm-W-mu} implies   the existence of a weak solution to equation \eqref{Leray-no-Scaling}.
\begin{corollary}
	Let $n\in (2,6)$ and $u_0=\frac{\sigma(x/|x|)}{|x|}$ with $\sigma\in W^{1,\infty}(\mathbb{S}^{n-1})$. Then  Leray problem \eqref{Leray-no-Scaling} admits at least one solution $U$ satisfying $V=U - e^{\Delta}u_{0} \in{H}^1(\mathbb{R}^n)$ and
	\[
	\|V\|_{{H}^1(\mathbb{R}^n)}\leq C(n,\|\sigma\|_{W^{1,\infty}(\mathbb{S}^{n-1})} ).
	\]
\end{corollary}
With this corollary in hand, together with the relation 
\[
u(x,t)=\frac{1}{\sqrt{t}}U\left(\frac{x}{\sqrt{t}}\right),
\]
we eventually establish the desired result in Theorem \ref{thm1.1}. This completes the proof of Theorem \ref{thm1.1}.

\subsection{Regular}
To establish the regularity of weak solutions to problem \eqref{PL} in dimension $n=4$, we first need to prove a compactness lemma for the convective term.
\subsubsection{Compactness lemma}
\begin{lemma}[Compactness lemma]\label{lem-comp}
 Let the divergence free vector field $ {U} \in$ $\dot{H}^{1}\left(\mathbb{R}^4\right)$. Assume $V\in W^{2,p}\left(\mathbb{R}^4\right)$ with $2 \leq p<4$. Then, for any $\varepsilon>0$, there exists a constant $C>0$ depending only on $\varepsilon, p$ and $ {U}$ such that
$$
\|\operatorname{div}({U} \otimes {V})\|_{L^p\left(\mathbb{R}^4\right)} \leq \varepsilon\| {V}\|_{\dot{W} ^{2,p}\left(\mathbb{R}^4\right)}+C_{\varepsilon, p,  {U}}\| \nabla {V}\|_{L^p\left(\mathbb{R}^4\right)}.
$$
\end{lemma}
\begin{proof}
		We argue by contradiction: assume that to given $\varepsilon _0>0,$ 
	there exists, for every $k\in\mathbb{Z}^+$, $V_k\in W^{2,p}\left(\mathbb{R}^4\right)$ such that 
	\begin{equation}\label{eq-251005-1}
		\|\operatorname{div}({U} \otimes {V_k})\|_{L^p\left(\mathbb{R}^4\right)} \geq \varepsilon_0\| {V_k}\|_{\dot{W} ^{2,p}\left(\mathbb{R}^4\right)}+ k\|\nabla {V_k}\|_{L^p\left(\mathbb{R}^4\right)}.
	\end{equation}
	Without of loss generality,  we assume that 
	\[\|{V_k}\|_{\dot{W} ^{2,p}\left(\mathbb{R}^4\right)}=1.\]
	By the H\"older inequality and the Sobolev embedding theorem, one has
	\begin{align*}
		\|\operatorname{div}({U} \otimes {V}_k\big)\|_{L^p\left(\mathbb{R}^4\right)} =&\|{U} \cdot\nabla  {V}_k\|_{L^p\left(\mathbb{R}^4\right)} \\
		\leq &\|U\|_{L^4(\mathbb{R}^4)}\|\nabla  {V}_k\|_{L^{\frac{4p}{4-p}}\left(\mathbb{R}^4\right)} \\
		\leq &C\|U\|_{\dot{H}^1(\mathbb{R}^4)}\| V_k\|_{\dot{W}^{2,p}(\mathbb{R}^4)}\leq C\|U\|_{\dot{H}^1(\mathbb{R}^4)}.
	\end{align*}
	Inserting this estimate into \eqref{eq-251005-1} yields
	\[k\| \nabla {V_k}\|_{L^p\left(\mathbb{R}^4\right)}\leq C\|U\|_{\dot{H}^1(\mathbb{R}^4)}, \]
	which implies 
	\begin{equation}\label{eq-251006-1}
	\lim_{k\to\infty}	\|\nabla {V_k}\|_{L^p\left(\mathbb{R}^4\right)}=0.
	\end{equation}
The assumption $\|{V_k}\|_{\dot{W}  ^{2,p}\left(\mathbb{R}^4\right)}=1$ together with \eqref{eq-251006-1} enables us to conclude that  there exists a subsequence $\{V_{k_\ell}\}$ such that, for each $q\in [p,\frac{4p}{4-p}]$, $\nabla V_{k_\ell}$  converges weakly $0$ in the  sense of $L^{q}(\mathbb{R}^4)$.  
Moreover,  for each $r\in (1,\frac{4}{4-p}]$, a subsequence  $|\nabla V_{k_\ell}|^p$ converges weakly in  $L^{r}(\mathbb{R}^4)$ to zero.

Therefore, we have 
	\[	\|\operatorname{div}({U} \otimes {V}_{k_\ell}\big)\|^p_{L^p\left(\mathbb{R}^4\right)} =\int_{\mathbb{R}^4}|{U} \cdot\nabla  {V}_{k_\ell}|^p\,\mathrm{d}x \to 0\]
 as $\ell\to\infty.$ This convergence and 	\eqref{eq-251005-1}  allows us to infer that 
 \[\lim_{\ell\to\infty}\|{V_{k_\ell}}\|_{\dot{W} ^{2,p}\left(\mathbb{R}^4\right)}=0,\]
which contradicts $\|{V_{k_\ell}}\|_{\dot{W} ^{2,p}\left(\mathbb{R}^4\right)}=1.$ 
\end{proof}

\subsubsection{The basic $W^{2,q}$ type estimate } The above compactness lemma \ref{lem-comp}, combined with Theorem \ref{thm-Stokes},
yields the following \(W^{2,q}\) estimate for weak solutions to problem
\eqref{PL}, which is essential for the regularity analysis of the four-dimensional critical problem.
\begin{proposition}\label{prop-H2}
Let $\sigma\in W^{1,\infty}(\mathbb{S}^3)$, and let the couple $(V,P)\in H^1(\mathbb{R}^4)\times L^2(\mathbb{R}^4)$ be the weak solutions to problem \eqref{PL}. Then, for each $q\in[2,\infty)$,  the couple $(V,P)\in W^{2,q}(\mathbb{R}^4)\times W^{1,q}(\mathbb{R}^4)$ satisfies
\[\left \|V\right\|_{W^{2,q}(\mathbb{R}^4)}+\left\|\nabla P\right\|_{L^q(\mathbb{R}^4)}<\infty.\]
\end{proposition}
\begin{proof}  
We first introduce the following approximation scheme for problem \eqref{PL}:
\begin{equation}\label{PL-N-V}
	\left.
	\begin{aligned}
		-\Delta V^N_M- \frac{1}{2}V^N_M-\frac{1}{2}x\cdot \nabla V^N_M+\nabla P^N_M+ \dot{S}_N V \cdot\nabla V^N_M &= F^N \quad \\
		\operatorname{div}V^N_M &= 0 \
	\end{aligned}
	\right\}\quad \text{in } B(0,M)
\end{equation}
subject to the homogeneous Dirichlet boundary condition
\begin{equation}\label{PL-N-V-I}
V^N_M=0\quad \text{on}\,\,\,\partial B(0,M).
\end{equation}
Here, we define \(U_{0}:=U_{0,1}=e^{\Delta}u_{0}\), and the source term \(F^N\) is given by
\[
F^N=-\dot{S}_N\left(U_{0} \cdot\nabla V+ V\cdot\nabla U_{0}-U_{0}\cdot\nabla U_{0}\right).
\]
Our goal now is to establish the existence of weak solutions to problem \eqref{PL-N-V}-\eqref{PL-N-V-I} for all $N,\,M\in\mathbb{N}^+$ using Lax‑Milgram theorem. For convenience, we first give the definition of a weak solution.
\begin{definition}
	We say that $V^N_M\in H^1_0(B(0,M))$ is a weak solution to \eqref{PL-N-V}-\eqref{PL-N-V-I}  provided that for every solenoidal vector field $\phi\in H^1_0(B(0,M))$, there holds
	\begin{equation}\label{eq-bilinear-form}
		B\left[V^N_M,\phi\right]=\left(F_M^N,\phi\right),
	\end{equation}
	where 
	\begin{align*}
		B\left[V^N_M,\phi\right]=&\int_{B(0,M)}\nabla V^N_M\cdot\nabla\phi\,\mathrm{d}x-\frac12\int_{B(0,M)}  V^N_M\cdot\phi\,\mathrm{d}x-\frac12\int_{B(0,M)} (x\cdot\nabla ) V^N_M\cdot\phi\,\mathrm{d}x\\&+\int_{B(0,M)} (\dot{S}_N  V\cdot\nabla ) V^N_M  \cdot \phi\,\mathrm{d}x.
	\end{align*}
\end{definition}
We now proceed to prove the existence of weak solutions via the Lax-Milgram theorem. We first establish the boundedness of the bilinear operator $B[\cdot,\cdot]$.  For all $u,v\in H^1_0(B(0,M))$, it follows from the H\"older inequality and the Bernstein inequality  that
\begin{equation}\label{eq-bou}
	\begin{split}
		\left|B[u,v]\right|\leq &\|\nabla u\|_{L^2(B(0,M)}\|\nabla v\|_{L^2(B(0,M))}+\frac{ 1}{2 }\| u\|_{L^2(B(0,M))}\| v\|_{L^2(B(0,M))}\\&+\frac{M}{2}\|\nabla u\|_{L^2(B(0,M))}\| v\|_{L^2(B(0,M))}+C2^{2N}\|V\|_{L^2(\mathbb{R}^4)}\|\nabla u\|_{L^2(B(0,M))}\| v\|_{L^2(B(0,M))}\\
		\leq& C(N,M)\|  u\|_{H^1_0(B(0,M))}\| v\|_{H^1_0(B(0,M))}.
	\end{split}
\end{equation}
We next establish the coercivity of the bilinear form $B[\cdot,\cdot]$. Integrating by parts yields
\begin{equation}\label{eq-coe}
	\begin{split}
		B[u,u]=&\|\nabla u\|_{L^2(B(0,M))}^2+\frac{1}{2}\int_{B(0,M)} |u|^2\mathrm{d}x \\
		\geq&\frac{1}{2}\| u\|_{H^1_0(B(0,M))}^2.
	\end{split}
\end{equation}
Finally, since $V\in H^1(\mathbb R^4)$, we immediately obtain
\[
\left\| F^N\right\|_{L^2(\mathbb R^4)}<\infty.
\]
With this together with \eqref{eq-bou}-\eqref{eq-coe}, the Lax--Milgram theorem yields the existence and uniqueness of a weak solution $V^N_M\in H^1_0(B(0,M))$ to problem \eqref{PL-N-V}-\eqref{PL-N-V-I}. In particular, taking $\phi=V_M^N\in H^1_0(B(0,M))$ in \eqref{eq-bilinear-form}, we deduce from \eqref{eq-coe} that
\begin{align*}
\left	\|\nabla V^N_M\right\|_{L^2(B(0,M))}^2+\frac{1}{2}\int_{ B(0,M)}   |V^N_M|^2\,\mathrm{d}x \leq &\left\| F^N\right\|_{L^2(\mathbb{R}^4)}\left\|  V_M^N\right\|_{L^2(B(0,M)}\\
\leq &\left\| F^N\right\|^2_{L^2(\mathbb{R}^4)}+
\frac{1}{4}\left\|  V_M^N\right\|^2_{L^2(B(0,M))},
\end{align*}
which implies 
\begin{equation}\label{eq-uniform-M}
	\left	\|\nabla V^N_M\right\|_{L^2(B(0,M))}^2+\frac{1}{4}\int_{ B(0,M)}|V^N_M|^2\,\mathrm{d}x\leq \left\| F^N\right\|^2_{L^2(\mathbb{R}^4)}.
\end{equation}
Furthermore, the Eberlein--\v{S}mulian theorem guarantees the existence of a subsequence, still denoted by $\{V^N_M\}$, such that $V^N_M\rightharpoonup V^N$ weakly in $H^1(\mathbb R^4)$ as $M\to\infty$. The limit function $V^N$ then solves
\begin{equation}\label{PL-N}
\left.
\begin{aligned}
 -\Delta V^N - \frac{ 1}{2 }V^N -\frac{1}{2 }x\cdot \nabla  V^N +\nabla  P^N  +  ( \dot{S}_N V \cdot\nabla )V^N  = F ^N \\
 \operatorname{div}V^N =0
\end{aligned}\ \right\} \quad\text{in}\,\,\, \mathbb{R}^4.
\end{equation}
More precisely,  $V^N\in H^1(\mathbb R^4)$ is a weak solution of \eqref{PL-N}  for each $N\in\mathbb{N}^+$, with its weak formulation stated below.
\begin{definition}[Weak solution]
We call that $V^N\in H^1(\mathbb{R}^4)$ is the weak solution to \eqref{PL-N} if 
for all solenoidal vector field $\varphi\in \mathscr{S}(\mathbb{R}^4)$ there holds
\begin{align*}
&\int_{\mathbb{R}^4}\nabla V^N \cdot\nabla\varphi\,\mathrm{d}x-\frac12\int_{\mathbb{R}^4} V^N \cdot\varphi\,\mathrm{d}x-\frac12\int_{\mathbb{R}^4}(x\cdot\nabla ) V^N \cdot\varphi\,\mathrm{d}x\\=&-\int_{\mathbb{R}^4} ( \dot{S}_N V \cdot\nabla \big)  V^N\cdot \varphi\,\mathrm{d}x+\int_{\mathbb{R}^4} F^N \cdot\varphi\,\mathrm{d}x.
\end{align*}
\end{definition}
Moreover, applying Fatou's lemma to the uniform estimate \eqref{eq-uniform-M}, we obtain
\begin{equation*}
	\|\nabla V^N\|_{L^2(\mathbb R^4)}^2+\frac14\int_{\mathbb R^4}|V^N|^2\,\mathrm{d}x
	\leq \|F^N\|_{L^2(\mathbb R^4)}^2.
\end{equation*}
We now investigate the regularity of the weak solution $V^N\in H^1(\mathbb{R}^4)$ to equations  \eqref{PL-N} for sufficiently large $N\in\mathbb{N}^+$.
Since $V\in H^1(\mathbb R^4)$, one readily sees that $F^N\in H^\infty(\mathbb R^4)$. On the other hand, again using $V^N\in H^1(\mathbb R^4)$, we deduce from H\"older's inequality and the Bernstein inequality  that
\[
 \left\|( \dot{S}_N V \cdot\nabla \big)  V^N\right\|_{L^2(\mathbb{R}^4)}\leq C2^N\|V\|_{L^2(\mathbb{R}^4)}\left\|V^N\right\|_{L^2(\mathbb{R}^4)}.
\]
Applying    Theorem \ref{thm-Stokes}, we further obtain the following estimate
\begin{equation}\label{eq-H2-1}
	\left \|V^N\right\|_{H^2(\mathbb{R}^4)}+\left\|\nabla P^N\right\|_{L^2(\mathbb{R}^4)}\leq C\left\|F^N\right\|_{L^2(\mathbb{R}^4)}+C\left\|(\dot{S}_N V)\cdot\nabla V^N\right\|_{L^2(\mathbb{R}^4)}.
\end{equation}
By Lemma \ref{lem-comp}, we derive the estimate
\begin{equation}\label{eq-H2-2}
	\left\| V\cdot\nabla V^N\right\|_{L^2(\mathbb{R}^4)}\leq \frac18\left\|V^N\right\|_{\dot{H}^2(\mathbb{R}^4)}+C\left\|V^N\right\|_{L^2(\mathbb{R}^4)}.
\end{equation}
On the other hand, it follows from the Hölder inequality that
\begin{align*} 
	\begin{split}
		\left\|(P_{\geq N} V)\cdot\nabla V^N\right\|_{L^2(\mathbb{R}^4)}\leq & \left\|P_{\geq N}V\right\|_{ L^4(\mathbb{R}^4)} \left\|\nabla V^N\right\|_{L^4(\mathbb{R}^4)}\\
		\leq &C\left\|P_{\geq N}\nabla V\right\|_{ L^2(\mathbb{R}^4)} \left\|V^N\right\|_{\dot{H}^2(\mathbb{R}^4)}.	
	\end{split}
\end{align*}

Since $V\in H^1(\mathbb{R}^4)$, we have 
\[\lim _{N\to \infty}\left\|P_{\geq N}\nabla V\right\|_{ L^2(\mathbb{R}^4)}=0.\]
Consequently, there exists $N_0\in\mathbb{N}^+$ such that
for all $N\geq N_0$
\begin{equation}\label{eq-H2-2-2}
		\left\|(P_{\geq N} V)\cdot\nabla V^N\right\|_{L^2(\mathbb{R}^4)}\leq \frac{1}{8} \left\|V^N\right\|_{\dot{H}^2(\mathbb{R}^4)}.	
\end{equation}
Combining the estimates \eqref{eq-H2-2} and \eqref{eq-H2-2-2}, and applying the
triangle inequality, we obtain that for every \(N\geq N_0\),
\begin{equation}\label{eq-H2-2-3}
	\begin{split}
			\left\|(\dot{S}_N V)\cdot\nabla V^N\right\|_{L^2(\mathbb{R}^4)}\leq&	\left\| V\cdot\nabla V^N\right\|_{L^2(\mathbb{R}^4)}+	\left\|(P_{\geq N} V)\cdot\nabla V^N\right\|_{L^2(\mathbb{R}^4)}\\
			\leq& \frac14\left\|V^N\right\|_{\dot{H}^2(\mathbb{R}^4)}+C\left\|V^N\right\|_{L^2(\mathbb{R}^4)}.
	\end{split}
\end{equation}
In terms of the H\"older inequality and the embedding relation that $\dot{H}^1(\mathbb{R}^4)\hookrightarrow L^4(\mathbb{R}^4)$, we can infer that  for each $N\geq N_0,$
\begin{equation}\label{eq-H2-3}
	\begin{split}
		\left\|F^N\right\|_{L^2(\mathbb{R}^4)}\leq C\left\|F\right\|_{L^2(\mathbb{R}^4)}\leq& C\left\|U_0\right\|_{L^\infty(\mathbb{R}^4)}\left\|\nabla V\right\|_{L^2(\mathbb{R}^4)}+C\left\|\nabla U_0\right\|_{L^4(\mathbb{R}^4)}\left\|\nabla V\right\|_{L^2(\mathbb{R}^4)}\\
		&+C\left\| U_0\right\|_{L^5(\mathbb{R}^4)}\left\|\nabla U_0\right\|_{L^{\frac{10}{3}}(\mathbb{R}^4)},
	\end{split}
\end{equation}
where 
\[
F=-\left(U_{0} \cdot\nabla V+ V\cdot\nabla U_{0}-U_{0}\cdot\nabla U_{0}\right).
\]
Inserting \eqref{eq-H2-2-3} and \eqref{eq-H2-3} into \eqref{eq-H2-1} , we deduce that 
there exits a costnat $C$ independent of $N$ such that   for all $N\geq N_0,$ \begin{equation}\label{eq-H2-4-M-N}
	\left \|V^N\right\|_{H^2(\mathbb{R}^4)}+\left\|\nabla P^N\right\|_{L^2(\mathbb{R}^4)}\leq C(U_0).
\end{equation}
Let $\delta V=V-V^N$ and $\delta P=P-P^N$. One can readily verify that the pair $(\delta V,\delta P)$ satisfies the following system in the weak sense:
\begin{equation*}
	\left.
	\begin{aligned}
		-\Delta \delta {V} - \frac{1}{2}\delta{V} -\frac{1}{2}x\cdot \nabla \widetilde{V} +\nabla \delta{P} + (\dot{S}_NV \cdot\nabla) \delta{V} &= P_{\geq N}F-\big((P_{\geq N}V\big)\cdot\nabla) \delta V\\
		\operatorname{div}\delta{V} &= 0
	\end{aligned}\ \right\} \quad\text{in } \mathbb{R}^4.
\end{equation*}
Applying Theorem \ref{thm-Stokes}, we further derive the estimate
\begin{equation}\label{eq-H2-1-delta}
	\left\|\delta V\right\|_{H^2(\mathbb{R}^4)}+\left\|\nabla \delta P\right\|_{L^2(\mathbb{R}^4)}\leq C\left\|P_{\geq N}F\right\|_{L^2(\mathbb{R}^4)}+C\left\|\big((P_{\geq N}V\big)\cdot\nabla) \delta V\right\|_{L^2(\mathbb{R}^4)}.
\end{equation}
It follows from \eqref{eq-H2-3} that $\left\|F\right\|_{L^2(\mathbb{R}^4)}\leq C(U_0)$, which yields
\begin{equation}\label{eq-delta}
	\lim_{N\to \infty}\left\|P_{\geq N}F\right\|_{L^2(\mathbb{R}^4)}=0.
\end{equation}
Arguing similarly as for \eqref{eq-H2-2-2}, we conclude that for all $N\geq N_0$,
\begin{equation}\label{eq-H2-2-2-delta}
	\left\|(P_{\geq N} V)\cdot\nabla V^N\right\|_{L^2(\mathbb{R}^4)}\leq \frac{1}{8} \left\|V^N\right\|_{\dot{H}^2(\mathbb{R}^4)}.	
\end{equation}
Substituting \eqref{eq-delta} and \eqref{eq-H2-2-2-delta} into \eqref{eq-H2-1-delta} and letting $N\to\infty$, we obtain
\begin{equation*}
	\lim_{N\to \infty}\left(\left\|\delta V\right\|_{H^2(\mathbb{R}^4)}+\left\|\nabla \delta P\right\|_{L^2(\mathbb{R}^4)}\right)=0,
\end{equation*}
which shows that $V^N\to V$ in $H^2(\mathbb{R}^4)$.

Furthermore, passing to the limit $N\to\infty$ in \eqref{eq-H2-4-M-N} gives
\begin{equation}\label{eq-H2-u-4}
	\left\|V\right\|_{H^2(\mathbb{R}^4)}+\left\|\nabla P\right\|_{L^2(\mathbb{R}^4)}\leq C(U_0)<\infty.
\end{equation}
When $q\in (2,4)$,  repeating the process as used in the above, we finally obtain by Theorem \ref{thm-Stokes} with $q\in(2,4)$ that  for each $N\geq N_0,$
\begin{equation*}\label{eq-W2q-1}
	\left \|V^N\right\|_{W^{2,q}(\mathbb{R}^4)}+\left\|\nabla P^N\right\|_{L^q(\mathbb{R}^4)}\leq C\left\|F^N\right\|_{L^q(\mathbb{R}^4)}+C\left\|\dot{S}_NV\cdot\nabla V^N\right\|_{L^q(\mathbb{R}^4)}.
\end{equation*}
We applay Lemma \ref{lem-comp} to  the second term in the right of  the above inequality, which combined with the embedding below  
\[H^2(\mathbb{R}^4)\hookrightarrow L^r(\mathbb{R}^4)\quad \text{for each }  r\in[2,\infty)\]
yields
\begin{equation}\label{eq-W2q-2}
	 \left\|\dot{S}_NV\cdot\nabla V^N\right\|_{L^q(\mathbb{R}^4)} \leq \left \|\nabla V^N\right\|_{L^{4}(\mathbb{R}^4)} \left\|V\right\|_{L^{\frac{4q}{4-q}}(\mathbb{R}^4)}\leq C\left \|V^N\right\|_{H^{2}(\mathbb{R}^4)} \left\|V\right\|_{H^{2}(\mathbb{R}^4)}.
\end{equation}
On the other hand, we infer by the H\"older inequality that 
\begin{equation}\label{eq-W2q-3}
	\begin{split}
		\left\|F^N\right\|_{L^q(\mathbb{R}^4)}\leq& C\left\|U_0\right\|_{L^\infty(\mathbb{R}^4)}\left\|\nabla V\right\|_{L^q(\mathbb{R}^4)}+C\left\|\nabla U_0\right\|_{L^{\infty}(\mathbb{R}^4)}\left\|  V\right\|_{L^{q}(\mathbb{R}^4)}\\
		&+C\left\| U_0\right\|_{L^{2q}(\mathbb{R}^4)}\left\|\nabla U_0\right\|_{L^{2q}(\mathbb{R}^4)}.
	\end{split}
\end{equation} 
Inserting   \eqref{eq-W2q-2}  and \eqref{eq-W2q-3} into \eqref{eq-W2q-1}  leads to that for each $q\in(2,4),$
\begin{equation}\label{eq-W2q-4}
	\begin{aligned}
		\left \|V^N\right\|_{W^{2,q}(\mathbb{R}^4)}+\left\|\nabla P^N\right\|_{L^q(\mathbb{R}^4)}\leq& C\left\|U_0\right\|_{L^\infty(\mathbb{R}^4)}\left\|\nabla V\right\|_{L^q(\mathbb{R}^4)}+C\left\|\nabla U_0\right\|_{L^{\infty}(\mathbb{R}^4)}\left\|  V\right\|_{L^{q}(\mathbb{R}^4)}\\
		&+C\left\| U_0\right\|_{L^{2q}(\mathbb{R}^4)}\left\|\nabla U_0\right\|_{L^{2q}(\mathbb{R}^4)}+C\|V\|_{H^2(\mathbb{R}^4)}\left\|V^N\right\|_{H^2(\mathbb{R}^4)}.
	\end{aligned}
\end{equation}
Thus, combining \eqref{eq-H2-4-M-N} with \eqref{eq-H2-u-4}, we obtain that for all $q\in[2,4)$ and $N\geq N_0$,
\begin{equation}\label{eq-W2q-5-M}
	\left \|V^N\right\|_{W^{2,q}(\mathbb{R}^4)}+\left\|\nabla P^N\right\|_{L^q(\mathbb{R}^4)}\leq C(U_0)<\infty.
\end{equation}
Letting $N\to\infty$, we conclude that for all $q\in[2,4)$,
\begin{equation}\label{eq-W2q-5}
	\left \|V\right\|_{W^{2,q}(\mathbb{R}^4)}+\left\|\nabla P\right\|_{L^q(\mathbb{R}^4)}\leq C(U_0)<\infty.
\end{equation}
Lastly, we need to deal with the case where $q\in[4,\infty)$. 
Repeating the process as used in the above, we finally obtain by Theorem \ref{thm-Stokes} with $q\in[4,\infty)$ that for each $N\geq N_0,$
\begin{equation}\label{eq-W2q-1-6}
	\left \|V^N\right\|_{W^{2,q}(\mathbb{R}^4)}+\left\|\nabla P^N\right\|_{L^q(\mathbb{R}^4)}\leq C\left\|F^N\right\|_{L^q(\mathbb{R}^4)}+C\left\|\dot{S}_NV\cdot\nabla V^N\right\|_{L^q(\mathbb{R}^4)}+C\|V^N\|_{L^q(\mathbb{R}^4)}.
\end{equation}
 By the H\"older inequality and the embedding inclusion that there exists $p_0=\frac{4q}{4+q}\in(2,4)$  such that $W^{2,p_0}(\mathbb{R}^4)\hookrightarrow B^0_{\infty,1}(\mathbb{R}^4)\cap W^{1,q}(\mathbb{R}^4)$, we can show that 
\begin{equation}\label{eq-W2q-2-2}
		C\left\| \dot{S}_NV\cdot\nabla V^N\right\|_{L^q(\mathbb{R}^4)}\leq C\left \|\nabla V^N\right\|_{ L^q(\mathbb{R}^4)}\|V\|_{L^\infty(\mathbb{R}^4)}\\
		\leq C\|V\|_{W^{2,p_0}(\mathbb{R}^4)}\left\|V^N\right\|_{W^{2,p_0}(\mathbb{R}^4)}.
\end{equation}
Inserting   \eqref{eq-W2q-3}  and \eqref{eq-W2q-2-2}   into \eqref{eq-W2q-1-6}  leads to that for all $q\in(4,\infty) $ and $N\geq N_0,$
\begin{equation}\label{eq-W2q-6}
	\begin{aligned}
	&	\left \|V^N\right\|_{W^{2,q}(\mathbb{R}^4)}+\left\|\nabla P^N\right\|_{L^q(\mathbb{R}^4)}\\
	\leq& C\left\|U_0\right\|_{L^\infty(\mathbb{R}^4)}\left\|\nabla V\right\|_{L^q(\mathbb{R}^4)}+C\left\|\nabla U_0\right\|_{L^{\infty}(\mathbb{R}^4)}\left\|  V\right\|_{L^{q}(\mathbb{R}^4)}\\
		&+C\left\| U_0\right\|_{L^{2q}(\mathbb{R}^4)}\left\|\nabla U_0\right\|_{L^{2q}(\mathbb{R}^4)}+C\|V\|_{W^{2,p_0}(\mathbb{R}^4)}\left\|V^N\right\|_{W^{2,p_0}(\mathbb{R}^4)}+C\left\|V^N\right\|_{H^2(\mathbb{R}^4)}.
	\end{aligned}
\end{equation}
Therefore, by taking  $N\to\infty$ in \eqref{eq-W2q-6}, we eventually get by performing  \eqref{eq-H2-u-4}, \eqref{eq-W2q-5-M} and \eqref{eq-W2q-5}  that for each $q\in[2,\infty),$
\begin{equation*}
	\left \|V\right\|_{W^{2,q}(\mathbb{R}^4)}+\left\|\nabla P\right\|_{L^q(\mathbb{R}^4)}<\infty,
\end{equation*}
which is the required estimate in Proposition \ref{prop-H2}.
\end{proof}

\subsection{Decay estimate }
In this subsection, we mainly establish decay estimates for solutions. We first
present weighted \(L^p\) estimates, which will be used to characterize the decay
behavior of solutions in the weighted \(L^p\) framework.

\subsubsection{Estimate in the weighted space}
In this subsubsection, we study decay estimates for solutions. We begin with the
following weighted \(L^p\) estimates.
\begin{proposition}\label{prop-1p-weighted}
	Let  $\sigma\in W^{1,\infty}(\mathbb{S}^3)$, and let the couple $(V,P)\in H^1(\mathbb{R}^4)\times L^2(\mathbb{R}^4)$ be the weak solutions to problem \eqref{PL}. Then, for each $p\in(2,\infty)$, the couple $(V,P)\in W^{2,p}(\mathbb{R}^4)\times W^{1,p}(\mathbb{R}^4)$ satisfies that 
	\[\left \|V\right\|_{L_{|x|}^p(\mathbb{R}^4)}+\left\|\nabla V\right\|_{L_{|x|}^p(\mathbb{R}^4)}<\infty.\]
\end{proposition}
\begin{proof}
	With the regularity result \(V\in W^{2,q}(\mathbb{R}^4)\) for all \(q\in[2,\infty)\), established in Proposition \ref{prop-H2}, we now proceed to derive the decay estimate. To this end, we construct an approximate sequence \(\{V^N\}\) for \(V\), where \(V^N\) satisfies the system
	\begin{equation}\label{PL-NNNN}
		\left.
		\begin{aligned}
			-\Delta V^N - \frac{1}{2}V^N -\frac{1}{2}x\cdot \nabla V^N
			+\nabla P^N
			+\big(\mathscr{P}(\varphi_NV)\cdot\nabla\big)V^N
			&= \varphi_NF\\
			\operatorname{div}V^N&=0
		\end{aligned}
		\right\}
		\quad \text{in } \mathbb{R}^4,
	\end{equation}
	where \(\varphi_N(x)=\varphi(x/N)\), and the cutoff function \(\varphi\) is defined in \eqref{eq-cutoff-smooth}. We first establish the existence and regularity of solutions to \eqref{PL-NNNN} by adapting the argument used in the proof of Proposition \ref{prop-H2}. More precisely, for each \(N\), we obtain a solution \(V^N\in H^1(\mathbb{R}^4)\) as the limit of a sequence \(\{V_M^N\}_{M\geq 1}\subset H^1(\mathbb{R}^4)\), where \(V_M^N\) solves
	\begin{equation}\label{PL-N-V-MM}
		\left.
		\begin{aligned}
			-\Delta V^N_M-\frac{1}{2}V^N_M-\frac{1}{2}x\cdot\nabla V^N_M
			+\nabla P^N_M
			+\big(\dot{S}_M\mathscr{P}(\varphi_NV)\big)\cdot\nabla V^N_M
			&=\dot{S}_M(\varphi_NF)\\
			\operatorname{div}V^N_M&=0
		\end{aligned}
		\right\}
		\quad \text{in }B(0,M),
	\end{equation}
	subject to homogeneous Dirichlet boundary conditions on \(\partial B(0,M)\). The existence of weak solutions \(V_M^N\) to \eqref{PL-N-V-MM} follows directly from the Lax--Milgram theorem. By adapting the regularity argument in the proof of Proposition \ref{prop-H2}, we obtain, for every \(q\in[2,\infty)\),
	\begin{equation}\label{eq-W2q-6-0831}
		\left\|V^N\right\|_{W^{2,q}(\mathbb{R}^4)}
		+\left\|\nabla P^N\right\|_{L^q(\mathbb{R}^4)}
		<\infty.
	\end{equation}
	More importantly, we will show that
	\[
	V^N\longrightarrow V
	\qquad\text{strongly in }H^2(\mathbb{R}^4)
	\quad\text{as }N\to\infty.
	\]
	By applying Remark \ref{rem-heat-weighted} to \(V^N\), we obtain that for every $p\in[2,4), $
	\begin{equation}\label{eq-weighted-0832-1}
		\begin{split}
	\left	\|V^N\right\|_{L^p_{\langle x\rangle^\alpha}(\mathbb{R}^4)}+	\left\|\nabla V^N\right\|_{L^p_{\langle x\rangle^\alpha}(\mathbb{R}^4)}
			\leq &C	\left\|(\mathscr{P}(\varphi_NV))\cdot\nabla V^N\right\|_{L^p_{\langle x\rangle^\alpha}(\mathbb{R}^4)}+C\left\|\varphi_NF\right\|_{L^p_{\langle x\rangle^\alpha}(\mathbb{R}^4)}\\
			\leq &C\left( \| \varphi_NV\|_{L^p_{\langle x\rangle^\alpha}(\mathbb{R}^4)}\| \nabla V^N\|_{L^\infty(\mathbb{R}^4)}+ \|\varphi_NF\|_{L_{\langle x\rangle^\alpha}^p(\mathbb{R}^4)}\right)\\
				\leq &CN^\alpha\left( \| V\|_{L^p(\mathbb{R}^4)}\| \nabla V^N\|_{L^\infty(\mathbb{R}^4)}+\|F\|_{L^p_{\langle x\rangle^\alpha}(\mathbb{R}^4)}\right),
			\end{split}
	\end{equation}
	in the last line of \eqref{eq-weighted-0832-1},  we have used the fact derived by  Theorem \ref {thm-weight-boundness} that  for each $p\in[2,4),$
	\[\| \mathscr{P}(\varphi_NV)\|_{L^p_{\langle x\rangle^\alpha}(\mathbb{R}^4)}\leq C \| \varphi_NV\|_{L^p_{\langle x\rangle^\alpha}(\mathbb{R}^4)}.\]
Combining	\eqref{eq-weighted-0832-1} with  \eqref{eq-W2q-6-0831} and Proposition \ref {prop-H2} yields  that for $p\in[2,4)$ and $\alpha\in[0,4(p-1))$,
	\begin{equation}\label{eq-claim-0831-2}
	\left\|V^N\right\|_{L^p_{\langle x\rangle^\alpha}(\mathbb{R}^4)}+	\left\|\nabla V^N\right\|_{L^p_{\langle x\rangle^\alpha}(\mathbb{R}^4)}\leq C(N).
\end{equation}
For \(p\in[4,\infty)\), we apply Remark \ref{rem-heat-weighted}  to $V^N$ once again to obtain
\begin{equation}\label{eq-weighted-0832-3}
	\begin{split}
	&\left	\|V^N\right\|_{L^p_{\langle x\rangle^\alpha}(\mathbb{R}^4)}+	\left\|\nabla V^N\right\|_{L^p_{\langle x\rangle^\alpha}(\mathbb{R}^4)}\\
		\leq &C	\left\|\mathscr{P}(\varphi_NV)\cdot\nabla V^N\right\|_{L^p_{\langle x\rangle^\alpha}(\mathbb{R}^4)}+C\left\|\varphi_NF \right\|_{L^p_{\langle x\rangle^\alpha}(\mathbb{R}^4)}+C	\|V^N\|_{L^p_{\langle x\rangle^\alpha}(\mathbb{R}^4)}\\
		\leq &CN^\alpha\left( \| V\|_{L^p(\mathbb{R}^4)}\| \nabla V^N\|_{L^\infty(\mathbb{R}^4)}+\|F\|_{L^p(\mathbb{R}^4)}\right)+C\|V^N\|_{L^p_{\langle x\rangle^\alpha}(\mathbb{R}^4)}.
	\end{split}
\end{equation}
By virtue of the Sobolev embedding theorem and the Leibniz differentiation rule, for all \(p\in [4,\infty)\) we deduce the following weighted embedding estimate for any function \(G\in  W^{1,\frac{4p}{4+p}}_{\langle x\rangle^\alpha}(\mathbb R^4)\): 
\begin{equation}\label{eq-embeding-0901}
	\begin{aligned}
 \|G\|_{L^p_{\langle x\rangle^\alpha}(\mathbb{R}^4)}\leq &C\left\|\langle \cdot \rangle^{\frac{\alpha}{p}}G\right\|_{\dot{W}^{1,\frac{4p}{4+p}}(\mathbb{R}^4)}\\
 \leq&C_{p,\alpha}\left\|\langle \cdot \rangle^{\frac{\alpha}{p}-1}G\right\|_{L^{\frac{4p}{4+p}}(\mathbb{R}^4)}+C\left\| \nabla G\right\|_{L_{\langle x\rangle ^\alpha}^{\frac{4p}{4+p}}(\mathbb{R}^4)}\\
 \leq&C_{p,\alpha}\left\| G\right\|_{L_{\langle x\rangle ^\alpha}^{\frac{4p}{4+p}}(\mathbb{R}^4)}+C\left\| \nabla G\right\|_{L_{\langle x\rangle ^\alpha}^{\frac{4p}{4+p}}(\mathbb{R}^4)}.
	\end{aligned}
\end{equation}
Applying the above general estimate \eqref{eq-embeding-0901} to \(G=V^N\), we further obtain
\begin{equation*}
	\left\|V^N\right\|_{L^p_{\langle x\rangle^\alpha}(\mathbb{R}^4)}\leq C\left\| V^N\right\|_{L_{\langle x\rangle ^\alpha}^{\frac{4p}{4+p}}(\mathbb{R}^4)}+C \left\| \nabla V^N\right\|_{L_{\langle x\rangle ^\alpha}^{\frac{4p}{4+p}}(\mathbb{R}^4)}.
\end{equation*}
Insering this inequality into \eqref{eq-weighted-0832-3} leads to 
\begin{equation}\label{eq-weighted-0832-4}
	\begin{split}
		&	\left\|V^N\right\|_{L^p_{\langle x\rangle^\alpha}(\mathbb{R}^4)}+	\left\|\nabla V^N\right\|_{L^p_{\langle x\rangle^\alpha}(\mathbb{R}^4)}\\
		\leq &CN^\alpha\left( \| V\|_{L^p(\mathbb{R}^4)}\left\| \nabla V^N\right\|_{L^\infty(\mathbb{R}^4)}+\|F\|_{L^p(\mathbb{R}^4)}\right)+C\left\| V^N\right\|_{L_{\langle x\rangle ^\alpha}^{\frac{4p}{4+p}}(\mathbb{R}^4)}+C\left\| \nabla V^N\right\|_{L_{\langle x\rangle ^\alpha}^{\frac{4p}{4+p}}(\mathbb{R}^4)}  .
	\end{split}
\end{equation}
Since $\frac{4p}{4+p}\in [2,4)$, it follows from    \eqref{eq-claim-0831-2} that 
for each $p\in[2,\infty) $ and $\alpha\in\big[0,4(\frac{4p}{4+p}-1)\big)$,
\begin{equation}\label{eq-claim-0831-5}
	\left\|V^N\right\|_{L^p_{\langle x\rangle^\alpha}(\mathbb{R}^4)}+	\left\|\nabla V^N\right\|_{L^p_{\langle x\rangle^\alpha}(\mathbb{R}^4)}\leq C(N).
\end{equation}
As the above estimate \eqref{eq-claim-0831-5} depends on the parameter \(N\), our next goal is to remove this dependence and establish a uniform bound. By applying Remark \ref{rem-heat-weighted} once again, we obtain, for each \(p\in(2,4)\),
	\begin{equation}\label{eq-weighted-0832-2}
		\left\|V^N\right\|_{L^p_{\langle x\rangle }(\mathbb{R}^4)}+\left	\|\nabla V^N\right\|_{L^p_{\langle x\rangle }(\mathbb{R}^4)}
		\leq C	\left\|\mathscr{P}(\varphi_NV)\cdot\nabla V^N\right\|_{L^p_{\langle x\rangle }(\mathbb{R}^4)}+C\left\| \varphi_NF \right\|_{L^p_{\langle x\rangle }(\mathbb{R}^4)}.
\end{equation}
By employing a localized spatial decomposition technique together with Theorem \ref{thm-weight-boundness}, we obtain, for each \(p\in[2,\infty)\),
\begin{equation}\label{eq-embeding-0901-6}
	\begin{aligned}
	&	\left\|\mathscr{P}(\varphi_N V)\cdot\nabla V^N\right\|_{L^p_{\langle x\rangle }(\mathbb{R}^4)}\\
	\leq& \left\|\mathscr{P}(\varphi^N\varphi_RV)\cdot\nabla V^N\right\|_{L^p_{\langle x\rangle }(\mathbb{R}^4)}+\left\|\mathscr{P}(\varphi_N\varphi_R^cV)\cdot\nabla V^N\right\|_{L^p_{\langle x\rangle }(\mathbb{R}^4)}\\
		\leq &\|\nabla V^N\|_{L^\infty(\mathbb{R}^4)}\left\|\mathscr{P}(\varphi_N\varphi_RV)\right\|_{L^p_{\langle x\rangle }(\mathbb{R}^4)} +\left\|\mathscr{P}(\varphi_N\varphi_R^cV)\right\|_{L^\infty(\mathbb{R}^4)}\left\| \nabla V^N\right\|_{L^p_{\langle x\rangle }(\mathbb{R}^4)}\\
			\leq &C\|\nabla V^N\|_{L^\infty(\mathbb{R}^4)}\left\| \varphi_N\varphi_RV\right\|_{L^p_{\langle x\rangle }(\mathbb{R}^4)} +C\left\| \varphi_N\varphi_R^cV\right\|_{W^{2,3}(\mathbb{R}^4)}\left\| \nabla V^N\right\|_{L^p_{\langle x\rangle }(\mathbb{R}^4)}\\
				\leq &CR\|\nabla V^N\|_{L^\infty(\mathbb{R}^4)}\left\|  V\right\|_{L^p (\mathbb{R}^4)} +C\left\| \varphi_N\varphi_R^cV\right\|_{W^{2,3}(\mathbb{R}^4)}\left\| \nabla V^N\right\|_{L^p_{\langle x\rangle }(\mathbb{R}^4)},
	\end{aligned}
\end{equation}
where \(\varphi_R=\varphi(x/R)\) and \(\varphi_R^c=1-\varphi_R\), with \(R>0\) to be fixed later.

Inserting \eqref{eq-embeding-0901-6} into \eqref{eq-weighted-0832-2}, we readily obtain, for each \(p\in(2,4)\) and \(R>R_0\),
\begin{equation}\label{eq-embeding-0901-61}
	\begin{aligned}
	&	\left\|V^N\right\|_{L^p_{\langle x\rangle }(\mathbb{R}^4)}+	\left\|\nabla V^N\right\|_{L^p_{\langle x\rangle }(\mathbb{R}^4)}\\
	\leq &	CR\|\nabla V^N\|_{L^\infty(\mathbb{R}^4)}\left\|  V\right\|_{L^p (\mathbb{R}^4)} +C\left\| \varphi_N\varphi_R^cV\right\|_{W^{2,3}(\mathbb{R}^4)}\left\| \nabla V^N\right\|_{L^p_{\langle x\rangle }(\mathbb{R}^4)} +C\left\|F \right\|_{L^p_{\langle x\rangle }(\mathbb{R}^4)}.
			\end{aligned}
	\end{equation}
	Taking \(R\) sufficiently large in \eqref{eq-embeding-0901-61}, there exists \(R_0>0\) such that, for all \(R>R_0\),
	$$C\left\| \varphi_N\varphi_R^cV\right\|_{W^{2,3}(\mathbb{R}^4)}\leq \frac12.$$
	Hence, it follows from \eqref{eq-embeding-0901-6} that, for each \(p\in(2,4)\) and all \(R>R_0\),
	\begin{equation}\label{eq-embeding-0901-66}
		\begin{aligned}
		 	\left\|V^N\right\|_{L^p_{\langle x\rangle }(\mathbb{R}^4)}+	\left\|\nabla V^N\right\|_{L^p_{\langle x\rangle }(\mathbb{R}^4)} 
			\leq  	CR\|\nabla V^N\|_{L^\infty(\mathbb{R}^4)}\left\|  V\right\|_{L^p (\mathbb{R}^4)}   +C\left\|F \right\|_{L^p_{\langle x\rangle }(\mathbb{R}^4)}.
		\end{aligned}
	\end{equation}
	By the H\"older inequality, we obtain, for each \(p\in[2,\infty)\) and all \(R>R_0\),
	\begin{equation}\label{eq-embeding-0901-7}
		\begin{aligned}
			\left\|F \right\|_{L^p_{\langle x\rangle }(\mathbb{R}^4)}\leq& \||\cdot|U_0\|_{L^\infty(\mathbb{R}^4)}	\|\nabla V\|_{L^p(\mathbb{R}^4)}+\||\cdot|\nabla U_0\|_{L^\infty(\mathbb{R}^4)}	\|  V\|_{L^p(\mathbb{R}^4)}\\&+\||\cdot|U_0\|_{L^\infty(\mathbb{R}^4)}	\|\nabla U_0\|_{L^p(\mathbb{R}^4)}.
		\end{aligned}
	\end{equation}
	Plugging \eqref{eq-embeding-0901-7} in \eqref{eq-embeding-0901-66} and using \eqref{eq-W2q-6-0831} and Proposition \ref{prop-H2}, we finally obtain that 
	 for each $p\in(2,4)$,
	\begin{equation}\label{eq-embeding-0901-8}
			\left\|V^N\right\|_{L^p_{\langle x\rangle }(\mathbb{R}^4)}+	\left\|\nabla V^N\right\|_{L^p_{\langle x\rangle }(\mathbb{R}^4)} 
			\leq  	C(U_0).
		\end{equation}
		Now we focus on the case \(p\in[4,\infty)\). From \eqref{eq-weighted-0832-3}, and by an argument analogous to that used in \eqref{eq-embeding-0901-66}, we obtain, for all \(R>R_0\),
 \begin{equation*}\label{eq-weighted-0832-9}
			\begin{split}
				&	\left\|V^N\right\|_{L^p_{\langle x\rangle }(\mathbb{R}^4)}+	\left\|\nabla V^N\right\|_{L^p_{\langle x\rangle }(\mathbb{R}^4)}\\
				\leq &CR\|\nabla V^N\|_{L^\infty(\mathbb{R}^4)}\left\|  V\right\|_{L^p (\mathbb{R}^4)}   +C\left\|F \right\|_{L^p_{\langle x\rangle }(\mathbb{R}^4)} +C\left\| \nabla V^N\right\|_{L_{\langle x\rangle  }^{\frac{4p}{4+p}}(\mathbb{R}^4)}  .
			\end{split}
		\end{equation*}
	Moreover,	 using \eqref{eq-W2q-6-0831}, \eqref{eq-embeding-0901-7}, \eqref{eq-embeding-0901-8} and Proposition \ref{prop-H2}, we finally obtain that 
		for each $p\in[4,\infty)$,
			\begin{equation}\label{eq-embeding-0901-10}
		\left	\|V^N\right\|_{L^p_{\langle x\rangle }(\mathbb{R}^4)}+	\left\|\nabla V^N\right\|_{L^p_{\langle x\rangle }(\mathbb{R}^4)} 
			\leq  	C(U_0).
		\end{equation}
		Taking $N\to \infty$ in both  estimates \eqref{eq-embeding-0901-8} and \eqref{eq-embeding-0901-10} yields the desired estimate. Thus, we finish the proof of the proposition.
\end{proof}

\subsubsection{The decay estimate with logarithmic loss}
  From Proposition \ref{prop-1p-weighted}, we can reasonably infer that \begin{equation}\label{eq-key-decay-0901-1}
 	\sup_{x\in\mathbb{R}^n}|x| |V(x)|\leq C\||\cdot|V\|_{W^{1,5}(\mathbb{R}^4)}
 	\leq 	C\|V\|_{L^5_{\langle x\rangle}(\mathbb{R}^4)}+C	\|\nabla V\|_{L^5_{\langle x\rangle}(\mathbb{R}^4)}<\infty.
 \end{equation}
 Starting from the first-order decay estimate established above, we now proceed to derive higher-order spatial decay estimates for the solution $(V,P)$. More precisely, in this subsubsection, we aim to establish the following result.
 \begin{proposition}\label{prop-decy-loss-log}
 	Let $\sigma\in W^{1,\infty}(\mathbb{S}^3)$, and suppose the pair $(V,P)\in H^1(\mathbb{R}^4)\times L^2(\mathbb{R}^4)$ is a weak solution to problem \eqref{PL}. Then for every $p\in(2,\infty)$, we have
 	\[
 	(V,P)\in W^{2,p}(\mathbb{R}^4)\times W^{1,p}(\mathbb{R}^4),
 	\]
 	and the pointwise decay estimate
 	\[
 	|V(x)|\leq C|x|^{-3}\log(1+|x|)
 	\]
 	holds for all $x\in\mathbb{R}^4$.
 \end{proposition}
 Before proceeding to the proof of this proposition, we introduce a notation and state a useful lemma. Let us define
 \[
 g_t(x):=\mathscr{F}^{-1}\big(\varphi(\xi)e^{-t|\xi|^2}\big)(x).
 \]
 Then
 \[
 \Delta_j e^{t\Delta}u_0
 =
 \mathscr{F}^{-1}\left(
 \varphi\left(\frac{\xi}{2^j}\right)
 e^{-t|\xi|^2}\widehat{u_0}(\xi)
 \right)
 =
 g_{t,j}\ast u_0(x),
 \]
 where
 \[
 g_{t,j}(x)
 =
 2^{nj}g_{2^{2j}t}(2^jx).
 \]
 The following lemma provides a  pointwise estimate for the kernel \(g_t\) that will be used repeatedly below.
 \begin{lemma}\label{lem-bcd-2.4}
 	There holds that 
 	$$\mid g_t(x)\mid \leq C\left(1+|x|^2\right)^{-n} e^{-c t}.$$
 \end{lemma}
 \begin{proof}
 	This inequality follows from the proof of Lemma 2.4 in Chapter 2 of \cite{BCD}.
 \end{proof}
 With this lemma at hand, we are now in a position to proceed with the proof of Proposition \ref{prop-decy-loss-log}. We first represent the solution by means of the heat semigroup and then derive the desired higher-order spatial decay estimates.
 \begin{proof}[Proof of Proposition \ref{prop-decy-loss-log}]
 	According to Theorem \ref{thm-Rep-forHeatSelf}, we write
 	\begin{equation}\label{eq-Integral-0901}
 		\begin{aligned}
 				V(x)=&\int_0^1e^{(1-s)\Delta}s^{-\frac32}(\operatorname{div}\mathbb{A}+\operatorname{div}\mathbb{A}_0+\nabla P)\Big(\frac{x}{\sqrt{s}}\Big)\,\mathrm{d}s\\
 				=&\int_0^1e^{(1-s)\Delta}s^{-\frac32}\mathscr{P}(\operatorname{div}\mathbb{A}+\operatorname{div}\mathbb{A}_0)\Big(\frac{x}{\sqrt{s}}\Big)\,\mathrm{d}s,
 		\end{aligned}
 	\end{equation}
 	where $\mathbb{A}=V\otimes U_0+U_0\otimes V+V\times V,$ $\mathbb{A}_0=U_0\otimes U_0$ and $P$ solves 
 	\[-\Delta P=\operatorname{div}\operatorname{div}\mathbb{A}+\operatorname{div}\operatorname{div}\mathbb{A}_0.\]
 By Proposition	\ref{prop-decay-divergence}, we have 
 \begin{equation*}
 	\begin{aligned}
 		 	\sup_{x\in\mathbb{R}^4}|x|^2 |V(x)|\leq C  \bigg(&\sup_{x\in\mathbb{R}^4}|x|^2 |U_0\otimes V(x)|+\sup_{x\in\mathbb{R}^4}|x|^2 |V\otimes U_0(x)|\\&+\sup_{x\in\mathbb{R}^4}|x|^2 |V\otimes V(x)|+\sup_{x\in\mathbb{R}^4}|x|^2 |U_0\otimes U_0(x)|\bigg)\\
 		 	\leq C&\left(\sup_{x\in\mathbb{R}^4}|x| |V(x)|\right)^2+C\left(\sup_{x\in\mathbb{R}^4}|x| |U_0(x)|\right)^2.
 	\end{aligned}
 \end{equation*}
 Moreover,  by \eqref{eq-key-decay-0901-1}, one has 
 \begin{equation}\label{eq-Integral-0901-1}
 	\sup_{x\in\mathbb{R}^4}|x|^2 |V(x)|<\infty.
 \end{equation}
 Letting 
 \[V_H=\int_0^1e^{(1-s)\Delta}s^{-\frac32}\mathscr{P}\operatorname{div}\mathbb{A}\Big(\frac{x}{\sqrt{s}}\Big)\,\mathrm{d}s,\]
 it follows from Proposition	\ref{prop-decay-divergence} that 
 \begin{equation*}
 	\begin{aligned}
 		\sup_{x\in\mathbb{R}^4}|x|^3|V_H(x)|\leq C  \bigg(&\sup_{x\in\mathbb{R}^4}|x|^3 |U_0\otimes V(x)|+\sup_{x\in\mathbb{R}^4}|x|^3 |V\otimes U_0(x)+\sup_{x\in\mathbb{R}^4}|x|^3 |V\otimes U_0(x)|\bigg)\\
 		\leq C&\left(\sup_{x\in\mathbb{R}^4}|x| |V(x)|\right)^2+\left(\sup_{x\in\mathbb{R}^4}|x|^2 |V(x)|\right)^2+\left(\sup_{x\in\mathbb{R}^4}|x| |U_0(x)|\right)^2.
 	\end{aligned}
 \end{equation*}
 By \eqref{eq-key-decay-0901-1} and \eqref{eq-Integral-0901-1}, we readily have 
  \begin{equation}\label{eq-Integral-0901-2}
 	\sup_{x\in\mathbb{R}^4}|x|^3|V_H(x)|<\infty.
 \end{equation}
 Now we just need to bound the following term 
 \begin{equation*}
 	\begin{split}
 	&\int_0^1e^{(1-s)\Delta}s^{-\frac32}\mathscr{P}\operatorname{div}\mathbb{A}_0\Big(\frac{x}{\sqrt{s}}\Big)\,\mathrm{d}s\\
 	=&\int_0^1e^{(1-s)\Delta}s^{-\frac32}\operatorname{div}\big(U_0\otimes U_0)\Big(\frac{x}{\sqrt{s}}\Big)\,\mathrm{d}s+\int_0^1e^{(1-s)\Delta}s^{-\frac32}\nabla P_0\Big(\frac{x}{\sqrt{s}}\Big)\,\mathrm{d}s=:V_0^\sharp+V_0^\natural,
 	\end{split}
 \end{equation*}
 where $P_0$ satisfies 
 \[-\Delta P_0=\operatorname{div}\operatorname{div}\big(U_0\otimes U_0).\]
 In view of Proposition \ref{prop-heat-decay-point}, we obtain 
 \begin{equation}\label{eq-Integral-0902-3}
 	\begin{split}
 			\sup_{x\in\mathbb{R}^4}|x|^3|V_0^\sharp(x)|\leq& C  \sup_{x\in\mathbb{R}^4}|x|^3 |U_0\cdot\nabla  U_0(x)|\\
 			\leq &\sup_{x\in\mathbb{R}^4}|x| |U_0(x)|\sup_{x\in\mathbb{R}^4}|x|^2 |\nabla U_0(x)|<\infty.
 	\end{split}
 \end{equation}
 Now we tackle with the term involving the pressure $P_0$ as follows
 \begin{equation*}
	\begin{split}
		 \int_0^1e^{(1-s)\Delta}s^{-\frac32}\nabla P_0\Big(\frac{x}{\sqrt{s}}\Big)\,\mathrm{d}s
		=&\sum_{j\in\mathbb{Z}}\int_0^1\dot{\Delta}_j\Phi_{1-s} \ast s^{-\frac32} \nabla P_0\Big(\frac{x}{\sqrt{s}}\Big)\,\mathrm{d}s \\	=&\sum_{j\in\mathbb{Z}}\int_0^1\tilde{\dot{\Delta}}_j\Phi_{1-s} \ast P_{0,j}(x,s) \,\mathrm{d}s, 		
	\end{split}
\end{equation*}
where 
$$P_{0,j}(x,s)= \dot{\Delta}_j \left(s^{-\frac32} \nabla P_0\Big(\frac{x}{\sqrt{s}}\Big)\right).$$
Using high-, medium-, and low-frequency decomposition techniques, we can express the term including  $\nabla P_0$ as follows
\begin{equation}\label{eq-decay-0902-0}
	\begin{aligned}
	 \int_0^1e^{(1-s)\Delta}s^{-\frac32}\nabla P_0\Big(\frac{x}{\sqrt{s}}\Big)\,\mathrm{d}s=&\left(\sum_{j\geq0}+\sum_{N_0\leq j<0}+\sum_{j<N_0}\right)\int_0^1\tilde{\dot{\Delta}}_j\Phi_{1-s} \ast P_{0,j}(x,s) \,\mathrm{d}s\\
	 =:&I+II+III, 
	 \end{aligned}
\end{equation}
where $N_0=[-\log_2 (1+|x|)]$.

Since 
\[I=\sum_{j\geq0}\int_0^1\tilde{\dot{\Delta}}_j\Phi_{1-s} \ast P_{0,j}(x,s) \,\mathrm{d}s,\]
we have by Lemma \ref{lem-point-LS} and Proposition \ref{prop-0902-P-D} that 
\begin{equation*}
	\begin{aligned}
	|I|	\leq &C|x|^{-3}\sum_{j\geq0} \int_0^1\left(\sup_{x\in\mathbb{R}^4}|x|^4|g_{1-s,j}|(x)+\|g_{1-s,j}\|_{L^1(\mathbb{R}^4)}\right)\sup_{x\in\mathbb{R}^4}|x|^3| P_{0,j}(x,s)|(x)\,\mathrm{d}s\\
	\leq&C|x|^{-3}\sup_{x\in\mathbb{R}^4}|x|^3|U_0\cdot\nabla U_0|\sum_{j\geq0}\int_0^1e^{-c2^{2j}(1-s)}\,\mathrm{d}s\\
	\leq&C|x|^{-3}\sup_{x\in\mathbb{R}^4}|x|^3|U_0\cdot\nabla U_0| \sum_{j\geq0}2^{-2j}.
	\end{aligned}
\end{equation*}
From this, it follows that 
\begin{equation}\label{eq-decay-0902-I}
|I|	\leq	C|x|^{-3}\sup_{x\in\mathbb{R}^4}|x||U_0(x)|\sup_{x\in\mathbb{R}^4}|x|^2| \nabla U_0(x)|.
\end{equation}
For the low-frequency regime, we rewrie 
\[III=\sum_{j<N_0}\int_0^1\nabla \tilde{\dot{\Delta}}_j\Phi_{1-s} \ast \dot{\Delta}_j \left(s^{-1}  P\Big(\frac{x}{\sqrt{s}}\Big)\right)  \,\mathrm{d}s. \]
By Lemma \ref{lem-point-LS} and Proposition \ref{prop-0902-P-D}, we obtain 
\begin{equation*}
	\begin{aligned}
		|III|	\leq &C|x|^{-2}\sum_{j<N_0} 2^j\int_0^1 \left(\sup_{x\in\mathbb{R}^4}|x|^4|g_{1-s,j}|(x)+\||g_{1-s,j}\|_{L^1(\mathbb{R}^4)}\right)\sup_{x\in\mathbb{R}^4}|x|^2\left| \dot{\Delta}_j \left(s^{-1}  P\Big(\frac{x}{\sqrt{s}}\Big)\right)\right|\,\mathrm{d}s\\
		\leq &C|x|^{-2}\sup_{x\in\mathbb{R}^4}|x|^3|U_0\cdot\nabla U_0| \sum_{j<N_0} 2^j\int_0^1(1-s)^{-1/2} \left(\sup_{x\in\mathbb{R}^4}|x|^{4}|g_{1,j}|(x)+\||g_{1,j}\|_{L^1(\mathbb{R}^4)}\right) \,\mathrm{d}s\\
		\leq&C|x|^{-2}\sup_{x\in\mathbb{R}^4}|x|^3|U_0\cdot\nabla U_0|\sum_{j<N_0} 2^j,
	\end{aligned}
\end{equation*}
where we have used the fact 
\(\nabla G_1(x)=\frac{x}{2}G_1(x).\)
Thus, we have 
\begin{equation}\label{eq-decay-0902-III}
		|III|	\leq	C|x|^{-3}\sup_{x\in\mathbb{R}^4}|x||U_0(x)|\sup_{x\in\mathbb{R}^4}|x|^2| \nabla U_0(x)|.
\end{equation}
As for $II,$ in the same fashion  as used in \eqref{eq-decay-0902-I}, we can show that 
\begin{equation}\label{eq-decay-0902-II}
	\begin{aligned}
		|II|	\leq &C|x|^{-3}\sum_{N_0\leq j<0} \int_0^1\left(\sup_{x\in\mathbb{R}^4}|x|^4|g_{1-s,j}|(x)+\|g_{1-s,j}\|_{L^1(\mathbb{R}^4)}\right)\sup_{x\in\mathbb{R}^4}|x|^3| P_{0,j}(x,s)|(x)\,\mathrm{d}s\\
		\leq&C|x|^{-3}\sup_{x\in\mathbb{R}^4}|x|^3|U_0\cdot\nabla U_0|\sum_{N_0\leq j<0} \\
		\leq&	C|x|^{-3}\log (1+|x|)\sup_{x\in\mathbb{R}^4}|x||U_0(x)|\sup_{x\in\mathbb{R}^4}|x|^2| \nabla U_0(x)|.
	\end{aligned}
\end{equation}
Inserting \eqref{eq-decay-0902-I},  \eqref{eq-decay-0902-III} and \eqref{eq-decay-0902-II} into \eqref{eq-decay-0902-0} leads to 
\begin{equation}\label{eq-decay-0902-Z}
	\left| V_0^\natural\right|	\leq	C|x|^{-3}\log (1+|x|)\sup_{x\in\mathbb{R}^4}|x||U_0(x)|\sup_{x\in\mathbb{R}^4}|x|^2| \nabla U_0(x)|.
\end{equation}
Collecting estimates \eqref{eq-Integral-0901-2}, \eqref{eq-Integral-0902-3} and \eqref{eq-decay-0902-Z} yields the desired estimate. Thus, we complete the proof of Proposition \ref{prop-decy-loss-log}.
 \end{proof}

\subsubsection{The sharp decay estimate }
 In Proposition \ref{prop-decy-loss-log}, the decay estimate exhibits a logarithmic loss. In the present subsubsection, we demonstrate that such a loss may be removed by assuming higher regularity on the initial data. We first state a useful lemma.
 
 \begin{lemma}\label{lem-UU00}
 		Let \(0<\gamma<1\) and 
 		\[
 		u_0(x)=\frac{\sigma(x/|x|)}{|x|},\qquad \sigma\in C^{1,\gamma}(\mathbb S^3).
 		\]
 		Let
 $	U_0(x)=e^{\Delta}u_0(x).
 		$
 		Then there exists a constant \(C>0\) depending only on \(\gamma\) and the \(C^{1,\gamma}\)-norm of \(\sigma\) such that for all \(x,y\in\mathbb R^4\setminus\{0\}\) satisfying
 		\[
 		|x-y|\leq \frac12\min\{|x|,|y|\},
 		\]
 		the following estimates hold:
 		\[
 		|\nabla U_0(x)-\nabla U_0(y)|\leq C|x-y|^\gamma \bigl(\min\{|x|,|y|\}\bigr)^{-2-\gamma},
 		\]
 		and
 		\[
 		|U_0(x)-U_0(y)|\leq C|x-y| \bigl(\min\{|x|,|y|\}\bigr)^{-2}.
 		\]
 \end{lemma}
 \begin{proof}
 	Let
 	\[
 	\rho:=\min\{|x|,|y|\},
 	\qquad
 	h:=|x-y|.
 	\]
 	We assume that
 	$
 	h\leq \frac12\rho.
 	$
In this setting, it holds that
 	\[
 	|x|\sim |y|\sim \rho.
 	\]
 	We first derive some estimates for $u_0$. Since
 	\[
 	u_0(x)=|x|^{-1}\sigma\left(\frac{x}{|x|}\right),
 	\qquad
 	\sigma\in C^{1,\gamma}(\mathbb S^3),
 	\]
 	a direct differentiation gives
 	\[
 	\nabla u_0(x)
 	=
 	|x|^{-2}A\left(\frac{x}{|x|}\right),
 	\]
 	where
 	$
 	A\in C^\gamma(\mathbb S^3;\mathbb R^4)
 	$
 	and
 	\[
 	\|A\|_{C^\gamma(\mathbb S^3)}
 	\leq
 	C\|\sigma\|_{C^{1,\gamma}(\mathbb S^3)}.
 	\]
 	Consequently,
 	\begin{equation}\label{eq:u0-grad-pointwise}
 		|\nabla u_0(x)|
 		\leq C|x|^{-2}.
 	\end{equation}
 	
 	Now we claim that, for all $a,b\in\mathbb R^4\setminus\{0\}$,
 	\begin{equation}\label{eq:global-grad-u0-holder}
 		|\nabla u_0(a)-\nabla u_0(b)|
 		\leq
 		C|a-b|^\gamma
 		\bigl(\min\{|a|,|b|\}\bigr)^{-2-\gamma}.
 	\end{equation}	Indeed, if
 	\[
 	|a-b|
 	\leq
 	\frac12\min\{|a|,|b|\},
 	\]
 	then \eqref{eq:global-grad-u0-holder} follows from the
 	$C^\gamma$ regularity of $A$. More precisely,
 	\[
 	\nabla u_0(a)-\nabla u_0(b)
 	=
 	|a|^{-2}
 	\left[
 	A\left(\frac{a}{|a|}\right)
 	-
 	A\left(\frac{b}{|b|}\right)
 	\right]
 	+
 	\left(|a|^{-2}-|b|^{-2}\right)
 	A\left(\frac{b}{|b|}\right).
 	\]
 	Since 
 	\[
 	\left|
 	\frac{a}{|a|}
 	-
 	\frac{b}{|b|}
 	\right|
 	\leq
 	C\frac{|a-b|}
 	{\min\{|a|,|b|\}},
 	\]
 	and
 	\[
 	\bigl||a|^{-2}-|b|^{-2}\bigr|
 	\leq
 	C|a-b|
 	\bigl(\min\{|a|,|b|\}\bigr)^{-3},
 	\]
 	we obtain, using $0<\gamma<1$ and
 	\[
 	|a-b|\leq \frac12\min\{|a|,|b|\},
 	\]
 	that
 	\[
 	|\nabla u_0(a)-\nabla u_0(b)|
 	\leq
 	C|a-b|^\gamma
 	\bigl(\min\{|a|,|b|\}\bigr)^{-2-\gamma}.
 	\]
 	
 	On the other hand, if
 	\[
 	|a-b|>
 	\frac12\min\{|a|,|b|\},
 	\]
 	then, by \eqref{eq:u0-grad-pointwise},
 	\[
 	|\nabla u_0(a)-\nabla u_0(b)|
 	\leq
 	C\bigl(|a|^{-2}+|b|^{-2}\bigr)
 	\leq
 	C\bigl(\min\{|a|,|b|\}\bigr)^{-2}.
 	\]
 	Since
 	\[
 	|a-b|^\gamma
 	\geq
 	2^{-\gamma}
 	\bigl(\min\{|a|,|b|\}\bigr)^\gamma,
 	\]
 	we again obtain
 	\[
 	|\nabla u_0(a)-\nabla u_0(b)|
 	\leq
 	C|a-b|^\gamma
 	\bigl(\min\{|a|,|b|\}\bigr)^{-2-\gamma}.
 	\]
 	Thus \eqref{eq:global-grad-u0-holder} holds globally.
 	
 	Similarly, we have the global estimate
 	\begin{equation}\label{eq:global-u0-lipschitz}
 		|u_0(a)-u_0(b)|
 		\leq
 		C|a-b|
 		\bigl(\min\{|a|,|b|\}\bigr)^{-2},
 	\end{equation}
 	for all $a,b\in\mathbb R^4\setminus\{0\}$.
 	Indeed, if
 	\[
 	|a-b|
 	\leq
 	\frac12\min\{|a|,|b|\},
 	\]
 	then \eqref{eq:global-u0-lipschitz} follows from
 	\eqref{eq:u0-grad-pointwise} and the mean value theorem.
 	If
 	\[
 	|a-b|>
 	\frac12\min\{|a|,|b|\},
 	\]
 	then
 	\[
 	|u_0(a)-u_0(b)|
 	\leq
 	C\bigl(|a|^{-1}+|b|^{-1}\bigr)
 	\leq
 	C\bigl(\min\{|a|,|b|\}\bigr)^{-1}
 	\]
 	and therefore
 	\[
 	|u_0(a)-u_0(b)|
 	\leq
 	C|a-b|
 	\bigl(\min\{|a|,|b|\}\bigr)^{-2}.
 	\]
 	
 	We now turn to tackle with $U_0$. Since
 	\[
 	U_0(x)
 	=
 	\int_{\mathbb R^4}
 	\Phi_1(z)u_0(x-z)\,\mathrm{d}z,
 	\]
 	it follows that
 	\[
 	\begin{aligned}
 		\nabla U_0(x)-\nabla U_0(y)
 		&=
 		\int_{\mathbb R^4}
 		\Phi_1(z)
 		\bigl(
 		\nabla u_0(x-z)-\nabla u_0(y-z)
 		\bigr)\,\mathrm{d}z.
 	\end{aligned}
 	\]
 	Hence, by \eqref{eq:global-grad-u0-holder},
 	\begin{equation}\label{eq-260912-1}
 		\begin{aligned}
 			|\nabla U_0(x)-\nabla U_0(y)|
 			&\leq
 			Ch^\gamma
 			\int_{\mathbb R^4}
 			\Phi_1(z)
 			\bigl(
 			\min\{|x-z|,|y-z|\}
 			\bigr)^{-2-\gamma}\,\mathrm{d}z.
 		\end{aligned}
 	\end{equation}
 	It remains for us to estimate the integral in \eqref{eq-260912-1}. We claim that
 	\begin{equation}\label{eq:kernel-weight}
 		\int_{\mathbb R^4}
 		\Phi_1(z)
 		\bigl(
 		\min\{|x-z|,|y-z|\}
 		\bigr)^{-2-\gamma}\,\mathrm{d}z
 		\leq
 		C\rho^{-2-\gamma}.
 	\end{equation}
 	To prove this,  we decompose 
 	$
 	\mathbb R^4=E_1\cup E_2,
 	$
 	where
 	\[
 	E_1:=\{|z|\leq \rho/2\},
 	\qquad
 	E_2:=\{|z|>\rho/2\}.
 	\]
 	For $z\in E_1$, since
 	\[
 	|x-z|\geq |x|-|z|\geq \frac{\rho}{2},
 	\qquad
 	|y-z|\geq \frac{\rho}{2},
 	\]
 	we have
 	\[
 	\bigl(
 	\min\{|x-z|,|y-z|\}
 	\bigr)^{-2-\gamma}
 	\leq
 	C\rho^{-2-\gamma}.
 	\]
 	Therefore,
 	\[
 	\int_{E_1}
 	\Phi_1(z)
 	\bigl(
 	\min\{|x-z|,|y-z|\}
 	\bigr)^{-2-\gamma}\,\mathrm{d}z
 	\leq
 	C\rho^{-2-\gamma}.
 	\]
 	For $E_2$, we observe that
 	\[
 	\min\{|x-z|,|y-z|\}^{-2-\gamma}
 	\leq
 	|x-z|^{-2-\gamma}
 	+
 	|y-z|^{-2-\gamma}.
 	\]
 	Consequently,
 	\[
 		\int_{E_2}
 		\Phi_1(z)
 		\bigl(
 		\min\{|x-z|,|y-z|\}
 		\bigr)^{-2-\gamma}\,\mathrm{d}z
 		\leq
 		C\int_{|z|>\rho/2}
 		e^{-|z|^2/4}
 		\left(
 		|x-z|^{-2-\gamma}
 		+
 		|y-z|^{-2-\gamma}
 		\right)\,\mathrm{d}z.
 	\]
 	We now focus on the term associated with $x$. To proceed with the estimate, we decompose
 	\[
 	\{|z|>\rho/2\}
 	=
 	F_1\cup F_2,
 	\]
 	where
 	\[
 	F_1:=\{|z|>\rho/2,\ |x-z|\leq \rho/4\},\quad 
 	F_2:=\{|z|>\rho/2,\ |x-z|>\rho/4\}.
 	\]
 	On $F_1$,
 	\[
 	|z|
 	\geq
 	|x|-|x-z|
 	\geq
 	\rho-\frac{\rho}{4}
 	=
 	\frac{3\rho}{4},
 	\]
 	and hence
 	\[
 	e^{-|z|^2/4}
 	\leq
 	e^{-c\rho^2}.
 	\]
 	Since $2+\gamma<4$,
 	\[
 	\int_{|x-z|\leq\rho/4}
 	|x-z|^{-2-\gamma}\,\mathrm{d}z
 	\leq
 	C\rho^{2-\gamma}.
 	\]
 	Thus
 	\[
 	\int_{F_1}
 	e^{-|z|^2/4}|x-z|^{-2-\gamma}\,\mathrm{d}z
 	\leq
 	Ce^{-c\rho^2}\rho^{2-\gamma}.
 	\]
 	On $F_2$,
 	\[
 	|x-z|^{-2-\gamma}
 	\leq
 	C\rho^{-2-\gamma},
 	\]
 	and therefore
 	\[
 	\int_{F_2}
 	e^{-|z|^2/4}|x-z|^{-2-\gamma}\,\mathrm{d}z
 	\leq
 	C\rho^{-2-\gamma}
 	\int_{\mathbb R^4}e^{-|z|^2/4}\,\mathrm{d}z
 	\leq
 	C\rho^{-2-\gamma}.
 	\]
 	The same estimates hold for the term involving $y$. Since
 	\[
 	e^{-c\rho^2}\rho^{2-\gamma}
 	\leq
 	C\rho^{-2-\gamma},
 	\]
 	we obtain \eqref{eq:kernel-weight}.
 	
 	Combining \eqref{eq-260912-1} and \eqref{eq:kernel-weight}, we conclude that
 	\[
 	|\nabla U_0(x)-\nabla U_0(y)|
 	\leq
 	C|x-y|^\gamma
 	\bigl(\min\{|x|,|y|\}\bigr)^{-2-\gamma},
 	\]
 	which implies \eqref{eq-grad-U0-holder}.
 	
 	Finally, using the representation
 	\[
 	U_0(x)-U_0(y)
 	=
 	\int_{\mathbb R^4}
 	\Phi_1(z)
 	\bigl(
 	u_0(x-z)-u_0(y-z)
 	\bigr)\,\mathrm{d}z
 	\]
 	and \eqref{eq:global-u0-lipschitz}, we obtain
 	\begin{equation}\label{eq-260912-6}
 		|U_0(x)-U_0(y)|
 		\leq
 		Ch
 		\int_{\mathbb R^4}
 		\Phi_1(z)
 		\bigl(
 		\min\{|x-z|,|y-z|\}
 		\bigr)^{-2}\,\mathrm{d}z.
 	\end{equation}
 	Exactly as above, we have
 	\begin{equation}\label{eq:kernel-weight-2}
 		\int_{\mathbb R^4}
 		\Phi_1(z)
 		\bigl(
 		\min\{|x-z|,|y-z|\}
 		\bigr)^{-2}\,\mathrm{d}z
 		\leq
 		C\rho^{-2}.
 	\end{equation}
 	Indeed, on $|z|\leq\rho/2$ the integrand is bounded by
 	$C\rho^{-2}\Phi_1(z)$, while on $|z|>\rho/2$ we use
 	\[
 	\min\{|x-z|,|y-z|\}^{-2}
 	\leq
 	|x-z|^{-2}+|y-z|^{-2},
 	\]
 	and the singularities are integrable in $\mathbb R^4$.
 	
 	Therefore, from \eqref{eq-260912-6} and \eqref{eq:kernel-weight-2},
 	\[
 	|U_0(x)-U_0(y)|
 	\leq
 	C|x-y|
 	\bigl(\min\{|x|,|y|\}\bigr)^{-2}.
 	\]
 	This implies \eqref{eq-U0-holder}.
 \end{proof}
 Based on this lemma, we next establish sharp decay estimates in H\"older spaces for the integral involving the forcing term $U_0\cdot\nabla U_0$. Specifically:
  \begin{proposition}\label{prop-Cr-weighted-regularized}
 	Let $0<\gamma<1$ and
 	\[
 	u_0(x)=\frac{\sigma(x/|x|)}{|x|},
 	\qquad
 	\sigma\in C^{1,\gamma}(\mathbb S^3).
 	\]
 	Let
 $
 	U_0=e^\Delta u_0.
 	$
 	Then there exists a constant
 	$
 	C=C\bigl(\gamma,\|\sigma\|_{C^{1,\gamma}(\mathbb S^3)}\bigr)>0
 	$
 	such that, for every $j\in\mathbb Z$,
 	\begin{equation}\label{eq-Cr-weighted-est-2026}
 		\sup_{x\in\mathbb R^4}
 		|x|^{3+\gamma}
 		\left|
 		\dot\Delta_j
 		\bigl(U_0\cdot\nabla U_0\bigr)(x)
 		\right|
 		\leq C2^{-j\gamma}.
 	\end{equation}
 \end{proposition}
 \begin{proof}
For simplicity of notation, we denote
 	$$
 	F=U_0\cdot\nabla U_0.
 	$$
 	We first establish the decay and H\"older estimates for $U_0$, which is  governed by
 	\[
 	U_0(x)=(\Phi_1\ast u_0)(x)
 	=
 	\int_{\mathbb R^4}\Phi_1(x-y)u_0(y)\,\mathrm{d}y.
 	\]
 	Since $u_0$ is homogeneous of degree $-1$, standard estimates for
 	convolution with the Gaussian kernel yield
 	\begin{equation}\label{eq-U0-decay}
 		|U_0(x)|
 		\leq C(1+|x|)^{-1},
 	\end{equation}
 	and 
 	\begin{equation}\label{eq-grad-U0-decay}
 		|\nabla U_0(x)|
 		\leq C(1+|x|)^{-2},
 	\end{equation}
 	where the constant $C$ depends only on
 	$\gamma$ and $\|\sigma\|_{C^{1,\gamma}(\mathbb S^3)}$.
 	
 	Moreover, the $C^{1,\gamma}$ regularity of $\sigma$ allows us to apply Lemma \ref{lem-UU00} to obtain that,
 	whenever
 	\[
 	|x-y|\leq \frac12\min\{|x|,|y|\},
 	\]
 	we have
 	\begin{equation}\label{eq-grad-U0-holder}
 		|\nabla U_0(x)-\nabla U_0(y)|
 		\leq
 		C|x-y|^\gamma
 		\bigl(\min\{|x|,|y|\}\bigr)^{-2-\gamma}
 	\end{equation}
 	and
 	\begin{equation}\label{eq-U0-holder}
 		|U_0(x)-U_0(y)|
 		\leq
 		C|x-y|
 		\bigl(\min\{|x|,|y|\}\bigr)^{-2}.
 	\end{equation}
 	From \eqref{eq-U0-decay} and
 	\eqref{eq-grad-U0-decay}, it follows that
 	\begin{equation}\label{eq-F-decay}
 		|F(x)|
 		=
 		|U_0(x)\cdot\nabla U_0(x)|
 		\leq C(1+|x|)^{-3}.
 	\end{equation}
 	Furthermore,
 	\[
 		F(x)-F(y)
 		=
 		\bigl(U_0(x)-U_0(y)\bigr)\cdot\nabla U_0(x)
 		+
 		U_0(y)\cdot
 		\bigl(\nabla U_0(x)-\nabla U_0(y)\bigr).
 	\]
 	Hence, by \eqref{eq-grad-U0-decay},
 	\eqref{eq-grad-U0-holder}, and \eqref{eq-U0-holder},
 	for
 	\[
 	|x-y|\leq \frac12\min\{|x|,|y|\},
 	\]
 	we obtain
 	\begin{equation}\label{eq-F-holder}
 		|F(x)-F(y)|
 		\leq
 		C|x-y|^\gamma
 		\bigl(\min\{|x|,|y|\}\bigr)^{-3-\gamma}.
 	\end{equation}
 	We now estimate the Littlewood--Paley projection of $F$.
 	Since $\varphi$ is supported away from the origin,
 	\[
 	\int_{\mathbb R^4}\check\varphi(z)\,\mathrm{d}z=0.
 	\]
 	Therefore,
 	\begin{subequations}
 		\begin{align}
 				\dot\Delta_jF(x)
 					=&
 				2^{4j}\int_{\mathbb R^4}
 				\check\varphi\bigl(2^j(x-y)\bigr)F(y)\,\mathrm{d}y\label{eq-LP-cancellation-R}\\
 			=&
 			2^{4j}\int_{\mathbb R^4}
 			\check\varphi\bigl(2^j(x-y)\bigr)
 			\bigl(F(y)-F(x)\bigr)\,\mathrm{d}y.\label{eq-LP-cancellation}
 		\end{align}
 	\end{subequations}
 	We now proceed by distinguishing two cases.

 	\noindent
 	\textbf{Case 1:} $|x|\leq 2^{-j}$.

 In term of	\eqref{eq-prop-main-lambda}, one has 
 \begin{equation*}
 	\begin{aligned}
 			|x|^3	|\dot\Delta_jF(x)|\leq& C \sup _{x\in \mathbb{R}^4}|x|^3F(x)\\
 			\leq& C\sup_{x\in\mathbb R^4}|x||U_0(x)|\sup_{x\in\mathbb R^4}|x|^2|\nabla U_0(x)|\\
 			\leq& C\left(\|\sigma\|_{W^{1,\infty}(\mathbb{S}^3)}\right).
 	\end{aligned}
 \end{equation*}
 	Consequently,
 	\[
 	|x|^{3+\gamma}
 	|\dot\Delta_jF(x)|
 	\leq
 	C2^{-j\gamma}	\sup_{x\in\mathbb{R}^4}|x|^3	|\dot\Delta_jF(x)|
 	\leq
 	C2^{-j\gamma}.
 	\]
 	Hence, we have 
 	\begin{equation}\label{eq-case1}
 		\sup_{|x|\leq 2^{-j}}
 		|x|^{3+\gamma}
 		|\dot\Delta_jF(x)|
 		\leq C2^{-j\gamma}.
 	\end{equation}
 	
 	\medskip
 	\noindent
 	\textbf{Case 2:} $|x|>2^{-j}$.
 	
 	We decompose the integral in \eqref{eq-LP-cancellation} into two parts by considering the cases
 	\[
 	|x-y|\leq\frac{|x|}{2} \quad \text{and} \quad |x-y|>\frac{|x|}{2}.
 	\]
 	For the near-field part,
 	$
 	|x-y|\leq\frac{|x|}{2},
 $
 	we have
 	\[\frac{3}{2}|x|\geq |x-y|+|x|\geq  |y|\geq |x|-|y-x|\geq \frac{1}{2}|x|,\]
 which implies $x\sim y.$
 	
 	Hence, by \eqref{eq-F-holder},
 	\[
 	|F(y)-F(x)|
 	\leq
 	C|x-y|^\gamma |x|^{-3-\gamma}.
 	\]
 	It follows that
 	\[
 	\begin{aligned}
 		2^{4j}
 		\int_{|x-y|\leq |x|/2}
 		\left|\check\varphi\bigl(2^j(x-y)\bigr)\right|
 		|F(y)-F(x)|\,\mathrm{d}y
 	\leq
 		C|x|^{-3-\gamma}2^{4j}
 		\int_{\mathbb R^4}
 		\left|\check\varphi\bigl(2^j(x-y)\bigr)\right|
 		|x-y|^\gamma\,\mathrm{d}y.
 	\end{aligned}
 	\]
 	Making the change of variables
 	$
 	z=2^j(x-y),
 	$
 	we obtain
 	\[
 	\begin{aligned}
 		2^{4j}
 		\int_{|x-y|\leq |x|/2}
 		\left|\check\varphi\bigl(2^j(x-y)\bigr)\right|
 		|F(y)-F(x)|\,\mathrm{d}y
 	\leq
 		C2^{-j\gamma}|x|^{-3-\gamma}
 		\int_{\mathbb R^4}
 		|\check\varphi(z)||z|^\gamma\,\mathrm{d}z.
 	\end{aligned}
 	\]
 	Therefore,
 	\begin{equation}\label{eq-near-field}
 		|\dot\Delta_jF(x)|_{\mathrm{near}}
 		\leq
 		C2^{-j\gamma}|x|^{-3-\gamma}.
 	\end{equation}
 	
 	For the far-field region satisfying
 	\[
 	|x-y|>\frac{|x|}{2},
 	\]
 	we invoke \eqref{eq-F-decay} to obtain the estimate
 	\[
 	\begin{aligned}
 		&2^{4j}
 		\int_{|x-y|>|x|/2}
 		\left|\check\varphi\bigl(2^j(x-y)\bigr)\right|
 		|F(y)-F(x)|\,\mathrm{d}y
 		\\
 		\leq&
 		C2^{4j}
 		\int_{|x-y|>|x|/2}
 		\left|\check\varphi\bigl(2^j(x-y)\bigr)\right|
 		\left((1+|y|)^{-3}+(1+|x|)^{-3}\right)\,\mathrm{d}y\\
 		=:&I_1+I_2.
 	\end{aligned}
 	\]
 	
 	On one hand,
 	\begin{equation}\label{eq-I2-far}
 		\begin{aligned}
 	 I_2
 		\leq&C2^\gamma |x|^{-3-\gamma}2^{4j}
 		\int_{|x-y|>|x|/2}|x-y|^{\gamma}
 		\left|\check\varphi\bigl(2^j(x-y)\bigr)\right|
 	\,\mathrm{d}y\\\leq& C 2^{-j\gamma}|x|^{-3-\gamma}.
 		\end{aligned}
 	\end{equation}
 		Next, we estimate $I_1$. To this end, we split its domain of integration into
 	\[
 	E_1
 	=
 	\left\{
 	|x-y|>\frac{|x|}{2},\quad
 	|y|\leq\frac{|x|}{2}
 	\right\}\quad \text{and}\quad  	E_2
 	=
 	\left\{
 	|x-y|>\frac{|x|}{2},\quad
 	|y|>\frac{|x|}{2}
 	\right\}.
 	\]
 	Accordingly, we write
 	\[
 	I_1=I_{1,1}+I_{1,2}.
 	\]
 	On $E_1$, we have
 	\[
 	|x-y|
 	\geq |x|-|y|
 	\geq\frac{|x|}{2}.
 	\]
 	For any $N>3+\gamma$, the Schwartz
 	decay of $\varphi$ gives
 	\[
 	|\check\varphi(z)|
 	\leq C_N(1+|z|)^{-N}.
 	\]
 	Hence
 	\[
 	1+2^j|x-y|
 	\geq \frac{2^j|x|}{2},
 	\]
 	and therefore
 	\[
 	\begin{aligned}
 		I_{1,1}
 		\leq
 		C2^{4j}(2^j|x|)^{-N}
 		\int_{|y|\leq |x|/2}
 		(1+|y|)^{-3}\,\mathrm{d}y
 			\leq
 		C2^{4j}(2^j|x|)^{-N}|x|
 		=
 		C2^{j(4-N)}|x|^{1-N}.
 	\end{aligned}
 	\]
 	Again,
 	\[
 	2^{j(4-N)}|x|^{1-N}
 	=
 	2^{-j\gamma}|x|^{-3-\gamma}
 	(2^j|x|)^{4+\gamma-N}.
 	\]
 	Since $2^j|x|>1$,  we choose $N>4+\gamma$ to derive the estimate
 	\begin{equation}\label{eq-I11-far}
 		I_{1,1}
 		\leq
 		C2^{-j\gamma}|x|^{-3-\gamma}.
 	\end{equation}
 	On $E_2$, we have
 	$
 	|y|>\frac{|x|}{2},
 	$
 	and hence
 	\[
 	(1+|y|)^{-3}
 	\leq C(1+|x|)^{-3}.
 	\]
 	Therefore,
 	\[
 	\begin{aligned}
 		I_{1,2}
 		&\leq
 		C2^{4j}(1+|x|)^{-3}
 		\int_{|x-y|>|x|/2}
 		(1+2^j|x-y|)^{-N}\,\mathrm{d}y.
 	\end{aligned}
 	\]
 	Using again the change of variables
 	$
 	z=2^j(x-y),
 	$
 	we readily obtain
 	\[
 	\begin{aligned}
 		I_{1,2}
 		\leq&
 		C(1+|x|)^{-3}
 		\int_{|z|>2^{j-1}|x|}
 		(1+|z|)^{-N}\,\mathrm{d}z
 		\\
 		\leq&
 		C(1+|x|)^{-3}
 		(2^j|x|)^{4-N}\\
 			\leq&
 		C2^{-j\gamma}|x|^{-3-\gamma}
 		(2^j|x|)^{4+\gamma-N}
 		\leq
 		C2^{-j\gamma}|x|^{-3-\gamma}.
 	\end{aligned}
 	\]
 	Thus
 	\begin{equation}\label{eq-I12-far}
 		I_{1,2}
 		\leq
 		C2^{-j\gamma}|x|^{-3-\gamma}.
 	\end{equation}
 		Collecting \eqref{eq-I11-far} and \eqref{eq-I12-far}, we obtain
 	\begin{equation}\label{eq-I1-far}
 		I_1
 		\leq
 		C2^{-j\gamma}|x|^{-3-\gamma}.
 	\end{equation}
 	Finally, combining \eqref{eq-I2-far} and \eqref{eq-I1-far}, we conclude that
 	\begin{equation*}\label{eq-far-field}
 		|\dot\Delta_jF(x)|_{\mathrm{far}}
 		\leq
 		C2^{-j\gamma}|x|^{-3-\gamma},
 		\qquad |x|>2^{-j}.
 	\end{equation*}
 	Consequently,
 	\[
 	\sup_{|x|>2^{-j}}
 	|x|^{3+\gamma}
 	|\dot\Delta_jF(x)|_{\mathrm{far}}
 	\leq
 	C2^{-j\gamma}.
 	\]
 	This together with \eqref{eq-near-field} leads to 
 	\[
 	|\dot\Delta_jF(x)|
 	\leq
 	C2^{-j\gamma}|x|^{-3-\gamma},
 	\qquad |x|>2^{-j}.
 	\]
 	Thus
 	\begin{equation}\label{eq-case2}
 		\sup_{|x|>2^{-j}}
 		|x|^{3+\gamma}
 		|\dot\Delta_jF(x)|
 		\leq C2^{-j\gamma}.
 	\end{equation}
 	Finally, combining \eqref{eq-case1} and \eqref{eq-case2}, we conclude that
 	\[
 	\sup_{x\in\mathbb R^4}
 	|x|^{3+\gamma}
 	\left|
 	\dot\Delta_j
 	\bigl(U_0\cdot\nabla U_0\bigr)(x)
 	\right|
 	\leq C2^{-j\gamma},
 	\]
 	for every $j\in\mathbb Z$. This proves the proposition.
 \end{proof}
 The sharp decay estimates in H\"older spaces for the integral involving the forcing term $U_0\cdot\nabla U_0$, established in Proposition  \ref{prop-Cr-weighted-regularized}, enable us to derive sharp decay estimates in H\"older spaces for $V$. In particular, these sharp H\"older decay estimates allow us to eliminate the logarithmic loss present in Proposition \ref{prop-decy-loss-log}.

  \begin{proposition}\label{prop-decy-sharp}
  	Let $\gamma \in (0,1)$ and $\sigma\in C^{1,\gamma}(\mathbb{S}^3)$, and let the pair $(V,P)\in H^1(\mathbb{R}^4)\times L^2(\mathbb{R}^4)$ be weak solutions to problem \eqref{PL}. Then,   the couple $(V,P)$ satisfies 
 	\begin{itemize}
 		\item[(i)]  for every $j\in\mathbb{Z},$	\begin{equation}\label{eq-Cr-wirghted-est-2}
 			\sup_{x\in\mathbb R^4}
 			|x|^{3+\gamma}
 			\left|
 			\dot\Delta_j
 			V(x)
 			\right|
 			\leq
 			C\,2^{-j\gamma};
 		\end{equation}
 		\item [(ii)]	\[ |V(x)|\leq C|x|^{-3}.\]
 	\end{itemize} 
 
 \end{proposition}
 \begin{proof}
 	Recall from \eqref{eq-Integral-0901} that 
 	 	\begin{equation}\label{eq-Integral-0901-26}
 			V(x)=
 		\int_0^1e^{(1-s)\Delta}s^{-\frac32}\mathscr{P}(\operatorname{div}\mathbb{A}+\operatorname{div}\mathbb{A}_0)\Big(\frac{x}{\sqrt{s}}\Big)\,\mathrm{d}s,
 	\end{equation}
 	where $\mathbb{A}=V\otimes U_0+U_0\otimes V+V\times V$ and  $\mathbb{A}_0=U_0\otimes U_0$.
 	
 	For the integral involving $V\times V$, we write
 	\begin{equation*}
 		\begin{aligned}
V_1:= &	\int_0^1e^{(1-s)\Delta}s^{-\frac32}\mathscr{P}\operatorname{div}(V\otimes V)\Big(\frac{x}{\sqrt{s}}\Big)\,\mathrm{d}s\\
 =& \int_0^1   \int_{\mathbb{R}^4} \nabla \mathscr{P}  \left(\Phi \big(x-y,1-s\big)\right) \frac{1}{s} \cdot (V\otimes V)\left(\frac{y}{\sqrt{s}}\right) \mathrm{d} y \mathrm{d} s\\
 =&\int_0^1  \frac{1}{\sqrt{1-s}} \int_{\mathbb{R}^4} \mathscr{P}\left(\nabla  \Phi \right) \big(x-y,1-s\big)\frac{1}{s} \cdot (V\otimes V)\left(\frac{y}{\sqrt{s}}\right) \mathrm{d} y \mathrm{d} s.
 	\end{aligned}
\end{equation*}
Furthermore, we have
	\begin{equation*}
		\begin{aligned}
			 \dot{\Delta}_jV_1 =&C2^{-j\gamma}\int_0^1  {(1-s)^{-\frac{1+\gamma}{2}}} \int_{\mathbb{R}^4}  \dot{\Delta}_j\mathscr{P}\left(\Lambda^\gamma\nabla  \Phi \right) \big(x-y,1-s\big)\frac{1}{s} \cdot (V\otimes V)\left(\frac{y}{\sqrt{s}}\right) \mathrm{d} y \mathrm{d} s\\
			=&C2^{-j\gamma}\int_0^1  {(1-s)^{-\frac{1+\gamma}{2}}} \dot{\Delta}_j\mathscr{P}\int_{\mathbb{R}^4} (1-s)^{-4} K_1\left(\frac{x-y}{1-s}\right)\frac{1}{s} \cdot (V\otimes V)\left(\frac{y}{\sqrt{s}}\right) \mathrm{d} y \mathrm{d} s,
		\end{aligned}
 	\end{equation*}
 	where $K_1=\Lambda^\gamma\nabla  \Phi_1$.

 	We adopt the standard kernel representation of the Littlewood-Paley projection operator by setting \[\dot{\Delta}_j\mathscr{P}f=2^{-4j}K_2(2^{-j}\cdot) \ast f(x),\]where the Fourier symbol of the kernel \(K_2\) is given by
 	\[\hat{K}_2=\sum_{k,\ell=1}^{4}\left(I_d-\frac{\xi_k\xi_\ell}{|\xi|^2}\right)\varphi(\xi)\in\mathscr{S}(\mathbb{R}^4).\]
By using  the estimate \eqref{eq-prop-main-lambda}  and  the fact that \(K_1,K_2\) are Schwartz functions, we further deduce that for each $\gamma\in(0,1)$,
 	\begin{equation}\label{eq-holder-decay-1}
 		\begin{aligned}
 			&2^{j\gamma}	|x|^{3+\gamma}
 			\left|
 			\dot\Delta_j
 			V_1(x)
 			\right|\\
 			\leq &C2^{-j\gamma}\int_0^1  {(1-s)^{-\frac{1+\gamma}{2}}} \sup_{x\in\mathbb{R}^4}|x|^{3+\gamma}\left|\int_{\mathbb{R}^4} (1-s)^{-4} K_1\left(\frac{x-y}{1-s}\right)\frac{1}{s} \cdot (V\otimes V)\left(\frac{y}{\sqrt{s}}\right) \mathrm{d} y\right| \mathrm{d} s\\
 			\leq&
 			C\sup_{x\in\mathbb{R}^4}|x|^{3+\gamma}|V\otimes V|(x)\int_0^1{(1-s)^{-\frac{1+\gamma}{2}}} s^{\frac{1+\gamma}{2}}\mathrm{d} s
 				\leq
 			C\left(\sup_{x\in\mathbb{R}^4}|x|^{\frac{3+\gamma}{2}}|V(x)|\right)^2.
 		\end{aligned}
 	\end{equation}
 	Similarly,  for each $\gamma\in(0,1)$, we have  
 	\begin{equation}\label{eq-holder-decay-2}
 		\begin{aligned}
 			2^{j\gamma}	|x|^{3+\gamma}
 			\left|
 			\dot\Delta_j
 			V_2(x)
 			\right|
 			\leq&C\sup_{x\in\mathbb{R}^4}|x|^{3+\gamma}|V\otimes U_0+U_0\otimes V|(x)\\
 			\leq&	C\sup_{x\in\mathbb{R}^4}|x|^{2+\gamma}|V(x)|\sup_{x\in\mathbb{R}^4}|x||U_0(x)|,
 		\end{aligned}
 	\end{equation}
 where
 \[V_2(x)=\int_0^1e^{(1-s)\Delta}s^{-\frac32}\mathscr{P}\operatorname{div}(V\otimes U_0+U_0\otimes V)\Big(\frac{x}{\sqrt{s}}\Big)\,\mathrm{d}s.\]
 	Collecting  estimate  \eqref{eq-holder-decay-1} and estimate  \eqref{eq-holder-decay-2}, and then using Proposition \ref{prop-H2} and Proposition \ref{prop-decy-loss-log}, we obtain 
 	\begin{equation*}
 		\begin{aligned}
 			&	\sup_{x\in\mathbb{R}^4}\left|
 			\int_0^1e^{(1-s)\Delta}s^{-\frac32}\mathscr{P}(\operatorname{div}\mathbb{A})\Big(\frac{x}{\sqrt{s}}\Big)\,\mathrm{d}s\right|\\
 			\leq&	C\left(\sup_{x\in\mathbb{R}^4}|x|^{\frac{3+\gamma}{2}}|V(x)|\right)^2+C\sup_{x\in\mathbb{R}^4}|x|^{2+\gamma}|V(x)|\sup_{x\in\mathbb{R}^4}|x||U_0(x)|\\
 		\leq &C(\sigma)<\infty.
 		\end{aligned}
 	\end{equation*}
 	This together with estimate \eqref{eq-Cr-weighted-est-2026} in Proposition \ref{prop-Cr-weighted-regularized} yields the first desired estimate  \eqref{eq-Cr-wirghted-est-2}.
 	
 	Next, we turnn to show (ii). According to the proof of Proposition \ref{prop-decy-loss-log}, it suffices to verify that
 	\begin{equation}\label{eq-diff-0902-1}
 		\left| II\right|	\leq	C|x|^{-3},
 	\end{equation}
 	where 
$$
 			II= \sum_{N_0\leq j<0} \int_0^1\tilde{\dot{\Delta}}_j\Phi_{1-s} \ast P_{0,j}(x,s) \,\mathrm{d}s.
 	$$
 	Let us recall $P_0$ satisfies 
 	\[-\Delta P_0=\operatorname{div}\operatorname{div}\big(U_0\otimes U_0).\]
 and 
 	$$P_{0,j}(x,s)= \dot{\Delta}_j \left(s^{-\frac32} \nabla P_0\Big(\frac{x}{\sqrt{s}}\Big)\right)=2^{-4j}K_2(2^{-j}\cdot) \ast \left(s^{-\frac32} \operatorname{div}(U_0\otimes U_0) \Big(\frac{x}{\sqrt{s}}\Big)\right).$$
 	By using  \eqref{eq-prop-main-lambda}  and \eqref{eq-Cr-weighted-est-2026} in Proposition \ref{prop-Cr-weighted-regularized}, we obtian  that for each $\gamma\in(0,1)$,
 	\begin{equation*}
 		\begin{aligned}
 			|II|	\leq &C|x|^{-3-\gamma}\sum_{N_0\leq j<0}2^{-j\gamma} \int_0^1\sup_{x\in\mathbb{R}^4}|x|^{3+\gamma}| P_{0,j}(x,s)|(x)\,\mathrm{d}s\\
 			\leq&C|x|^{-3-\gamma}\sup_{x\in\mathbb{R}^4}|x|^{3+\gamma}|U_0\cdot\nabla U_0|\sum_{N_0\leq j<0} 2^{-j\gamma}\\
 			\leq&	C|x|^{-3-\gamma}\log (1+|x|),
 		\end{aligned}
 	\end{equation*}
 	 in the last line, we have used 	 $N_0=[-\log_2 (1+|x|)]$.  
 	 
 	 We thus finish the proof of the proposition.
 \end{proof}

\end{document}